%% file: random-model-apr.tex
\documentclass[a4paper,12pt]{article}

\usepackage[utf8x]{inputenc}
\usepackage{amsmath}
\usepackage{amsfonts}
\usepackage{latexsym}
\usepackage{amsthm}
\usepackage{amssymb,amscd}
\usepackage{xargs}
\usepackage{graphicx}
\usepackage{color}
\usepackage{enumitem}

\newtheorem{thm}{Theorem}[section]
\newtheorem{conj}{Conjecture}[section]
\newtheorem{lem}[thm]{Lemma}
\newtheorem{definition}[thm]{Definition}
\newtheorem{prop}[thm]{Proposition}
\theoremstyle{definition}
\newtheorem{cor}[thm]{Corollary}
\newtheorem{rmk}[thm]{Remark}
\newcommand{\be}{\begin{eqnarray}}
\newcommand{\ee}{\end{eqnarray}}
\newcommand{\beal}{\begin{aligned}}
\newcommand{\enal}{\end{aligned}}

\newcommand{\eps}{\varepsilon}
\newcommand{\tet}{\varphi}
\newcommand{\al}{\alpha}
\newcommand{\gm}{\gamma}
\newcommand{\sg}{\sigma}
\newcommand{\lb}{\lambda}
\newcommand{\E}{\mathbb{E}}
\newcommand{\Ev}{\mathbb{E}v}

\newcommand{\T}{\mathbb{T}}
\newcommand{\R}{\mathbb{R}}

\newcommand{\A}{\mathbb{A}}
\newcommand{\Z}{\mathbb{Z}}
\newcommand{\Lb}{\Lambda}
\newcommand{\om}{\omega}
\newcommand{\dt}{\delta}
\newcommand{\bt}{\beta}
\newcommand{\SM}{\mathcal{SM}}
\newcommand{\DD}{\mathcal{D}}
\newcommand{\cS}{\mathcal{S}}
\newcommand{\Eu}{\mathbb{E}u}
\newcommand{\cO}{\mathcal{O}}
\newcommand{\cM}{\mathcal{M}}
\newcommand{\cH}{\mathcal{H}}
\newcommand{\cN}{\mathcal{N}}
\newcommand{\cF}{\mathcal{F}}

\newcommand{\wh}{\widehat }

\title{Normally Hyperbolic Invariant Laminations and 
diffusive behaviour for the generalized Arnold example
away from resonances}
\author{
V. Kaloshin\footnote{University of Maryland at College Park,
vadim.kaloshin@gmail.com},\ \ \ 
J. Zhang\footnote{University of Toronto,
jianlu.zhang@utoronto.ca},\ \ \ 
K. Zhang\footnote{University of Toronto,
kzhang@math.toronto.edu} } 
\begin{document}
\maketitle

\begin{abstract} In this paper we study existence of Normally 
Hyperbolic Invariant Laminations (NHIL) for a nearly integrable 
system given by the product of the pendulum and the rotator perturbed 
with a small coupling between the two. This example was introduced 
by Arnold \cite{Arn}. Using a {\it separatrix map}, introduced in a low 
dimensional case by Zaslavskii-Filonenko \cite{ZF68} and studied 
in a multidimensional case by Treschev and Piftankin \cite{Pif,PT,T0,T1}, 
for an open class of trigonometric perturbations we prove that NHIL do exist. 
Moreover, using a second order expansion for the separatrix map from 
\cite{GKZ}, we prove that the system restricted to this NHIL  is a skew 
product of nearly integrable cylinder maps. Application of the results 
from \cite{CK} about random iteration of such skew products show that in the proper $\eps$-dependent time scale 
the push forward of a Bernoulli measure supported on this NHIL weakly converges to an Ito diffusion process on the line
as $\eps$ tends to zero.
\end{abstract}

\tableofcontents

\section{The main result}

Consider the following nearly integrable Hamiltonian system:
\be \label{eq:Hamiltonian}
\beal 
H_\eps (p,q,I,\varphi,t)=H_0(p,q,I)+\eps H_1(p,q,I,\varphi,t)
:=
\qquad \qquad  \qquad \qquad 
\\
=\underbrace{\dfrac{\ \ I^2\ \ }{2}}_{rotor}+
\underbrace{\dfrac{p^2}{2}+(\cos q-1)}_{pendulum}+
\eps H_1 (p,q,I,\varphi,t),
\enal
\ee
where $q,\varphi,t\in \T$ are angles, $p,I\in \R$ (see Fig. \ref{fig:rotor-pendulum}).
In the case $H_1=(\cos q-1)(\cos \varphi+\cos t)$ 
this example was proposed by Arnold \cite{Arn}.

\begin{figure}[h]
  \begin{center}
  \includegraphics[width=12.8cm]{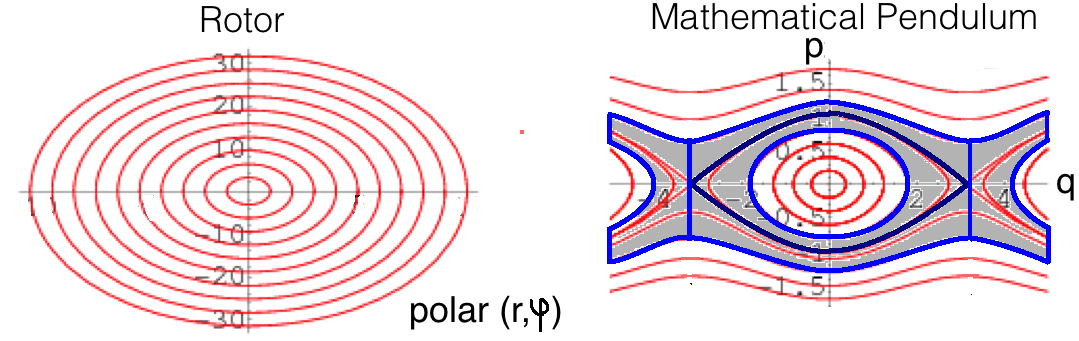}
  \end{center}
  \caption{The rotor times the pendulum }
  \label{fig:rotor-pendulum}
\end{figure}

For $\eps=0$ we have a direct product of 
the rotor $\{\dot \theta=I,\ \dot I=0\}$ and 
the pendulum $\{\dot q=p,\ \dot q=\sin q\}$. We shall study 
dynamics of this systems when the $(p,q)$-component 
is near the separatrices $\dfrac{p^2}{2}+(\cos q-1)=0$.  
Perturbations of systems, given by the product of the rotor 
and an integrable system with a separatrix  loop, 
are called {\it apiori unstable}. Since 
they were introduced by Arnold \cite{Arn}, they recieved 
a lot of attention both in mathematics, astronomy, and 
physics community, see e.g.
 \cite{Ber08,Be,CY1,CY2,Ch,CV,Ci,DH,DLS,DLS13,GL,La,LL,MS,R,T2,T3}. 
It also inspired a variety of examples with instabilities, see e.g. 
\cite{BBB,BB,BB1,BK,DKRS,FGKR,GK,KL1,KL2,KLS,KS,LMS,MS4,
Ma1,Ma2,Ma3,Moe}.

Numerical experiments and heuristic arguments proposed 
by Chirikov and his followers indicate that if we choose many 
initial conditions so that the $(p,q)$-component is close to 
$(p,q)=0$ and integrate solutions over 
$\sim \eps^{-2}\ln 1/\eps$-time, the outcome is that 
the $r$-displacement behaives stochastically, where 
the randomness comes from initial conditions. This is 
the reason Chirikov called this phenomenon  {\it Arnold diffusion}. 

\subsection{Random fluctuations of eccentricity in Kirkwood
gaps in the asteroid belt}
A similar diffusive behavior was observed numerically in many 
other nearly integrable problems. To give another illustrative 
example consider motion of asteroids in the asteroid belt.
The asteroid belt is located between orbits of Mars and 
Jupiter and has around one million  asteroids of diameter 
of at least one kilometer. When astromoters build 
a histogram based on orbital perioid of asteroids
there are well known gaps called {\it Kirkwood gaps.}
These gaps occur when ratio of Jupiter and of an asteroid
is a rational with small denominator: $3:1,5:2,7:3$
(see Fig. \ref{fig:kirkwood}).  
This correspond to so called {\it mean motion resonances
for the three body problem}. Wisdom \cite{Wi} made 
a numerical analysis of dynamics at mean motion resonance 
and observed {\it random jumps of eccentricity} of asteroids
for $3:1$ resonances. Later similar behavior was observed 
for $5:2$ resonance. For other resonances, following
the mechanism from \cite{FGKR}, one could expect 
that eccentricity has random fluctuations and as they 
accumulate eccentricity reaches a certain critical value an orbit 
of asteroid starts to cross the orbit of Mars. This eventually leads 
either to a collision with Mars, or capture by Mars, or a close 
encounter (see also \cite{Mo1}). The latter changes the orbit 
so drastically that almost certainly it disappears from the 
asteroid belt. In \cite{FGKR} in the $3:1$ Kirkwood gap
and small Jupiter's eccenricity we prove existence of certain 
orbits whose eccentricity change by $0.32$ for the restricted 
planar three body problem. 

 \begin{figure}[h]
  \begin{center}
  \includegraphics[width=10.51cm]{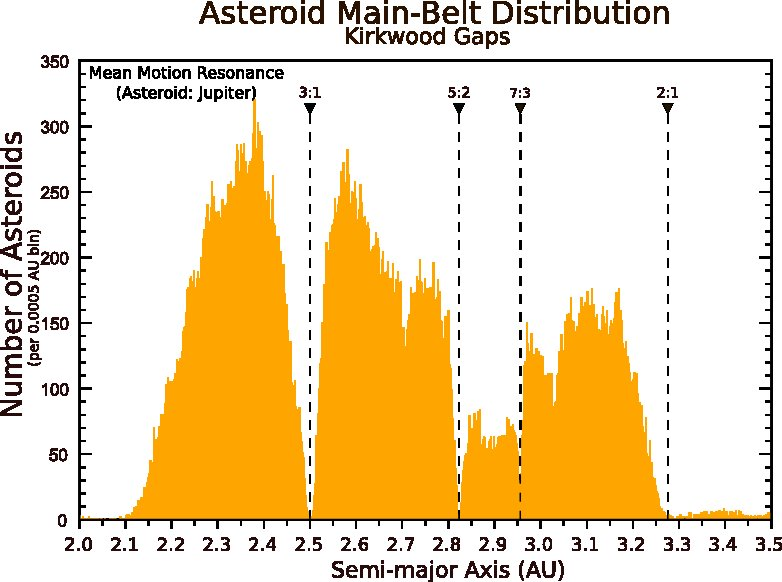}
  \end{center}
  \caption{The distribution of  Asteroids in the asteroid belt
   and Kirkwood gaps}
  \label{fig:kirkwood}
\end{figure}

\subsection{Diffusion processes and infinitesimal generators}
In order to formalize the statement about diffusive behavior we need 
to recall some basic probabilistic notions. A random process 
$\{W_t, t\ge 0\}$ called {\it the Wiener process} or {\it a Browninan motion}
if the following four conditions hold:

$B_0=0$, $B_t$ is almost surely continuous, 
$B_t$ has independent increments, 
$B_t-B_s\sim \mathcal N(0,t-s)$ for any $0\le s\le t$,
where $\mathcal N(\mu,\sigma^2)$ denotes the normal distribution 
with expected value $\mu$ and variance $\sigma^2$. 

The condition that it has independent increments means that if 
$0 \leq s_1 \leq t_1 \leq s_2 \leq t_2$, then $B_{t_1}-B_{s_1}$ 
and $B_{t_2}-B_{s_2}$ are independent random variables.

A Brownian motion is a properly chosen limit of the standard 
random walk. A generalization of a Brownian motion is 
{\it a diffusion process} or {\it an Ito diffusion}. To define it 
let $(\Omega,\Sigma,P)$ be a probability space. Let 
$X:[0,+\infty) \times \Omega \to \R$. It is called an Ito diffusion if 
it satisfies {\it a stochastic differential equation} of the form
\be \label{eq:diffusion}
\mathrm{d} X_{t} = b(X_{t}) \, \mathrm{d} t + 
\sigma (X_{t}) \, \mathrm{d} B_{t},
\ee
where B is an Brownian motion and $b : \R \to \R$ and 
$\sigma : \R \to \R$ are the drift and the variance respectively. 
For a point $x \in \R$, let $\mathbf{P}^x$ denote the law of $X$ 
given initial data $X_0 = x$, and let $\mathbf{E}^x$ denote 
expectation with respect to $\mathbf{P}^x$.

The {\it infinitesimal generator} of $X$ is the operator $A$, which 
is defined to act  on suitable functions $f :\R\to \R$ by
\[
A f (x) = \lim_{t \downarrow 0} \dfrac{\mathbf{E}^{x} [f(X_{t})] - f(x)}{t}.
\]
The set of all functions $f$ for which this limit exists at a point $x$ 
is denoted $D_A(x)$, while $D_A$ denotes the set of all $f$'s for which 
the limit exists for all $x\in \R$. One can show that any 
compactly-supported $C^2$  function $f$ lies in $D_A$ and 
that
\[
Af(x)=b(x) \dfrac{\partial f}{\partial x}+ \dfrac 12 \sigma(x)
\dfrac{\partial^2 f}{\partial x \partial x}.
\]
In particular, we can characterize a diffusion process by 
the drift $b(x)$ and the variance $\sigma(x)$. Thus, we can 
identify an Ito diffusion if we know the drift $b(x)$ and 
the variance $\sigma(x)$. 

\subsection{Conjecture on rotor's stochastic diffusive behavior}

Consider the Hamiltonian $H_\eps$ of the form (\ref{eq:Hamiltonian}).
Let $\A=\R\times \T$ be a $2$-dimensional annulus, and 
$B^2_{\sqrt \eps}(0)$ be the ${\sqrt \eps}$-ball around the origin 
in $\A \ni (p,q)$, and $B_\eps(I^*)$ be an $\eps$-neighborhood of $I^*$ in $\R$. Let $X=(p,q,I,\varphi,t)$ denote a point in the whole 
phase space and by $X_t^\eps$ the time $t$ map of $H_\eps$ 
with $X$ as the initial condition. Pick any $I^*\in \R$. Denote by 
$\mu^\eps(I^*)$ the normalized Lebesgue measure 
supported inside 
$$
D_\eps(I^*):=B^2_{\sqrt \eps}(0) \times B_\eps(I^*) \times \T^2\ni (p,q,I,\varphi,t).
$$
Denote by $\mu^\eps_t$ the image of $\mu^\eps(I^*)$ under 
the time $t$ map of $H_\eps$, by $\Pi_I$ the projection onto 
the $I$-component, by $t_\eps=- \frac{\ln \eps}{\eps^2}$ 
the rescaled time.   
 
\begin{conj} 
Let the initial distribution be the normalized 
Lebesgue measure $\mu^\eps(I^*)$ for some $I^*$. 
Then for a generic perturbation $\eps H_1(\cdot)$
there are smooth functions $b(I)$ and $\sigma(I)>0$,
depending on $H_1$ and $H_0$ only, such that 
for each $s> 0$ and $\,t_\eps=s\, \eps^{-2}\log \frac 1\eps$
the distribution $\Pi_I ( \phi^{t_\eps}_* \mu^\eps)$ 
converges weakly, as $\eps \to 0$,  to the distribution 
of $I_s$, where $I_\bullet$ is the diffusion process with 
the drift $b$ and the variance $\sigma$, starting at $I_0=I^*$. 
\end{conj}

This conjecture can be viewed as formalization of the
discussion in chapter 7 of \cite{Ch}. As a matter of fact
presence of a possible drift in not mentioned there. 
In this paper Chirikov coined the term for this instability 
phenomenon --- {\it Arnold diffusion}.

\begin{rmk} The strong form of this conjecture is to find 
a family of measures $\mu^\eps$ such that for some $c>0$ 
\[
\lim_{\eps \to 0}\ 
\dfrac{\text{Leb }(\text{supp }\mu^\eps)}{
\text{Leb }(B^2_{\sqrt \eps}(0) \times B_\eps(I_0) 
\times \T^2)}>0,
\]
where Leb is the $5$-dimensional Lebesgue measure. 
In other words, the conditional probability to start $\eps$-close 
to the unstable equilibria of the pendulum and action $r_0$
and exhibit stochastic diffusive behavior is uniformly positive. 
\end{rmk}

In \cite{KR} we give numerical evidence 
in favour of this conjecture. Here is the description of 
numerical experiments in \cite{KR}. 
Let $\eps=0.01$ and $T=\eps^{-2}\ln 1/\eps$. 
On Figure \ref{fig:histograms} we present several 
histograms plotting displacement of the $I$-component 
after time $T, \,2T,\, 4T,\,8T$ with 6 different groups of 
initial conditions. Each group has of $10^6$ points. In each 
group we start with a large set of initial conditions close 
to $p=q=0,\ I=I^*$. 

\subsection{Statement of the Main Result}
In this paper we study a simplified versions of $H_\eps$
in (\ref{eq:Hamiltonian}). Namely, we consider 
the following family of perturbations 

\newpage 

\begin{figure}[h]
  \begin{center}
  \includegraphics[width=11.8cm]{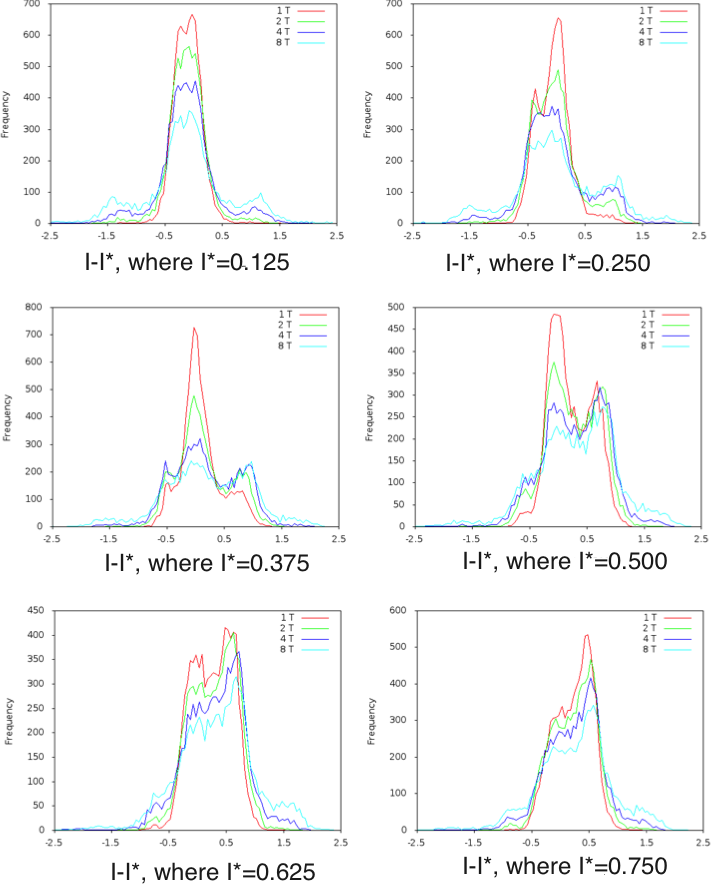}
  \end{center}
  \caption{Histograms of the $I$-dispacement}
  \label{fig:histograms}
\end{figure}

\be \label{eq:our-Hamiltonian}
H_\eps (p,q,I,\phi,t)=
\dfrac{I^2}{2}+
\dfrac{p^2}{2}+(\cos q-1)+
\eps P_N(\exp(i q),\exp(i \varphi),\exp( it)),
\ee
where $P_N(\exp(i q),\exp(i \varphi),\exp( it))$ is a real valued 
trigonometric polynomial, i.e. 
for some $N\ge 2$ and real coefficients $p'_{k_1,k_2,k_3}$ 
and $p''_{k_1,k_2,k_3}$ with $|k_i|\le N, i=1,2,3$
 we have 
\be \label{eq:perturb}
\beal 
P_N(\exp(i q),\exp(i \varphi),\exp( it))=& 
\\ 
\sum_{|k_i|\le N, i=1,2,3} 
p'_{k_1,k_2,k_3} \cos (k_1q+k_2 \varphi+k_3 t)&+
p''_{k_1,k_2,k_3} \sin (k_1q+k_2 \varphi+k_3 t).
\enal 
\ee
In the example proposed by Arnold \cite{Arn} we have 
$P_2=(1-\cos q)(\cos \varphi+\cos t)$. 

Denote by $\R^{m(N)}$ the space of real coefficients 
of $P_N$ and by $\phi^t$ the time $t$ map of 
the Hamiltonian vector field of $H_\eps$. Let 
\[
 \mathcal N_{\bt}(P_N)=\{k\in\Z^{3}: (p'_k,p''_k)\neq 0\}
\]
and 
\[
 \mathcal N_{\bt}^{(2)}(P_N)=
\{k\in\Z^{3}: k=k_1+k_2, k_1,k_2\in \mathcal N_{\bt}(P_N)\}.
\]

Fix $\beta>0$. 
Define a $\beta$-non-resonant domains
\be \label{eq:non-res-domain}
\DD_{\bt}(P_N)=\{I\in \R: \forall k\in \mathcal N_{\bt}^{(2)}(P_N)\ 
\text{we have }\ |k_2I+k_3|\ge \bt \}. 
\ee
Notice that $\DD_{\bt}(P_N)$ contains the subset of $\R$  
with  $\bt$-neighborhoods of all rational numbers 
$p/q$ with $0<|q|\le 2N$ removed. Here $N$ is degree of $P_N$. 
Let $I^* \in \DD_{\bt}(P_N)$ and $X^*=(p,q,I^*,\varphi,t)$. 
Denote {\small
\be 
\widetilde \phi^{t} X^*= 
\begin{cases}
\phi^{t} X^* \qquad \qquad \ &  \text{if}\ 
\Pi_I(\phi^{s} X^*) \in \DD^{(2)}_{\bt}(P_N) \ \text{ for all } 
\  0< s \le t. 
\\
\widetilde \phi^{t} X^*=\phi^{t^*} X^* \ & \text{if}\ 
\Pi_I(\phi^{s} X^*) \in \DD^{(2)}_{\bt}(P_N) \text{ for } 
0< s < t^* \\ 
 &\&  \ \Pi_I(\phi^{t^*} X^*) \in \partial \DD^{(2)}_{\bt}(P_N). 
\end{cases}
\ee}

\begin{thm} \label{thm:main-thm}
For the Arnold's example 
(\ref{eq:our-Hamiltonian}--\ref{eq:perturb}) there is  an open set 
of trigonometric polynomials $P_N$
and smooth functions $b(I)$ and $\sigma(I)$, depending on 
$P_N$ only, such that:

for each $\beta,s>0$ and each $I^*\in \DD^{(2)}_{2\bt}(P)$ 
there exists a probability measure  $\mu^\eps$, supported in 
$D_\eps(I^*)$, with the property that  for 
$\,t_\eps=s\, \eps^{-2}\log \frac 1\eps$ the distribution 
$\Pi_I (\widetilde \phi^{t_\eps}_* \mu^\eps)$ 
converges weakly, as $\eps \to 0$, to the distribution of 
$I_{\min\{s,\tau\}}$, where $I_\bullet$ is the diffusion process 
with the drift $b(I)$ and the variance $\sigma(I)$, starting at 
$I_0=I^*$, and $\tau$ is the first time that the process 
$I_\bullet$ reaches the boundary $\partial \DD^{(2)}_{\bt}(P)$. 
\end{thm}

The proof of this Theorem consists of three steps: 
\begin{enumerate}
\item ({\bf A separatrix map}) Write a separatrix map $\SM_\eps$ 
for the generalized Arnold example 
(\ref{eq:our-Hamiltonian}--\ref{eq:perturb}). 
First, the map $\SM_\eps$ is defined for 
general an apriori unstable systems in 
section \ref{sec:separ-map} and computed for this example in 
Corollary \ref{separatrix-for-arnold-example}. One can view 
the separatrix map as an induced return map of the time 
one map $\phi^1$ of $H_\eps$ into a carefully chosen 
fundamental domain (see Fig. \ref{fig:separatrix-fund-region}).

\begin{figure}
\qquad \qquad 
\includegraphics[width=13.1cm]{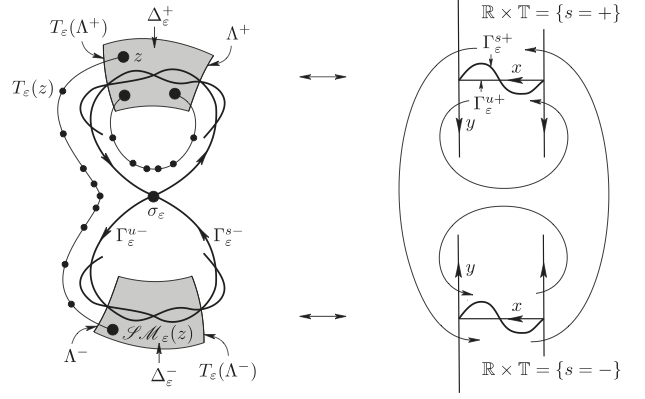}
\qquad 
\caption{The fundamental domain $\Delta_\eps^\pm$}
\label{fig:separatrix-fund-region}
\end{figure} 

\item ({\bf Isolating block and Normally Hyperbolic 
Laminations (NHIL)}) In Appendix \ref{sec:suff-nhil} using 
Conley's idea of isolating block (see e.g. \cite{BKZ,McG}) 
we derive a sufficient condition for existence 
of a NHIL and in section \ref{sec:construction-nhil}, after 
a careful analysis of the separatrix map and its linearization, 
we verify this sufficient condition and construct a NHIL 
${\bf \Lb}_\eps$. Leaves of this NHIL ${\bf \Lb}_\eps$
are $2$-dimensional cylinders. 

\item ({\bf A skew product of cylinder maps})
In section \ref{sec:derivation-random-model},
using results from \cite{GKZ}, we find coordinates
such that the restricted system 
$\SM_\eps|_{{\bf \Lb}_\eps}:{\bf \Lb}_\eps\to {\bf \Lb}_\eps$ 
has the following skew-product form of maps 
of a cylinder $\A=\R\times \T\ni (R,\theta)$
{\small 
\be
\label{eq:skew-shift-model} 
\beal 
R^*  =  R 
+ \eps\log \eps \cdot N^{[\om]_k}_1
\left(\theta,\small {\frac {R}{\log \eps}}\right)
+ \eps^2 \log \eps \cdot N^{[\om]_k}_2
\left(\theta,\small{\frac{R}{\log \eps}}\right)  
+ \mathcal O_\om (\eps^3 )|\log \eps|
\\
\theta^*= \theta + R +\mathcal O_\om(\eps\log \eps).
\qquad \qquad \qquad \qquad 
\qquad \qquad \qquad \qquad 
\qquad \qquad \ \quad 
\enal
\ee}
where $\om_i=0$ or $1,$ and $\om=(\dots,\om_i,\dots)
\in \{0,1\}^\Z$, $ N^{[\om]_k}_i,\ i=1,2$ are smooth functions, 
depending on only finite terms of $\omega$, i.e. $[\omega]_k=(\omega_{-k},\cdots,\omega_0,\cdots,\omega_k)$ and both remainder terms depend on $\om$. 
See Corollary \ref{cor:random-model}.  
This model fits into the framework of \cite{CK}. 
\end{enumerate}

\subsection{Possible extensions of Theorem \ref{thm:main-thm}.}
\label{sec:poss-extensions}

\begin{itemize}
\item ({\bf Extension to the whole $\R$})
We hope to extend our results to the whole $\R$, i.e. 
to neighborhood of rationals $p/q$ with $|q|\le N$. 
The difficulties are of purely technical nature.
For $I$ in the $\bt$-non-resonant domain $\DD^{(2)}_{\bt}(P)$ 
in \cite{GKZ} we show that the separatrix map $\SM_\eps$ 
has a relatively simple expression 
(see Theorem \ref{thm:SM-higher-order:SR})
In the $\beta$-resonant domain $\R\setminus \DD^{(2)}_{\bt}(P)$
we also compute the separatrix map $\SM_\eps$ 
with high accuracy,  but the corresponding expression 
is more involved (see Theorem 3.4, section 2 \cite{GKZ}). 
However, this leads to a skew product of cylinder maps not 
covered by \cite{CK}. It seems feasible that technique 
developped in \cite{CK} still applies.

\item ({\bf Generic trigonometric perturbations})
Even though it seems plausible, at the moment we are not 
able to construct a NHIL for a generic trigonometric 
perturbations. Our pertubations are close to purely time 
dependent perturbations, namely, 
$H_1(q,\varphi,t)=(\cos q-1)f(t)+ag(q,\varphi,t)$, where 
$f,g$ are trigonometric polynomials, $f(t)$ satisfies some 
nondegeneracy condition and $a$ is sufficiently small 
(see condition (\ref{eq:Melnikov-deriv-perturb})). 

\item ({\bf Generic smooth/analytic perturbations})
At the moment our scheme uses 
trigonometric nature of the perturbations
in a very essential way\footnote{Dependence 
on $q$ can be chosen smooth or analytic}. In this setting 
we can divide the fundamental region $\Delta$ into 
the $\beta$-resonant and the $2\beta$-non-resonant zones
(see definition (\ref{eq:non-res-domain})). In general, 
this definition is combersome. However, in \cite{DH} 
this problem is treated for generic smooth perturbations.

Removing this trigonometricity assumption leads to 
considerable technical difficulties. 
\begin{enumerate}
\item The second order expansion 
of the separatrix map \cite{GKZ} has to be redone. 
\item Derivation of the skew product model of maps of 
the cylinder from section \ref{sec:derivation-random-model} 
has to be worked out in that setting. 
\item For a new skew product one needs to 
adapt the technique from \cite{CK}. 
\end{enumerate}
\end{itemize}

\subsection{Remarks  
on Theorem \ref{thm:main-thm}.}

\begin{itemize}

\item Notice that the Hamiltonian $H_\eps$ in
(\ref{eq:our-Hamiltonian}) has a $3$-dimensional 
normally hyperbolic invariant cylinder, denoted $\Lb_\eps$, 
near the cylinder $\Lb_0:=\R\times \T^2=\{p=q=0\}$ 
(see section \ref{sec:nhil-skew} for definitions). The orbits 
we study always stay close to stable (resp. unstable) 
$W^s(\Lb_\eps)$ (resp. $W^u(\Lb_\eps)$) manifold of 
$\Lb_\eps$. Naturally, the dynamics of each such an orbit 
can be decomposed into ``loops'' starting and ending 
near $\Lb_0$.

\item A measure $\mu^\eps$ can be chosen so that
$\Pi_I (\mu^\eps)$ is the $\delta$-measure at $I^*$. 
The support of supp $\mu^\eps$ belongs to 
a NHIL ${\bf \Lb}_\eps$ constructed in 
section \ref{sec:construction-nhil}.

\item The NHIL ${\bf \Lb}_\eps$ is ``located'' near 
two connected components of intersections of 
stable \& unstable manifolds $W^s(\Lb_\eps)$ and 
$W^u(\Lb_\eps)$ resp. of the NHIC $\Lb_\eps$.

\item Locally ${\bf \Lb}_\eps$ is a prodict of 
a $3$-dimensional cylinder 
$\mathbb A^3=\R \times \T \times \T$ and a Cantor set 
$ \Lb_\eps$. This Cantor set is homeomorphic 
to $\Sigma=\{0,1\}^\Z$. 

\item $\mu^\eps$ can be chosen as a Benoulli measure
on ${\bf \Lb}_\eps\cap\{(I,\varphi,t)=(I^*,\varphi^*,t^*)\}$ 
for some $(I^*,\varphi^*,t^*)$ in the domain of definition,
which is homeomorphic to  $\Sigma$. 


\item Since $\mu^\eps$ is supported on the 
NHIL, Lebesgue measure of its support is zero.

\item Notice that such a lamination is not invariant, it is weakly 
invariant in the following sense: Let ${\bf \Lb}_\eps\cap 
\{I\in \DD_{\beta}(P) \}$. Then if $X\in {\bf \Lb}_\eps$ and 
$I\in \DD_{\beta}(P)$, then 
$\phi^1(X)\in {\bf \Lb}_\eps\cap \{I\in \DD_{\beta/2}(P) \}$. 
Indeed, $\beta$ is independent of $\eps$. 
In other words, the only way orbits can escape from 
${\bf \Lb}_\eps$ is through the top (resp. bottom) boundary 
given by intersections with $\partial \DD_{\beta}(P)$.

\item An open set $U\subset \R^{m(N)}$ of validity of 
this theorem
is stated in terms of an associated  Poincar\'e-Melnikov 
integral (or splitting potentials) $M(I,\varphi,t)$ (see 
section \ref{sec:NHL-diffusion-condition}). See also comments 
in the previous section \ref{sec:poss-extensions} about 
extensions to general trigonometric perturbations. 
 
\end{itemize}

Here is a detailed plan of the proof and of structure of the paper:
\begin{itemize}
\item Computation of a separatrix map $\SM_\eps$:
\begin{itemize}
\item Write a general separatrix map $\SM_\eps$ (section 
\ref{sec:separatrix-map}); 

\item Derive a specific form of $\SM_\eps$ for the generalized 
Arnold example (section \ref{sec:separatrix-map-for-Arnold});

\item The map $\SM_\eps$ involves the splitting potential,
which is computed in section \ref{sec:splitting-potential};

\item Properties of the splitting potential are analysed in section. 
\ref{sec:NHL-diffusion-condition}
\end{itemize}

\item Analysis of the linearization of the separatrix map
and construction of a normally hyperbolic invariant 
lamination (NHIL):

\begin{itemize}
\item We state the main existence theorem of NHILs 
in section \ref{sec:existence-nhil};

\item We start the proof of this Theorem by analyzing
the linearization of the separatrix map $\SM_\eps$ in 
section \ref{sec:SM-linearize};

\item In section \ref{sec:IB-centers} we compute
almost fixed cylinders $\SM_\eps(C_i)\approx C_i, i=0,1$
and almost period two cylinders 
$\SM_\eps(C_{01})\approx C_{10},\ 
\SM_\eps(C_{10})\approx C_{01}$. 
These cylinders serve as centers of the isolating blocks. 

\item In section \ref{sec:cone-cond} 
we construct isolating blocks and present cone fields on them, then proved the [C1-C5] conditions defined in Appendix
\ref{sec:suff-nhil}.
\end{itemize}

\item In section \ref{sec:derivation-random-model}
we derive a skew product of cylinder maps model
(\ref{eq:skew-shift-model}). This consists of two
steps. 

\begin{itemize} 
\item In section \ref{sec:separatrix-map-SR} 
we state a result from \cite{GKZ} about the expansion 
of the separatrix map up to the second order in actions.

\item In section \ref{sec:symplectic-coord} on each of 
cylindric leaves of the NHIL ${\bf \Lb}_\eps$ we introduce 
a concervative coordinates and derive the random 
cylinder map model (\ref{eq:skew-shift-model}).
\end{itemize}

\item In Appendix \ref{sec:suff-nhil}
we state a sufficient condition of existence of a NHIL,
which essentially goes back to Conley (see e.g. \cite{McG});

\item In Appendix \ref{sec:nhil-skew} we define 
normally hyperbolic invariant laminations and skew products.

\item In Appendix \ref{sec:oriol-result} we state the result 
from \cite{CK} about weak convergence to a diffusion process 
for distributions of the vertical component of random 
iterations of cylinder maps (\ref{eq:skew-shift-model}).

\item In Appendix \ref{sec:eapt} we study certain classes
of exact nearly integrable maps of a cylinder. This is used 
in derivation of the random cylinder map model 
(\ref{eq:skew-shift-model}) in 
section \ref{sec:symplectic-coord}.

\end{itemize}

{\bf Acknowledgement} The authors would like to warmly thank 
Marcel Guardia for many useful discussions of analytic
aspects of this work. Comments of Leonid Polterovich on 
symplectic structure of normally hyperbolic laminations led to 
a special normal form from section \ref{sec:symplectic-coord} 
and is a significant step in the proof. Remarks of Leonid Koralov, 
Dmitry Dolgopyat, Anatoly Neishtadt, Amie Wilkinson were useful 
for this project. The authors warmly thank to all these people. 
The first author acknowledges partial support of the NSF grant 
DMS-1402164.

\section{A separatrix map of apriori unstable systems}
\label{sec:separ-map}

 Consider a Hamiltonian system
\[
H_\eps(p,q,I,\varphi,t)=H_0(p,q,I)+\eps H_1(p,q,I,\varphi,t),
\]
where $H_0(p,q,I)=H_0(0,0,I)$ has two separatrix loops.
Denote by $\mathcal D_0$ any bounded region. 
For example, $H_0$ is the harmonic oscillator times the pendulum: 
\[
H_0(p,q,I)=\dfrac{I^2}{2}+\dfrac{p^2}{2}+(\cos q-1),
\]
where $p,I\in \mathbb R$ are actions and $q,\varphi\in \mathbb T$ 
are angles. We can use formula (1.10) in Piftankin-Treschev with 
$n=1$ and no $\eps^2$-term. In order to apply results of this paper
impose the following conditions: 

\begin{itemize}
\item[{\bf [H1]}] The function $H$ is $C^r$--smooth with respect 
to $(I, \varphi, p, q, t)$, where $r\ge 13$.

We consider the  alternative assumption. 

\item[\textbf{[H$1'$]}] The function $H_0$ is $C^r$
for $r\ge 50$ and $H$ is $C^s$-smooth in all arguments for 
$s\ge 6$ and $r\geq 8s+2$.  
\end{itemize}

Notice that regularity of $H_0$ exceeds that one of $H_1$.
The more regular $H_0+\eps H_1$, the better estimates 
of the remainder terms of the separatrix map we have. 
For a $C^1$ analysis of the separatrix map, it would suffice 
$s\geq 5$ and $r\geq 42$, $r\geq 8s+2$.

\begin{itemize}
\item[{\bf [H2]}] For any $r \in \mathcal D_0$ the function 
$H_0(I_0, p, q)$ has a non-degenerate saddle point 
$(p, q) = (p_0, q_0)$. Every point $(p_0, q_0)$ belongs to 
a compact connected component of the set 
$$
\{(p, q) \in \text{Fig}_{\bf 8} : H_0(I_0, p, q) = H_0(I_0, p_0, q_0)\}.
$$
Moreover, $(p_0, q_0)$ is the unique critical point of $H_0(I_0, p, q)$ 
on this component (see Fig. \ref{fig:figure-eight}).
\end{itemize}

\begin{rmk}\label{rmk:saddle}
 Using Prop.1, \cite{T1}, if one  
assumes that the saddle is at a certain point $(p,q)=(p^0,q^0)$ 
which depends smoothly on $I$, then, one can perform a
symplectic change of coordinates so that the critical point 
is at $(p,q)=(0,0)$ for all $I\in\mathcal D$. After such a coordinate 
change $C^r$ in \textbf{H1} is replaced by $C^{r-2}$.   
\end{rmk}

The point $(p_0, q_0) \in$ Fig$_{\bf 8}$ depends smoothly on 
$I_0$ and is a hyperbolic equilibrium point of a system with 
one degree of freedom and with Hamiltonian $H_0(I_0, p, q)$.
The corresponding separatrices are doubled and form 
a curve of figure-eight type. Below we denote the loops 
of the figure-eight by $\widehat \gm^\pm(I_0)$, where 
$\widehat \gm^+(I_0)$ is called {\it the upper loop} and 
$\widehat \gm^-(I_0)$ --- {\it the lower loop.} The loops 
$\widehat \gm^\pm(I_0)$ have a natural orientation 
generated by the flow of the system. The orientation on 
Fig$_{\bf 8}$ is determined by the system of coordinates $p,q$.

Notice that in our case these loops do not depend on $r_0$. To 
satisfy [{\bf H2}] consider the cylinder $\A=\R \times \T \ni (p,q)$ 
and a diffeomorphism from the set 
$|\frac{p^2}{2}+(\cos q-1)|\le 0.1$ to 
the figure-eight. 

\begin{figure}[h]
  \begin{center}
  \includegraphics[width=5.5cm]{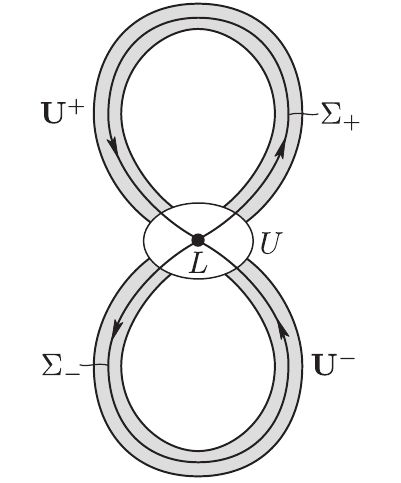}
  \end{center}
  \caption{Separatrices in the form of the figure-eight}
  \label{fig:figure-eight}
\end{figure}

\begin{itemize}
\item[{\bf [H3]}] For any $I_0 \in \mathcal D_0$ the natural 
orientation of $\widehat \gm^\pm(I_0)$ coincides with 
the orientation of the domain Fig$_{\bf 8}$, i.e.the motion along 
the separatrices is counterclockwise (see Fig. \ref{fig:figure-eight}).

\item[{\bf [H4]}] The variables $I$ are separated from 
$p$ and $q$ in the non-perturbed Hamiltonian, i.e. 
$H_0(I, p, q) = F(I, f(p, q)))$.
\end{itemize}
Both [{\bf H3}] and [{\bf H4}] are clearly satisfied
for the generalized example of Arnold 
(\ref{eq:our-Hamiltonian}--\ref{eq:perturb}). 

Now we define the separatrix map from \cite{PT} describing 
the dynamics of the systems satisfying assumptions [{\bf H1-H4}]. 
As an intermediate step it is also convenient to study 
perturbations vanishing on the cylinder $\Lb_0$:  
\be \label{eq:perturb1}
H_1(p,q,I,\varphi,t):=(\cos q-1)P(\exp(i \varphi),\,\exp( it)), 
\ee
where $P$ is a real valued trigonometric polynomial.This is 
a particular case of triginometric polynomials of the form 
(\ref{eq:perturb}). For the classical Arnold example \cite{Arn} 
we have $P= \cos \varphi+\cos t$. 

\subsection{Formulas of the separatrix map of a priori unstable 
systems}\label{sec:separatrix-map}

We would like to apply Theorem 6.1 from Piftankin-Treschev
\cite{PT} presenting almost explicit formulas with a remainder 
for the separatrix map. It uses the Poincar\'e-Melnikov potential 
for the ``outer'' dynamics and the restriction of the perturbation 
to $(p,q)=0$ for the ``inner'' dynamics. The words ``inner'' dynamics 
is used to describe dynamics of the Hamiltonian flow restricted 
to the normally hyperbolic invariant cylinder $\Lb$ and the ``outer'' 
dynamics to describe evolution along invariant manifolds of 
$\Lb$.\footnote{This is analogous to the ``inner'' and ``outer'' 
dynamics from \cite{DLS}. However, the separatrix map contains 
more information then the outer maps from \cite{DLS} 
as it is not constrained to invariant submanifolds. }

Consider the frequency map $\nu(I)=\partial_I H_0(0,0,I)=I$
as the map $\nu:\mathcal D_0\to \R^n$.
It gives the frequency of the torus $\mathcal T(I):=\{(0,0,I)\}$. 
Let $\phi: \mathbb R \to [0,1]$ be a $C^\infty$-smooth function 
such that $\phi(I) = 0$ for $|I| \ge  1$ and $\phi(I) = 1$ for 
$|I| <1/2$. Fix some $\dt \in (0,1/4]$. In (6.1--6.2) Piftankin-Treschev 
chapter 6 \S2 they introduce an auxiliary Hamiltonian 
\be \label{ave-remainder}
\beal 
\overline {\bf H}_1(I,\varphi,t)&=&\sum_{(k,k_0)\in \mathbb Z^2} 
\phi\left( \dfrac{k \varphi + k_0}{\eps^{\dt}}\right)
H^{k,k_0}_1(I) \exp( 2\pi\, i(k\varphi+k_0t)),\\
{\bf H}_1(I,\varphi)\ \ &=&\bar {\bf H}(I,\varphi,0), 
\qquad \qquad \qquad \qquad 
\qquad \qquad \qquad \qquad 
\enal
\ee
where $H^{k,k_0}_1(I) $ are Fourier coefficients of 
$H_1(0,0,I,\varphi,t)$. 
The function $\overline {\bf H}_1$ is the mollified mean of 
$H_1(0,0,I,\varphi,t)$ along the non-perturbed trajectories 
on the tori $\mathcal T(I)$. This procedure is similar to local 
averaging proposed in \cite{BKZ}, Thms 3.1, 3.2. This function 
tends pointwise to the usual average as $\eps \to 0$ 
$$
 \sum_{k I+k_0=0} H^{k,k_0}_1(I) \exp(2\pi i (k \tet+k_0 t) )=
\lim_{T \to \infty} \dfrac{1}{T} 
\int_0^T H_1(0,0,I,\varphi+\nu(I)\,s,t+s)\,ds. 
$$
Since averaged are discontinuous in $I$ 
we prefer to deal with ${\bf H}_1$ and $\overline {\bf H}_1$.
For the generalized Arnold example these functions vanish.

Let $\mathcal D \subset \mathcal D_0$ be an open connected 
domain with compact closure $\overline {\mathcal D}$. 
Let ${\bf K}$ be a compact set in $\R^{n+1}$. In 
the spaces $C^r(\overline \DD \times {\bf K})$ we introduce 
the following norms: for $f \in C^r(\overline \DD \times {\bf K})$ 
let 
\[
\| f(r,z)\|_r^{(b)} = \max_{0\le l'+l''\le r}
b^{l'}\left| \dfrac{\partial^{l'+l'' f}}{ \partial r^{l'}
\partial z_1^{l_1''}\dots
\partial z_m^{l_m''}},
\right|
\]
where $l''=l_1''+\dots+l_m''$. It is assumed that $f$ can take 
values in $\R^s$, where $s$ is an arbitrary positive integer. 
The norms $\|\cdot \|^{b}_r$ are anisotropic, and the variables 
$r$ play a special role in these norms because the additional 
factor $b$ corresponds to the derivatives with respect to $r$. 
Obviously, $\|\cdot \|^{1}_r$ is the usual $C^r$-norm. This norm 
is similar to a skew-symmetric norm introduced in \cite{KZ12}, 
section 7.2. The same definition applies to functions periodic 
in $z$, i.e. $z\in \T^{n+1}$.

For brievity denote 
\be \label{anisotropic-norm}
\|\cdot \|_r^*=\|\cdot \|_r^{(\eps^\dt)}. 
\ee
For functions $f \in C^r(\overline  {\mathcal D}\times \bf K)$ and 
$g\in C^0(\overline {\mathcal D}\times \bf K)$ we say that
$$
f = O^{(b)}_0(g)\ \textup{  if }\ \|f \|_r^{(b)}\le C\,g^{k},
$$
where $C$ does not depend on $b$. For brevity we write
\be \label{eq:skewsym-norm}
\|\cdot \|_r^*=\|\cdot \|_r^{(\eps^\dt)},  \ \ \ 
O^{(b)} =O^{(b)}_1, \ \ \  O^*_k=O^{(\eps^\dt)}_k
\ee 

Notice that for the generalized Arnold example  we have 
$n=1, \ E(r)=\frac{r^2}{2}$. 


\begin{thm} \label{thm:separatrix-map}
For the Hamiltonian $H_\eps$ there
are smooth functions
\[
\lb, \kappa^\pm : \bar \DD \to \R,
\  M^{\pm} : \bar \DD \times \T^2 \to \R,
\]
a constant $c>0$ and 
coordinates $(\eta,\xi,h,\tau)$ such that the following conditions hold:
\begin{itemize}
\item $\omega = d\eta \wedge d\xi+ dh \wedge d\tau$;
\item 
$\eta = I+O^*(\eps^{3/4}, H_0 -E(r)),\xi + \nu(I)\, \tau = q+f$, 
where the function $f$ depends only on $(p,I,\varphi,\eps)$ and is 
such that $f(I,0,0,0) = 0, \ h = H_0 +O^*(\eps^{3/4},H_0 -E(I))$, 
and $H_0 = H_0(p,q,I).$ Let 
\be \label{dist-to-separatrix}
w_0&:=&h^+ - E(\eta^+)-\eps {\bf H}(\eta^+,\xi+\nu(\eta^+)\tau,t), 
\ee
where $w_0$ measures distance to the invariant manifolds.
\item For any $(\eta^+,\xi,h^+,\tau)$ such that
\be \label{region-separatrix-map}
\beal 
c^{-1} \,\eps^{5/4}\, |\log \eps| < |w_0| < c \,\eps^{7/8},\ 
\footnote
{This condition puts  a range of admissible distances 
from the invariant manifolds. In particular, it gives 
a range for the transition time $t^+$. } \qquad \qquad 
\\
|\tau| < c^{-1},\qquad  c < |w_0|\exp(\lb(\eta^+)t^+) < c^{-1},
\enal
\ee
the map $g^{t^+}_\eps T^{t^+}_\eps = \mathcal{SM}_\eps$ 
at time $t^+$ is defined as follows: 
\[
\mathcal{SM}_\eps (\eta,\xi,h,\tau,s,t^+) = 
(\eta^+,\xi^+,h^+,\tau^+,s^+),
\] 
where
\be \label{eq:separatrix-map}
\beal 
\eta^+&=&\eta - \eps M^\sg_\xi (\eta^+,\xi,\tau)- \dfrac{\partial_\xi w_0}{\lambda}
\log \left| \dfrac{\kappa^\sg w_0}{\lambda}\right|+{\bf O}_2 \ 
\\
\xi^+&=&\xi + \eps M^\sg_\eta (\eta^+,\xi,\tau)+ \dfrac{\partial_{\eta^+}w_0}{\lambda}
\log \left| \dfrac{\kappa^\sg w_0}{\lambda}\right|+{\bf O}_1  
\\
h^+&=&h - \eps M^\sg_\tau (\eta^+,\xi,\tau)- 
\dfrac{\partial_{\tau}w_0}{\lambda}
\log \left| \dfrac{\kappa^\sg w_0}{\lambda}\right|+
{\bf O}_2 \   
\\
\tau^+&=&\tau + t^+ + \ \qquad \qquad \dfrac{\partial_{h^+}w_0}{\lambda}
\log \left| \dfrac{\kappa^\sg w_0}{\lambda}\right|+{\bf O}_1\
\\
\sg^+&=&\sg\text{ sgn }w. \quad \qquad \qquad 
\qquad \qquad \qquad \qquad  
\enal
\ee
\end{itemize}
where $\lb, \nu,$ and $\kappa^\sg$ are functions of $\eta^+$ 
and $t^+$ is an integer such that 
\be \label{t-condition}
\left| \tau+t^+ + \dfrac{\partial_{h^+}w_0}{\lambda}
\log \left| \dfrac{\kappa^\sg w_0}{\lambda}\right|
\right|<c^{-1}
\ee
\[
{\bf O}_1 = O^{(\eps^{1/4})}(\eps^{7/8}) \log \eps, \quad 
{\bf O}_2 = O^{(\eps^{1/4})}(\eps^{5/4}) \log^2 \eps.
\]
The superscript $\sg$ fixes the separatrix loop passed along 
by the trajectory.
\end{thm}

\rmk\label{exact symplectic}
 For $t^+$ satisfying (\ref{t-condition}) the separatrix 
map is given by 
\[
\eta = \cS_\xi, \quad \xi^+ = \cS_{\eta^+},\quad 
h=\cS_\tau,\quad \tau^+=t^+ + \cS_{h^+}, \quad 
\sg^+=\sg \cdot \text{sgn }w_0,
\]
where the generating function $\cS$ has the form 
\[
\cS(\eta^+,\xi,h^+,\tau,s,t^+)=\eta^+ \xi + h^+\tau+
\eps \Theta^\sg(\eta^+,\xi,\tau)+
\frac{w_0}{\lb}\log \left| \dfrac{\kappa^\sg w_0}{\lb e}\right|  
+O^*(\eps^{9/8}) \log ^2 \eps. 
\]
Notice that the map $\cS$ depends on $t^+$ only via the last term.

\subsection{Parameters of the separatrix maps for 
the generalized Arnold example}
\label{sec:separatrix-map-for-Arnold}

Notice that for the Arnold's example the unperturbed Hamiltonian 
is given by a direct product of $(I,\varphi)$ and $(p,q)$ variables:
$H_0=\frac{I^2}{2}+\frac{p^2}{2}+(\cos q-1).$
Using explicit formulas for 
$\lb, \kappa^\pm$ and $M^\pm$ in Section 6 \S  \cite{PT} 
we compute them.

The functions $\lb>0$, $\kappa^\pm>0$ and $\mu^\pm\in \R$ 
are defined by the unperturbed Hamiltonian $H_0$ as follows. 
Hypothesis \textbf{H2} implies that both eigenvalues of the matrix
\begin{equation}\label{def:MatrixLambda}
\Lb(I)=
\begin{pmatrix}
 -\partial_{pq}H_0(I,0,0)&-\partial_{qq}H_0(I,0,0)\\
 \partial_{pp}H_0(I,0,0)&\partial_{pq}H_0(I,0,0)
\end{pmatrix}
\end{equation}
are real and the trace of this matrix is equal to 0 for all $I$. 
We denote by $\lb(I)$ the positive eigenvalue of this matrix.

Let $\gm^\pm(I,\cdot):\R \to
\{(p, q) \in \text{Fig}_{\bf 8} : H_0(I_0, p, q) = H_0(I_0, 0, 0)\}$
be the natural parametrizations of the separatrix 
loops $\widehat \gm^\pm(r)$, i.e. 
$$
\dot \gm^\pm(y,t)=(-\partial_q H_0(I,\gm^\pm(t)),
\partial_p H_0(I,\gm^\pm(t)))
$$
and $a_\pm=a_\pm(I)$ be the left eigenvectors of $A$,
i.e. 
$$
a_+A=\lb a_+, \qquad a_-A=\lb a_-. 
$$
such that the $2\times 2$ matrix with rows $a_+$ and 
$a_-$ has unit determinant.  

In Proposition 6.3 \cite{PT} there are explicit 
formulas for $\kappa^\pm(I)$, given as integrals
of $a_+$ along $\gm^\pm$. In the case that the separatrix 
loops $\widehat \gm^\pm(I)$ are independent of $I$ we 
have that $\kappa^\pm$ are also independent 
of $I$ (see formulas (6.13--6.14)).

The natural parametrizations on $\widehat \gm^\pm$ 
are determined up to a time shift $t \mapsto t+\phi_\pm(I)$.  
Natural parametrizations are said to be {\it compatible} if they 
depend smoothly on $I$ and
\[
\lim_{t \to - \infty} \frac{ \langle a_+(I), \gamma^+(I, t) \rangle }{
\langle a_+(I), \gamma^-(I, t) \rangle  }=-1.
\]

Compatible paramet\-rizations are determined up to 
a simultaneous shift, namely, if 
$\gm^+(I,t^+(I,t)), \gm^-(I,t^-(I,t))$ is another pair 
of compatible paramet\-rizations, then 
$t^+(I,t) = t^-(I, t) = t - t_0(I)$ with a smooth function $t_0$.

If a solution of the non-perturbed system belongs to 
$\Gamma^\pm(I)$, it has the form
\be \label{eq:separatrix} 
\beal 
(I, \varphi, p, q)(t) &= 
\Gamma^\sg(I, \xi, \tau + t),\qquad \xi \in \T, \ \tau \in \R, 
\ \sg \in \{+, -\},
\\
 &  \Gamma^\sg(I,\xi,\tau) = (I,\xi + \nu(I) \tau,
\gm^\sg(I,\tau)).
\qquad \qquad \quad 
\enal
\ee
Let 
$$
H^\sg_*(I,\xi,\tau,t)=
H_1(\Gamma^\sg(I,\xi,\tau),t-\tau)-
H_1(I,\xi+\nu t,0,0,t-\tau).
$$
The functions $H^\sg_*(I,\xi,\tau,t)$ vanishes as $t\to \pm \infty$. 

\begin{prop} \label{poincare-melnikov}
Suppose that the parametrizations $\gm^\pm$ 
are compatible. Then
\[
M^\sg(I,\xi,\tau)=
-\int_{-\infty}^\infty H^\sg_*(I,\xi,\tau,t) dt.
\]
\end{prop}

The functions $M^\sg$ are called {\it splitting potentials.} 
They are 1-periodic with respect to $\xi$ and $\tau$.
We proved the following 

\begin{cor} \label{separatrix-for-arnold-example} 
For the generalized Arnold example (\ref{eq:our-Hamiltonian})
with trigonometric perturbations of the form (\ref{eq:perturb1})
there are constants $\kappa^\pm, c>0,$ and $\lb>0$ 
such that for $w=h^+-E(\eta^+)$ satisfying 
$c^{-1} \eps^{2}  < |w| < c \eps^{7/8}$ 
the separatrix map $\SM_\eps$ has the form 
\be \label{eq:separatrix-for-arnold-example} 
\beal 
\eta^+&=&\eta - \eps M^\sg_\xi (\eta^+,\xi,\tau)\ \qquad \qquad 
\qquad \quad +{\bf O}_2
\\
\xi^+&=&\xi + \eps M^\sg_{\eta^+}(\eta^+,\xi,\tau)-
\dfrac{\eta^+}{\lambda}
\log \left| \dfrac{\kappa^\sg w}{\lambda}\right|+{\bf O}_1
\\
h^+&=&h -  \eps M^\sg_{\tau}(\eta^+,\xi,\tau) \quad 
\qquad  \qquad \quad \quad 
+{\bf O}_2 
\\
\tau^+&=&\tau + t^+ +\quad \qquad  \qquad \frac 1 \lambda\ 
\log \left| \dfrac{\kappa^\sg w}{\lambda}\right|+{\bf O}_1 
\\
\sg^+&=&\sg\text{ sgn }w, \quad \qquad \qquad \qquad 
\qquad \qquad \qquad 
\enal
\ee
where
\[
{\bf O}_1 = O^*_1(\eps \log \eps), \qquad 
{\bf O}_2 = O^*_3(\eps^{2}) 
\]
and $t^+$ is an integer chosen so that $|\tau^+|<1$.  
\end{cor} 
\begin{rmk}
Here we expand the available domain to $c^{-1} \eps^{2}  < |w| < c \eps^{7/8}$ and re-evaluate the reminder $\mathbf{O}_1$, $\mathbf{O}_2$. This is because we improved the separatrix map and got a more precise expression in \cite{GKZ}, i.e. we can always find a canonical change of coordinate such that $\mathcal{SM}_{\epsilon}$ can be defined as follows:
\end{rmk}

\begin{thm}
For fixed $\beta>0$, $1\ge \varpi>0$ and $\eps$ sufficiently small,
there exist $c>0$ independent of $\eps$ and a canonical 
system of coordinates $(\eta,\xi,h,\tau)$ such that
\[
\eta = I +\mathcal O^*_{1}(\eps)+\mathcal O^*_{2}(H_0-E(I)), \ 
\xi + \nu(\eta)\tau=\varphi+f, \ 
h=H_0+\mathcal O^*_{1}(\eps)+\mathcal O^*_{2}(H_0-E(I)),
\]
where $f$ denotes a function depending only on 
$(I,p,q,\eps)$ and such that $f(I,0,0,0)=0$ and 
$f=\mathcal O(w+\eps)$. For any 
$\sigma\in \{-,+\}$ and $(\eta^+,h^+)$ such that 
\[
 c^{-1}\eps^{1+\varpi}<|w(\eta^+,h^+)|<c\eps,\qquad |\tau|<c^{-1}, \qquad
c<|w(\eta^+,h^+)|\,e^{\lb(\eta^+)\bar t}<c^{-1},
\]
where $\omega=\lambda^{-1}(h-E(\eta))+\mathcal{O}((h-E(\eta))^2)$ is a function of $h-E(\eta)$, the separatrix map $(\eta^+, \xi^+,h^+,\tau^+)=
\SM(\eta, \xi,h,\tau)$ is defined implicitly as follows
 \[
  \begin{split}
 \eta^+=&\ \eta- \ \ \eps
\partial_{\xi}M^{\sigma}(\eta^+,\xi,\tau)+\ \ \eps^2 M_2^{\sigma,\eta}+\ 
\ \mathcal O_3^*(\eps,|w|)|\log|w||\\
   \xi^+=&\ \xi
+\partial_{\eta^+} w(\eta^+,h^+) 
\left[\log |w(\eta^+,h^+)| + \log|\kappa^{\sigma}|
\right] 
\\
&+
\mathcal O_1^*(\eps+|w|)\,( |\log\eps|+|\log|w||)+\mathcal{O}_2^*(|\omega|)\\
 h^+=&\ h - \ \ \eps \partial_{\tau}M^{\sigma}(\eta^+,\xi,\tau)
+\ \ \eps^2 M_2^{\sigma,h}+\ \ \mathcal O_3^*(\eps,|w|)\\
 \tau^+=&\ \tau+\ \qquad \qquad  \bar t +\qquad \quad 
\partial_{h^+} w(\eta^+,h^+) 
\left[\log| w(\eta^+,h^+) |+ \log|\kappa^{\sigma}|)
\right] 
\\
&+ \mathcal O_1^*(\eps+|w|)\,( |\log\eps|+|\log|w||)+\mathcal{O}_2^*(|w|),
  \end{split}
 \]
where
$\bar t$ is an integer satisfying 
\be 
 \left|\tau+\bar t+{\partial_{h^+} w}
\log \left|{\kappa^\sigma w}\right|\right|<c^{-1}
\ee
and the functions $M_2^{\sigma,*}$ 
are evaluated at  $(\eta^+, \xi,h^+,\tau)$. 
\begin{rmk}
This Theorem is an application of Theorem \ref{thm:SM-higher-order:SR} for the Arnold-type Hamiltonian (\ref{eq:perturb1}). Recall that $c^{-1}\epsilon^{1+\varpi}<|w(h^+-E(\eta^+))|<c\epsilon$ and $w'(0)=\lambda^{-1}$, so we can simplify aforementioned expression by:
\begin{eqnarray}\label{modified SM}
 \eta^+=&\ \eta- \ \ \eps
\partial_{\xi}M^{\sigma}(\eta^+,\xi,\tau)+\ \ \eps^2 M_2^{\sigma,\eta}+\ 
\ \mathcal O_3^*(\eps)|\log\epsilon|\nonumber\\
   \xi^+=&\ \xi+\epsilon\partial_{\eta^+}M^{\sigma}(\eta^+,\xi,\tau)
-\frac{\eta^+}{\lambda}
\log \Big{|}\frac{\kappa^{\sigma}(h^+-E(\eta^+))}{\lambda}\Big{|}+
\mathcal O_1^*(\eps)|\log\eps|\\
 h^+=&\ h - \ \ \eps \partial_{\tau}M^{\sigma}(\eta^+,\xi,\tau)
+\ \ \eps^2 M_2^{\sigma,h}+\ \ \mathcal O_3^*(\eps)\nonumber\\
 \tau^+=&\ \tau+\bar t +\frac{1}{\lambda}\log\Big{|}\frac{\kappa^{\sigma(h^+-E(\eta^+))}}{\lambda}\Big{|}+ \mathcal O_1^*(\eps)|\log\eps|,\nonumber
\end{eqnarray}
which is of the same form with (\ref{eq:separatrix-for-arnold-example}) but has a preciser estimate of the reminders.
\end{rmk}

\end{thm}
In turns out that this Theorem also applies to general 
trigonometric perturbations of the form (\ref{eq:perturb})
after an additional change of coordinates. 
\begin{lem} 
For the the generalized Arnold example, i.e. for 
the Hamiltonian $H_\eps$ of the form
for (\ref{eq:our-Hamiltonian}) with trigonometric 
pertubations $\eps P$ (\ref{eq:perturb}) for 
any $k\ge 2$ in the $\beta$-nonresonant region 
$\mathcal D^{(2)}_\beta(P)$ has smooth change 
of coordinates $\Phi$ such that 
\[
H_\eps \circ \Phi (p,q,I,\varphi,t)= H_0(p,q,I)+
\eps H^*_1(p,q,I,\varphi,t)+\mathcal O_k^*(\eps^3),
\]
where $H_1^*(0,0,I,\varphi,t)\equiv 0$.
\end{lem}

\begin{rmk}\label{rmk:nonvanish-perturb}
Notice that the fact that $H_1^*$ vanishes on 
the cylinder $(p,q)=0$ implies that $w_0$ has 
the form $h^+-E(\eta^+)$. Indeed,  if the term 
$\mathcal O_k^*(\eps^3)$ is added to $w_0$,
then its partials  
$$
\left| \frac{\partial_* w_0}{\lb} 
\log \left|\frac{\kappa^\sg w_0}{\lb}\right|\,\right|
\le -C\eps^3\log \left|\frac{\kappa^\sg w_0}{\lb}\right|
$$
for some $C>0$ and $*\in \{\eta,\xi,h,\tau\}$. Notice 
that the $C^1$-norm of this expression on the right 
for $w\in (\eps^{3/2},\eps)$ is bounded by 
$\mathcal O(\eps^{3/2})$ and belongs to the remainder 
term in (\ref{eq:separatrix-for-arnold-example}). 
\end{rmk}
 
\begin{proof} The proof is an application of the normal
form derived in \cite{GKZ}. The set up studied there 
covers the generalized Arnold's example. 
In Lemma 4.1 \cite{GKZ} we rewrite the Hamiltonian 
$H_\eps (p,q,I,\varphi,t)$ in Moser's coordinates 
$$
\cH_\eps=H_\eps \circ \mathcal F_0 (x,y,I,\varphi,t)=
\cH_0+\eps \cH_1=
$$, 
$$
H_0 \circ \mathcal F_0 (x,y,I,\varphi,t)+
\eps H_1\circ \mathcal F_0 (x,y,I,\varphi,t),
$$
where $x=0$ is the stable manifold 
and $y=0$ is the unstable manifold of saddle 
$(p,q)=0$. 

In Lemma 4.5 \cite{GKZ} for $|xy| \in (\eps^{3/2},\eps)$ 
we find a smooth coordinate change $\Phi'$ such that 
$$
\cH_\eps \circ \Phi'=
\cH_0 +\mathcal O^*\left(\eps^{3}\beta^{-4}
+\eps^{2}\beta^{-2}\wh x\wh y+\eps(\wh x\wh y)^2\right),
$$
where the skew symmetric norm is defined in 
(\ref{eq:skewsym-norm}). Notice that 
$\overline {\bf H}_j,\ j=1,2,3$ from this Lemma vanish in 
the $\beta$-nonresonant region (see Section 5.1 right after 
this Lemma). 

The change of coordinates $\Phi'$ is $\eps$-close to 
the identity and can be molified outside of a neighborhood 
of $(p,q)=0$  as the identity. 
\end{proof}

\subsection{Computation of the splitting potential}
\label{sec:splitting-potential}
Consider the generalized Arnold example with 
the Hamiltonian (\ref{eq:our-Hamiltonian}) 
with perturbations of the form (\ref{eq:perturb1}). 
By Remark \ref{rmk:nonvanish-perturb} the case of 
general trigonometric perturbations reduces to this case. 
Thus, in this case we have 
\begin{equation}\label{h1}
H_\pm^\sg(\eta,\xi,\tau,t)=(1-\cos q^\sg(t-\tau)) 
P(\exp(i (\xi+\eta t)),\exp(i t)),
\end{equation}
where $P$ is a real valued trigonometric polynomial, i.e. 
for some $N$ we have 
\[
P(\exp(i \xi),\exp(i t))= \sum_{|k_1|,|k_2|\le N}  
p'_{k_1,k_2} \cos (k_1 \xi+k_2 t)+
p''_{k_1,k_2} \sin (k_1 \xi+k_2 t).
\]
The case of general trigonometric perturbations 
in discussed above.

Using formula (1.2) in Bessi \cite{Be} for the Arnold example 
for every harmonic 
$$
 p_{k_1,k_2} \exp i(k_1 \xi_t+k_2 t)=
p_{k_1,k_2} \exp i(k_1 (\xi_0+ \eta t) +k_2 t),\ \xi=\xi_0
$$ 
and we have 
\[
\int_\R [1-\cos q^\sg(t)] 
\cos 2\pi (k_1 (\xi_0 + \eta \tau)  + k_2 t+k_2 \tau))\ dt=\]\[
=2\pi \dfrac{(k_1 \eta +k_2 )
\cos 2\pi (k_1 \xi_0  +(k_1\eta+k_2) \tau)}{\sinh 
\frac{\pi (k_1 \eta +k_2 )}{2}},
\]
where $\xi_t=\xi+\eta t$ for all $t\in \R$. 

Combining we have 
\begin{lem} \label{lem:melnikov}
Let $H_\pm^\sg(\eta,\xi,\tau,t)=(1-\cos q^\sg(t-\tau)) 
P(\exp(i (\xi+\eta t)),\exp(i t))$, then the associated splitting potential 
has the form: 
\begin{eqnarray*}
M^\sg(\eta,\xi,\tau)&=&-2 \pi \sum _{|k_1|,|k_2|\le N}   
\left[p'_{k_1,k_2} \,\dfrac{(k_1\eta+k_2)}{\sinh \frac{\pi (k_1\eta+k_2)}{2}}\cos (k_1\xi+ (k_1\eta+k_2)\tau) \right.\\
&&\left.+p''_{k_1,k_2} \,\dfrac{(k_1\eta+k_2)}{\sinh \frac{\pi(k_1\eta+k_2)}{2}}\sin (k_1\xi+ (k_1\eta+k_2)\tau)\right],
\end{eqnarray*}
where $\xi,\tau\in\T,\ \eta\in \R$. 
\end{lem}

%
%

\subsection{Properties of the Melnikov potential}
\label{sec:NHL-diffusion-condition}

Suppose the splitting potentials $M^+ (\eta,\xi,\tau)$ satisfies 
the following condition:

\begin{itemize}
\item[{\bf [M1]}]
 There are two smooth families $\tau_{i}(\eta,\xi), \ i=0,1$ such that 
for each point $(\eta,\xi)\in {\bf K}\times \T$ we have 
\[
(\partial_\tau M^+ (\eta,\xi,\tau)-
\eta\,\partial_\xi M^+ (\eta,\xi,\tau)) |_{\tau=\tau_i(\eta,\xi)}=0
\ \text{ and }\]
\[ 
(\partial^2_{\tau\tau} M^+ (\eta,\xi,\tau)- 
2\eta\,\partial^2_{\xi \tau} M^+ (\eta,\xi,\tau)+
\eta^2 \partial^2_{\xi \xi} M)|_{\tau=\tau_i(\eta,\xi)}
\ne 0.\quad   
\]
\end{itemize}

We choose $\tau_\pm(I,\varphi)$ with values in $(-1,1)$. 
Similarly, one can define this condition for $M^-(I,\varphi,\tau)$. 
Condition {\bf [M1]} is natural in the sense that 
$$
(\partial_\tau M^+ (\eta,\xi,\tau)-
\eta\,\partial_\xi M^+ (\eta,\xi,\tau))
$$ 
is the time derivative of  
the Melnikov function and \newline 
$$
\partial^2_{\tau\tau} M^+ (\eta,\xi,\tau)- 
2\eta\,\partial^2_{\xi \tau} M^+ (\eta,\xi,\tau)+
\eta^2 \partial^2_{\xi \xi} M^+(\eta,\xi,\tau)
$$ 
is the second order time derivative.

In this section we verify that the condition {\bf [M1]} holds
for an open class of trigonometric pertrubations 
$H_1(q,\varphi,t)$. By Lemma \ref{lem:melnikov} we have 
\be \label{eq:Melnikov-deriv}
\beal
M^+(\eta,\xi,\tau)&=&2 \pi \sum _{|k_1|,|k_2|\le N}   
\left[p'_{k_1,k_2}(\eta) \dfrac{(k_1\eta+k_2) 
\cos (k_1 \xi+(k_1\eta+k_2)\tau )}{\sinh \frac{\pi (k_1\eta+k_2)}{2}}\right.\\
&&\left.-p''_{k_1,k_2}(\eta) \dfrac{(k_1\eta+k_2)
\sin (k_1 \xi+(k_1\eta+k_2)\tau )}{\sinh \frac{\pi (k_1\eta+k_2)}{2}}\right],\\
M^+_\tau(\eta,\xi,\tau)&=&2 \pi \sum _{|k_1|,|k_2|\le N}   
\left[p'_{k_1,k_2}(\eta) \dfrac{(k_1\eta+k_2)^2 
\sin (k_1 \xi+(k_1\eta+k_2)\tau )}{\sinh \frac{\pi (k_1\eta+k_2)}{2}}\right.\\
&&\left.
-p''_{k_1,k_2}(\eta) \dfrac{(k_1\eta+k_2)^2 
\cos (k_1 \xi+(k_1\eta+k_2)\tau )}{\sinh \frac{\pi (k_1\eta+k_2)}{2}}\right]\\
M^+_\xi(\eta,\xi,\tau)&=&2 \pi \sum _{|k_1|,|k_2|\le N}   
\left[p'_{k_1,k_2}(\eta) \dfrac{(k_1\eta+k_2)k_1 
\sin (k_1 \xi+(k_1\eta+k_2)\tau )}{\sinh \frac{\pi (k_1\eta+k_2)}{2}}\right.\\
&&\left.
-p''_{k_1,k_2}(\eta) \dfrac{(k_1\eta+k_2)k_1
\cos (k_1 \xi+(k_1\eta+k_2)\tau )}{\sinh \frac{\pi (k_1\eta+k_2)}{2}}\right].
\enal
\ee
Fix $\rho>0$.  
Consider the generalized Arnold example and 
assume that for some $a>0$ we have  
\be \label{eq:pert-cond}
\beal 
p'_{0,1}=\sinh \frac \pi 2,\quad  |p'_{1,0}|\le a,
\ |p'_{i,j}|,|p''_{i,j}|\le a, \textrm{ for all }0\le i,j\le N, i+j\ge 2,
\\
|(p'_{1,1},p''_{1,1})|,
|(p'_{i,j},p''_{i,j})|\ge \rho a\ \textrm{ for some odd }i\ne 0
\textrm{ and an even } j.
\enal 
\ee

In addition, assume that $a$ is small, then 
by Lemma \ref{lem:melnikov} we have  
\be \label{eq:Melnikov-deriv-perturb}
\beal
M^+(\eta,\xi,\tau)&=:& 2\pi \cos \tau+ 
2\pi a \overline M^+(\eta,\xi,\tau)\ \, \\
M^+_\tau(\eta,\xi,\tau)&=:& -2\pi \sin \tau+ 
2\pi a \overline M_\tau^+(\eta,\xi,\tau).
\enal
\ee

\begin{lem} \label{lem:suff-cond-lamination}
If conditions (\ref{eq:pert-cond}) holds, then 
conditions {\bf [M1]} are satisfied for 
all $(\eta,\xi)\in \R\times \T$. 
\end{lem}
\begin{proof} 
Notice that coefficients in front of each harmonic 
$\sin (k_1 \xi+(k_1\eta+k_2)\tau )$ and 
$\cos (k_1 \xi+(k_1\eta+k_2)\tau )$ have the form
$(k_1\eta+k_2)^d/\sinh \frac{\pi (k_1\eta+k_2)\tau )}{2}$
for $d=1,2$. This expression tends to zero as 
$\eta\to \infty$. Since we have only finitely many 
harmonics, we can choose $a$ small enough so that 
we have $O(a)$ uniformly in $\eta$. 

Due to the implicit theorem and previous coefficient estimate, the condition 
$$
M^+_\tau(\eta,\xi,\tau)=0
$$ 
holds for $\tau_-=O(a)$ or $\tau_+=\pi+O(a)$. 
This is because
\be \label{eqn:Melnikov-nondeg}
|M^+_{\tau\tau}(\eta,\xi,\tau_\pm)| >2\pi-\mathcal{O}(a)>\pi \text{ for each } 
(\eta,\xi)\in 
{\bf K}\times \T
\ee
by taking $a$ small enough.

\end{proof}

One can check that even in the case 
$H_1=(1-\cos q)(\cos \varphi+\cos  t )$
condition {\bf [M1]} is violated at $I=1,\ \varphi=\pm 1/2$. 
In this case, we have only one zero $\tau=\mp 1/2$.  
In the case $H_1=(1-\cos  q)(a \cos \varphi+\cos  t)$
with any $|a|<1$ condition {\bf [M1]} is satisfied. 
In addition, we need to assume that $a$ is small.

\section{Construction of isolating blocks 
and existence of a NHIL}
\label{sec:construction-nhil}

In this section we construct a normally hyperbolic 
invariant lamination ${\bf \Lb}_\eps$. 
It has three steps. We state the main result 
of this section in subsection \ref{sec:existence-nhil}.
Then in subsection \ref{sec:SM-linearize}
we analyze the linearization of $\SM_\eps$.
In subsection \ref{sec:IB-centers} we construct 
almost fixed cylinders $\SM_\eps(C_{ii})\approx C_{ii},
\ i=0,1$ and almost period two cylinders 
$ \SM_\eps(C_{01})\approx C_{10},\ \SM_\eps(C_{10})
\approx C_{01}$. In subsection \ref{sec:cone-cond}
we construct a Lipschitz NHIL by verifying {\bf C1 to C5} conditions 
from Appendix \ref{sec:suff-nhil} and finally in subsection \ref{sec:cr-regularity}
we improve the smoothness of leaves by Theorem \ref{cr} and prove the 
H\"older continuity between different leaves.

\subsection{A Theorem on existence of NHIL}
\label{sec:existence-nhil}

In this section we construct Normally Hyperbolic
Invariant Lamination (NHIL) using isolating block 
construction presented in Appendix \ref{sec:suff-nhil}. 

\begin{figure}[h]
  \begin{center}
  \includegraphics[width=8.5cm]{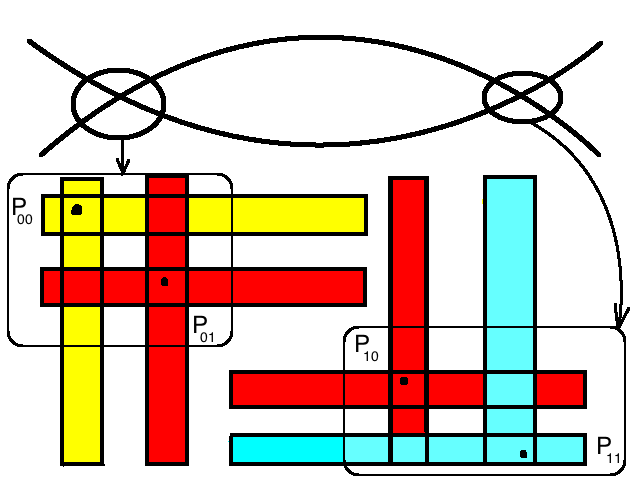}
  \end{center}
  \caption{Isolating blocks for NHIL}
  \label{fig:NHIL}
\end{figure}

To define centers of isolating blocks $P_{ij},\ i,j=\{0,1\}$
as on Fig. \ref{fig:NHIL} we prove existence of 
four sets of functions: 
\be \label{eq:laminations-isol-block-centers}
\beal 
h_{ii}(\eta,\xi,\eps), \ w_{ii}(\eta,\xi,\eps),\ \tau_i(\eta,\xi,\eps),\quad i=0,1
\qquad \text{and }\\
h_{ij}(\eta,\xi,\eps), \ w_{ij}(\eta,\xi,\eps),\ 
\tau_{ij}(\eta,\xi,\eps),\quad i\ne j,\
i,j \in \{0,1\}. 
\enal 
\ee
such that for all $(\xi,\eta)\in {\bf K}\times \T$  equations 
(\ref{eq:nhil-fixed-legs}) and (\ref{eq:nhil-permuted-legs}) hold. 
See Lemmas \ref{lem:fixed-centers} and \ref{lem:period-two-centers}. 
We also have 
\[
w_*(\eta,\xi,\eps)\equiv h_*(\eta,\xi,\eps)-\frac{\eta^2}{2}.
\]  
In Lemma \ref{lem:eigenvalues} we compute eigenvectors 
$v_j(x)$ and eigenvalues $\lb_j(x),\ j=1,\dots,4$ of the rescaled 
linearization of the separatrix map (under new $(\eta,\xi,I,t)-$coordinate).
Since $\SM_\eps$ is symplectic, eigenvalues of its linearization 
$D\SM_\eps$ at any point at come in pairs: one pair of eigenvalues 
$\lb_{1,2}$ is close to one, the other pair is $\lb_3 \sim c\dt$ and  
$\lb_4 \sim (c\dt)^{-1}$. Note that there is no immediate dynamical implication 
from these eigenvectors as we do not claim existence of 
fixed points. However, these eigenvectors are used 
to construct a cone field in section \ref{sec:cone-cond}. 


Denote for $i,j\in\{0,1\}$
\be \label{eq:vij}
v_{4}^{ij}(\eta,\xi,\eps)=
v_{4}^{ij}(\eta,\xi,h_{ij}(\eta,\xi,\eps),\tau_{ij}(\eta,\xi,\eps)). 
\ee
Fix small $\dt>0$, some $\kappa>0$ and define 
the following four sets:  

\be \label{eq:isolating-block}
\beal
\Pi^{\dt,\kappa}_{ij}:=\left\{(\eta,\xi,h,\tau): \text{there is }
(\eta_0,\xi_0) \in {\bf K}\times \T,\ 
|\dt_3|\le \kappa_1 \dt,|\dt_4|\le \kappa_2 \dt^2\ \right. \\ 
\text{ such that } \quad  (\eta,\xi,h,\tau)=
\qquad \qquad \qquad \qquad \qquad \qquad \qquad \qquad 
\\
\left. 
(\eta_0,\xi_0,h_{ij}(\eta_0,\xi_0,\eps),\tau_{ij}(\eta_0,\xi_0,\eps))
+\dt_3 L_\eps v_{3}^{ij}(\eta_0,\xi_0,\eps) +
\dt_4 L_\eps v_{4}^{ij}(\eta_0,\xi_0,\eps)\right\}.
\enal 
\ee


These sets $\Pi^{\dt,\kappa}_{ij}, \ i,j\in\{0,1\}$ can be viewed 
as the union of parallelograms centered at  $(\eta_0,\xi_0,
h^{ij}_\eps(\eta_0,\xi_0), \tau^{ij}_\eps(\eta_0,\xi_0))$ 
with $(\eta_0,\xi_0)$ varying inside ${\bf K}\times \T$.


\vskip 0.1in 

Consider the Hamiltonian $H_\eps=H_0+\eps P_N$, 
given by (\ref{eq:our-Hamiltonian}) and let $P$ be a polynomial 
such that the associated Melnikov function $M^{\pm}$, given by 
Lemma \ref{lem:melnikov}, satisfies (\ref{eq:Melnikov-deriv-perturb}).

Let $\Sigma=\{0,1\}^\Z$ be the space of infinite sequences on two symbols, $\om=(\dots,\om_{-1},\om_0,\om_1,\dots)\in \Sigma$,
and $\sg:\Sigma\to \Sigma$ be the shift, i.e. 
$\sg \om=\om'$, where $\om'_{i+1}=\om_i$ for all $i\in\Z$.
Let $\mathbb A_0:=\mathcal D_0\times \T \subset 
\mathbb A:=\R \times \T$ be a cylinder, 
$(\eta,\xi)\in \mathcal D_0\times \T$. We call a map
$$
F:\mathbb A_0 \times \Sigma \to \mathbb A \times \Sigma
$$
a $C^r$ smooth skew-product map, if it is given by 
$$
F:(\eta,\xi,\om)\longmapsto (\eta',\xi',\om')=(f_\om(\eta,\xi),\sigma \om),
$$
where $f_\om:\mathbb A_0\to \mathbb A$ is a family of 
$C^r$ smooth cylinder maps with $C^r$ dependence on $\om$, 
i.e. the difference of $f_\om-f_{\om'}$ goes to zero with respect 
to the $C^r$ norm if $\om-\om' \to 0$. 
See also Appendix \ref{sec:nhil-skew} for related
definitions. 

\begin{thm} Fix small $\rho>0$. Suppose the trigonometric 
polynomial $P$ from (\ref{eq:perturb}) satisfies 
(\ref{eq:Melnikov-deriv-perturb}) for small $a>0$, then for 
$\kappa^s>0$, depending on $H_0$ and $H_1$ only, and 
any $\eps>0$ small enough the associated separatrix 
map $\SM_\eps$, given by 
(\ref{modified SM}), 
has a NHI\footnote{as a matter of fact this lamination is 
weakly invariant in the sense that if we extend this lamination 
to a $\mathcal O(\eps)$-neighbourhood of $\mathcal D_0$, 
then $\SM_\eps ({\bf \Lb}_{\eps})$ is a subset of the extension 
of ${\bf \Lb}_\eps$. In other words, the only way orbits can 
escape from ${\bf \Lb}_{\eps}$ are throught the boundary 
$\partial \mathcal D_0$.}L, 
denoted ${\bf \Lb}_{\eps}$, i.e. 
$$
{\bf \Lb}_{\eps} \subset \cup_{ij \in\{0,1\}}\ \Pi^{\dt,\eps}_{ij}.
$$
Moreover, there is a map 
$$
C:\mathbb A_0 \times \Sigma \to {\bf \Lb}_{\eps}
$$
such that for each 
$(\eta,\xi,\om)\in \mathcal D_0\times \T\times \Sigma$ 
we have   
$$
\SM_\eps (C(\eta,\xi,\om))=
C(f_\om(\eta,\xi),\sigma \om).
$$
In other words, for a $C^r$ smooth skew-product map $F$
the following diagram commutes:
\be \label{eq:first-skew-shift-diagram}
\begin{array}[c]{ccc} 
\quad\ {\bf \Lb}_\eps&
\quad  \stackrel{\SM_\eps} {\longrightarrow} & 
\quad {\bf \Lb}_\eps \\
C\ \uparrow &  & C\ \uparrow \\ 
\ \ \ \ | &  & \ \ \ \ | \\ 
\mathbb A_0 \times \Sigma
& \quad  \stackrel{F} {\longrightarrow} & \ \ 
\mathbb A \times \Sigma.
\end{array}
\ee
In addition, there exists $k\in\mathbb{N}$ such that
\[
[\omega]_k:=(\omega_{-k},\cdots,\omega_0,\cdots,\omega_k)
\]
is a truncation and exist functions $\tau_{[\omega]_k}(\eta,\xi)$, $I_{[\omega]_k}(\eta,\xi)$
such that
the map $F(\cdot,\om)$ has the following form 
\be \label{eq:separatrix-for-arnold-example-model} 
\beal 
\eta^+&=&\eta - \eps M^+_\xi (\eta^+,\xi,\tau_{[\om]_k}(\eta,\xi))+
{\bf O}_2,
\\
\xi^+&=&\xi +\epsilon M_{\eta^+}^+(\eta^+,\xi,\tau_{[\omega]_k}(\eta,\xi))
-\dfrac{\eta^+}{\lambda} \log 
\left| \dfrac{\kappa^s\,I_{[\sigma\om]_k}(\eta^+,\xi^+)}{\lambda}\right|
+{\bf O}_1.
\enal
\ee

\end{thm} 

\begin{rmk} 
Smallness of $a$ is independent from the size the compact 
domain ${\bf K}\subset \R$, because $\xi$-dependent 
components of the Melnikov function average out (see 
the proof of Lemma \ref{lem:suff-cond-lamination}).  

Notice that in (\ref{eq:separatrix-for-arnold-example-model}) there exists one invalid term $\epsilon M_{\eta^+}^+(\eta^+,\xi,\tau_{[\omega]_k}(\eta,\xi))$ because it's smaller than the reminder $\mathbf{O}_1$. We leave it in this position to match the system (\ref{eq:separatrix-for-arnold-example}) better.

Actually, $\Big{\{}(\eta,\xi,I_{\omega}(\eta,\xi),t_{\omega}(\eta,\xi))\Big{|}(\eta,\xi)\in{\bf K}\times\mathbb{T},\omega\in\Sigma\Big{\}}$ is the coordinate of NHIL (see section \ref{sec:cr-regularity}). The H\"older continuity of $\omega$ benefits us with a finite truncation and we just need to consider $I_{\omega,k}$ and $t_{\omega,k}$ instead. The error caused by truncation can be much less than the $\mathbf{ O}_1$ and $\mathbf{ O}_2$ terms. 


\end{rmk} 

\begin{proof} The proof consists of following parts:
\begin{itemize}
\item Derive properties of the linearization $D\SM_\eps$ near 
zeroes of the Melnikov potential $M^+_\tau-\eta M^+_\xi=0$ such 
as eigenvalues and eigenvectors (see Lemma \ref{lem:eigenvalues}).

\item Find an approximately invariant cylinders for 
separatrix map $\SM_\eps$: for $i,j\in \{0,1\}$ we have 
$$
C_{ij}:{\mathcal D}_0\times \T \to {\mathcal D}\times \T \times \R \times \T  
$$
$$
C(r,\theta)=(\eta(r,\theta),\xi(r,\theta),\tau_{ij}(\eta,\xi,\eps),w_{ij}(\eta,\xi,\eps))
$$
so that 
$$
\SM_\eps (C_{01}({\mathcal D}_0\times \T)) \approx 
C_{10}({\mathcal D}_0\times \T)$$
$$
\SM_\eps (C_{10}({\mathcal D}_0\times \T))\approx 
C_{01}({\mathcal D}_0\times \T).
$$ 
These cylinders play the role of centers of the isolating 
blocks containing the normally hyperbolic lamination
(see points $P_{ij}$ on Fig. \ref{fig:NHIL}). 

\item Show that for proper $\kappa$ the 
$\kappa_1 \dt \times \kappa_2 \dt^2$-paralellogram 
neighborhoods of these cylinders
$\Pi^{\dt,\kappa}_{ij}$, given by (\ref{eq:isolating-block})
satisfy conditions [C1-C5].  
\item Prove NHIL's H\"older dependence of $\omega$ and the smoothness of every leaf, which leads to a skew product satisfiy (\ref{eq:separatrix-for-arnold-example-model}).
\end{itemize}
\end{proof}

\subsection{Properties of the linearization of $\SM_\eps$}
\label{sec:SM-linearize}
We star with the setting of Corollary \ref{separatrix-for-arnold-example}. Actually, (\ref{eq:separatrix-for-arnold-example}) is enough to achieve the existence of NHIL. But we should keep in mind the reminders $\mathbf{O}_1$ and $\mathbf{O}_2$ can be further evaluated due to (\ref{modified SM}). In the sequel we limit to the symbol $\sigma = +$, 
in that we consider the map to be undefined when $w < 0$, 
and show that $\SM_\epsilon$ has an NHIL. With almost the same procedure we can get the NHIL for the case $\sigma=-$. The system \eqref{eq:separatrix-for-arnold-example} can be seen as two 
coupled subsystems, to see this more clearly, define
\[
    I =\frac{1}{\epsilon}( h - E(\eta)),
\]
which also includes a rescaling. Note that 
\begin{multline*}
   \epsilon I^+= h^+ -  E(\eta^+) =  h  - \epsilon M^+_\tau -[E(\eta)+E'(\eta)(\eta^+-\eta)]+ {\bf O}_2 \\ 
    = \epsilon I -\epsilon \bigl(M^+_\tau(\eta, \xi, \tau) -E'(\eta) M^+_\xi(\eta, \xi, \tau) \bigr) + {\bf O}_2
\end{multline*}
since $\eta^+ - \eta  = O(\epsilon)$. We will also omit the superscripts from $M^+$ and $\kappa^+$. Then
\begin{equation}\label{eq:rescaled-sep}
\begin{aligned}
\eta^+ & = \eta &{}-  \epsilon M_\xi (\eta, \xi, \tau) & &{} + 
{\bf O}_2  &\\
\xi^+ & = \xi &{} +  \epsilon M_\eta (\eta, \xi, \tau) & - \frac{\eta^+}{\lambda} \log \left(  \frac{\epsilon \kappa I^+}{\lambda}  \right) &{} + {\bf O}_1 &\\
I^+ & = I & {} -  \bigl(M_\tau - E'(\eta)M_\xi)(\eta, \xi, \tau\bigr) & &{} + \frac{1}{\epsilon}{\bf O}_2 & \\
\tau^+ & = \tau &   + \frac{1}{\lambda} \log \left(  \frac{\epsilon \kappa I^+}{\lambda}  \right) &\mod 2\pi &+\;  {\bf O}_1  &
\end{aligned}
\end{equation}
We removed the absolute value from the $\log$ term and noting the map is undefined for $I^+ < 0$. As (\ref{eq:perturb1}) is mechanical, $\omega=h^+-{\eta^{+2}}/{2}$.

\begin{lem} \label{lem:eigenvalues}
Consider the separatrix map $\SM_\eps$ for 
the generalized example of Arnold 
(\ref{eq:our-Hamiltonian}-\ref{eq:perturb}). Suppose the Melnikov potential 
$M(\eta,\xi,\tau)$ satisfies condition {\bf [M1]}. Then 
for some positive $C,\nu$ and any sufficiently small $\dt$ such 
that $\eps^{\varpi}\leq \dt$, $1\geq\varpi>1/4$ for any 
$$
x=(\eta,\xi,I,\tau)\in {\bf K}\times \T\times
(-C\dt,C\dt)\times \T
$$ 
the differential D$\SM_\eps$ has eigenvalues 
$$
|\lambda_i-1|\le C\eps^{1/8} \log \eps,\ i=1,2,
\qquad  |\lambda_4|<C\dt ,\ |\lambda_3|>\frac{1}{2C\dt}.
$$
For $|\eta|\ge \nu$ there are eigenvectors 
$e_j(x),\ j=1,2,3,4$, i.e.
\be \label{eq:eigenvectors}
  \textup{D} \SM_\eps(x)\, e_j(x)=\lb_j(x)e_j(x).
\ee
such that 
$$
e_3(x)=\frac{(0,\eta,0,-1)}
{\sqrt{1+\eta^2}}+\mathcal O(\dt)
$$
$$
e_4(x)=\frac{(0,\eta,
\Delta M,-1)
}{\sqrt{1+\eta^2+\Delta M^2}}
+\mathcal O(\dt^2)
$$
$$
e_{1,2}(x)=
\frac{(0,-M_{\tau\tau}+\eta M_{\xi\tau},0,M_{\xi\tau}-\eta M_{\xi\xi})}{\sqrt{(-M_{\tau\tau}+\eta M_{\xi\tau})^2+(-M_{\xi\tau}+\eta M_{\xi\xi})^2}}+\mathcal{O}(\epsilon^{1/8}\log\epsilon)
$$
with $\Delta M=M_{\tau\tau}-2\eta M_{\tau\xi}+\eta^2 M_{\xi\xi}$.

In particular, for each $(\eta,\xi)\in {\bf K}\times \T$
angles between $e_i(x)$ and $e_j(x)$ with $i\ne j$ 
and $\{i,j\}\ne \{1,2\}$ is uniformly away from zero. Moreover, for each $x$ such that $\dt$ satisfies the above conditions the vector $\textup{D}\SM_\eps(x)e_4$ in absolute value is bounded by $C\dt$. 

\end{lem}

\begin{proof} Denote $w=\eps\dt/\lb$. 
The differential of the separatrix map D$\SM_\eps^+$ for 
the Arnold's example (\ref{eq:separatrix-for-arnold-example})
is given by:
\[
\begin{pmatrix}
\frac{\partial\eta^+}{\partial \eta}& \frac{\partial\eta^+}{\partial \xi} & 0 &\frac{\partial\eta^+}{\partial \tau}\\
-\frac{1}{\lambda}\log\frac{\epsilon I^+}{\lambda}\frac{\partial\eta^+}{\partial\eta}-\frac{\eta^+}{\lambda I^+}\frac{\partial\Delta}{\partial \eta}   &  1-\frac{\eta^+}{\lambda I^+}\frac{\partial \Delta}{\partial \xi}-2\pi l\frac{\partial \eta^+}{\partial \xi} & -\frac{\eta^+}{\lambda I^+}  &    -\frac{\eta^+}{\lambda I^+}\frac{\partial \Delta}{\partial \tau}-2\pi l\frac{\partial \eta^+}{\partial \tau}  \\
\frac{\partial \Delta}{\partial \eta} & \frac{\partial \Delta}{\partial \xi}
 &1& \frac{\partial \Delta}{\partial \tau}\\
\frac{1}{\lambda I^+}\frac{\partial \Delta}{\partial \eta}  &  \frac{1}{\lambda I^+}\frac{\partial \Delta}{\partial \xi}   & \frac{1}{\lambda I^+} &   1+\frac{1}{\lambda I^+}\frac{\partial \Delta}{\partial \tau} \\
\end{pmatrix} +
\begin{pmatrix}
      \mathbf{O}_2    \\
            \mathbf{O}_1    \\
      \mathbf{O}_2/\epsilon    \\
      \mathbf{O}_1 
\end{pmatrix},
\]
which can be translated into

{\tiny
$$ \left( \begin {array}{cccc} 
1-\eps M_{\xi \eta}
& 
-\eps M_{\xi\xi} 
&
0
&
-\eps M_{\xi\tau}
\\ 
\noalign{\medskip}
-\eta^+
\frac {-  M_{\tau \eta}+M_\xi+\eta M_{\xi\eta}}{\lb I^+}
- (1-\eps M_{\xi \eta}) \log  \left| {\eps I^+} \right| 
&
1
-\frac {\eta^+\,\gamma}{\lb I^+}
+\eps M_{\xi\xi} \log|{\eps I^+}|
&
\frac {-\eta^+}{\lb I^+} 
&
\frac {-\eta^+\,\alpha}{\lb I^+}
+\eps M_{\xi\tau} \log|{\eps I^+}|
\\ 
\noalign{\medskip}
-  M_{\tau \eta}+M_\xi+\eta M_{\xi\eta}&
\gamma&
1&
\alpha
\\ 
\noalign{\medskip}
\frac {-  M_{\tau \eta}+M_\xi+\eta M_{\xi\eta}}{\lb I^+}
&
{\frac {\gamma}{\lb I^+}}
&
{\frac {1}{\lb I^+}}
&
1+{\frac {\alpha}{\lb I^+}}\end {array} \right),
$$}
where:
$$\Delta=-M_\tau+\eta M_\xi,$$
$$\gamma=-M_{\xi\tau}+\eta M_{\xi\xi},$$
$$ \alpha=-M_{\tau\tau}+\eta M_{\xi\tau},$$
$$\zeta=-M_{\tau\eta}+M_\xi+\eta M_{\xi\eta},$$
$$\beta=\frac{1}{\lambda I^+},$$
$$ l\in\mathbb{Z}\;\text{such that}\; l=[\frac{1}{2\pi\lambda}\log\frac{\kappa\epsilon}{\lambda}]$$
and the error of entries in the first and third rows is 
$\mathcal O(\eps^{5/4}\log^2 \eps)$, $\mathcal O(\eps^{1/4}\log^2 \eps)$ and the error of 
entries in the second and forth rows is 
$\mathcal O(\eps^{7/8}\log^2 \eps)$. Notice that the $\kappa$ and $\lambda$ are just contants in the original separatrix map (\ref{eq:perturb1}), so we can remove them from the matrix.

As the separatrix map is symplectic, the determinant of this matrix should be one (although we take a new coordinate). So we can get a couple of eigenvalues close to 1

\[
\lambda_{1,2}(x)=1\pm \mathcal \mathcal O(\eps^{1/8} \log \eps).
\]

 This point can be verified from a simple calculation:
\[
det(D\mathcal{SM}_{\epsilon}^+-\lambda Id)=(1-\lambda)^4-(1-\lambda)^2\lambda\frac{\alpha-\eta\gamma}{\delta}+\mathcal{O}(\epsilon^{1/4}\log^2\epsilon)=0.
\]
Neglecting error terms of order $\cO(\eps^{1/4}\log^2 \eps)$,
we get 
%
%
 \be \label{eq:trace}
 \text{trace} (\text{D}\SM_\eps^+)=
 4+\frac{\alpha-\eta \gamma}{\dt}.
 \ee
Due to {\bf [M1]}, for each $(\eta,\xi)\in {\bf K}\times \T$ we have 
$$
\alpha-\eta\gamma=\Delta M=M_{\tau\tau}-2\eta M_{\tau\xi}+\eta^2 M_{\xi\xi}\ne 0
$$
uniformly hold, then for small enough $\delta$, this trace should be $\mathcal{O}(1/\delta)$. So there should existsthe other couple of eigenvalues 
\[
\lambda_3(x)\sim\mathcal{O}(1/\delta),\quad\lambda_4(x)\sim\mathcal{O}(\delta)
\]
because the determinant equals one.

Now we compute approximation of the eigenvectors: 
$$
D\SM_\eps(x)e_j(x)=\lb_j(x)e_j(x), \quad j=1,2,3,4, 
$$

%
so we can estimate the eigenvalues 
\begin{equation}\label{eigenvalue}
\begin{aligned}
\lambda_{1,2} & =1+\mathcal{O}(\epsilon^{1/8}\log\epsilon),\\
\lambda_3 & =\frac{2+\alpha\beta-\eta^+\beta\gamma+sgn(\alpha)\sqrt{(2+\alpha\beta-\eta^+\beta\gamma)^2-4}}{2}+\mathcal{O}(\epsilon/\delta),\\
\lambda_4 & =\frac{2+\alpha\beta-\eta^+\beta\gamma-sgn(\alpha)\sqrt{(2+\alpha\beta-\eta^+\beta\gamma)^2-4}}{2}+\mathcal{O}(\epsilon\delta),
\end{aligned}
\end{equation}
and corresponding eigenvectors by
\begin{equation}
\begin{aligned}
e_{1,2} &=(0,\frac{\alpha}{\sqrt{\alpha^2+\gamma^2}},0,-\frac{\gamma}{\sqrt{\alpha^2+\gamma^2}})^t+\mathcal{O}(\epsilon^{1/8}\log\epsilon),\\
e_3 &=\frac{|\lambda_3-1|\cdot|\beta|}{\sqrt{\lambda_3^2(1+\eta^{+2})\beta^2+(\lambda_3-1)^2}}(0,\frac{\eta^+\lambda_3}{1-\lambda_3},\frac{1}{\beta},\frac{\lambda_3}{\lambda_3-1})^t+\mathcal{O}(\delta),\\
 &\approx\frac{1}{\sqrt{1+\eta^2}}(0,\eta,0-1)+\mathcal{O}(\delta),\\
e_4 &=\frac{1}{\sqrt{\lambda_4^2\beta^2(1+\eta^{+2})+(\lambda_4-1)^2}}(0,-\lambda_4\beta\eta^+,\lambda_4-1,\lambda_4\beta)^t+\mathcal{O}(\delta^2),\\
& \approx\frac{1}{\sqrt{1+\eta^2+(\alpha-\eta\gamma)^2}}(0,\eta,\alpha-\eta\gamma,-1)+\mathcal{O}(\delta^2)\\
\end{aligned}
\end{equation}
with $\beta\sim\mathcal{O}(1/\delta)$. {\bf [M1]} ensures the angles between different eigenvectors are uniformly away from zero. Change $\alpha$, $\gamma$ back into the notation depending on $M$ we proved the Lemma.

\end{proof}

\subsection{Calculation of centers of isolating blocks}
\label{sec:IB-centers}

In this section we calculate the set of functions $w$'s and $h$'s
from (\ref{eq:laminations-isol-block-centers}). 
Recall that the sepatatrix map can be written in the new coordinate
$$
\SM_\eps(\eta,\xi,I,\tau)=(\eta^+,\xi^+,I^+,\tau^+)
$$ 
with
$w=\epsilon I+\epsilon\Delta(\eta,\xi,\tau)+\mathbf{O}_2$. So we just need to get weak invariant functions 
\be 
\beal 
I_{ii}(\eta,\xi,\eps),\ \tau_{ii}(\eta,\xi,\eps),\quad i=0,1
\qquad \text{and }\\
I_{ij}(\eta,\xi,\eps), \ 
\tau_{ij}(\eta,\xi,\eps),\quad i\ne j,\
i,j \in \{0,1\}
\enal 
\ee
which satisfy the following Lemma.

\begin{lem}  \label{lem:fixed-centers}
Suppose condition (\ref{eq:Melnikov-deriv-perturb}) holds 
for a sufficiently small $a>0$. Fix $1\geq \varpi>1/4>\rho>0$. Then for 
a sufficiently small $\eps>0$ and $\dt \in (\eps^{\varpi},\eps^{\rho})$
there are functions $\tau_{ii}$ and $I_{ii},\ i=0,1$  such that 
$I_{ii}=\mathcal O( \dt)$ and these functions satisfy 
\be \label{eq:nhil-fixed-legs}
\beal 
\eta^+&=\eta -  \eps M^+_{\xi}(\eta,\xi,\tau_{ii}(\xi,\eta,\eps))
\\
\xi^+&=\xi + \eps M^+_{\eta^+}(\eta,\xi,\tau_{ii}(\xi,\eta,\eps))-
\dfrac{\eta^+}{\lambda}
\log \left| \dfrac{\kappa^+ \epsilon I_{ii}(\xi^+,\eta^+,\eps)}{\lambda}\right|\;\text{mod}\;2\pi
\\
 \left| \ \tau_{ii}\right .& \left.(\eta,\xi,\eps) + 2n\pi + \frac 1 \lambda\ 
\log \left| \dfrac{\kappa^+ \epsilon I_{ii}(\xi^+,\eta^+,\eps)}{\lambda}\right|-
\tau_{ii}(\eta^+,\xi^+,\eps) \right | \le \mathcal O(a\,\dt^2)
\\
\left| I_{ii}\right .& \left.(\eta,\xi,\eps)  
+\Delta(\eta,\xi,\tau_{ii}(\eta,\xi,\eps))- I_{ii}(\eta^+,\xi^+,\eps) \right| \le \mathcal O( \dt^2),
\enal
\ee
where $I_{ii}(\eta,\xi,\eps)>0$ for all $(\eta,\xi)\in {\bf K}\times \T$. 

Moreover, these solutions satisfy 
\be \label{eq:wii}
I_{ii}(\eta,\xi,\eps)&=& \dt +a \dt \bar{I}_{ii}(\eta,\xi,a)\\
\tau_{ii}(\eta,\xi,\eps)&=& {i\pi}+a\tau_{i}^1(\eta,\xi,a)+
a\dt\, \tau_{ii}^2(\eta,\xi,a)\nonumber
\ee
for some smooth functions $\bar I_{ii},\ \tau_{i}^1,
\  \tau_{ii}^2$
with $i\pi+a\tau_{i}^1$ solving the first implicit equation in {\bf [M1]}. 
\end{lem}

Notice that this lemma says that neglecting the error terms in 
the separatrix map $\SM_\eps$ from Corollary \ref{separatrix-for-arnold-example}
has two weak invariant cylinders
\[
\Lb_{ii}=
\{(\eta,\xi,h_{ii}(\eta,\xi),\tau_{ii}(\eta,\xi)):\ (\eta,\xi) \in {\bf K}\times \T\}
\] 
with $\omega_{ii}=I_{ii}(\eta,\xi,\epsilon)+\Delta(\eta,\xi,\tau_{ii}(\eta,\xi,\epsilon))$
and $h_{ii}=\epsilon I_{ii}+\eta^2/2$. Denote by $\Lb_{ii}^*$ the invariant cylinder obtained by extending 
$h_{ii}$ and $\tau_{ii}$ for an $O(\eps)$-neighbour\-hood of $\bf K$. 
Then up to error terms $\SM_\eps(\Lb_{ii})\subset \Lb^*_{ii}.$

\begin{proof}

Start by proving existence of $w_{ii},\, \tau_{ii}$'s solving functional 
equations (\ref{eq:nhil-fixed-legs}) for $i=0,1$. 

\begin{itemize}
\item By {\bf M1} we have 
\[
M^+_{\tau}(\eta,\xi,\tau_i(\eta,\xi)) -\eta
M^+_{\xi}(\eta,\xi,\tau_i(\eta,\xi))  =0.
 \]
Actually we can take $\tau_i(\eta,\xi)={i\pi}+a\tau_{i}^1(\eta,\xi,a)$ where $a$ is sufficiently small. This can be derived from the $\mathcal{O}(a)$ estimate. We formally solve the $\tau_{i}^1(\eta,\xi,a)$ by
\[
\tau_{i}^1(\eta,\xi,a)=\frac{\partial_t\overline{M}(\eta,\xi,i\pi)-\eta\partial_{\xi}\overline{M}(\eta,\xi,i\pi)}{2\pi\cos i\pi}+\mathcal{O}(a).
\]

\item Take the formal solution (\ref{eq:wii}) into the separatrix map. It should satisfies
\begin{equation} \label{1-order estimate}
\begin{aligned}
\eta^{+} & = \eta-a\epsilon \partial_{\xi}\overline{M}^s(\eta,\xi,\tau_{ii}(\eta,\xi,a))+\mathbf{O}_2,\\
\xi^{+} & = \xi+2\pi\{n\eta\}-\frac{\eta}{\lambda}\log(1+a{\bar{I}_{ii}(\eta^+,\xi^+)})+\mathbf{O}_1,\\
a\delta\bar{I}_{ii}(\eta^+,\xi^+) & = a\delta\bar{I}_{ii}(\eta,\xi)-\Big{\{}\Delta(\eta,\xi,i\pi+a\tau_{i}^1)+[\partial_{tt}M^s(i\pi+a\tau_{i}^1)\\
&-a\eta\partial_{\xi t}\overline{M}(i\pi+a\tau_{i}^1)\big{]}a\delta \tau_{ii}^2\Big{\}}+\mathcal{O}(a^2\delta^{2}),\\
a\tau_{i}^1(\eta^+,\xi^+)& = a\tau_{i}^1(\eta,\xi)+\frac{1}{\lambda}\ln(1+{a\bar{I}_{ii}(\eta^+,\xi^+)})+\mathcal{O}(a\delta),
\end{aligned}
\end{equation}
where $\delta=\frac{\lambda}{\kappa\epsilon}\exp(2n\lambda\pi)$. Within the third equation,
 \[
 \Delta(\eta,\xi,i\pi+a\tau_{i}^1)=0\]
  due to the first item, and 
  \[
  \partial_{tt}M^s(i\pi+a\tau_{i}^1)-a\eta\partial_{\xi t}\overline{M}(i\pi+a\tau_{i}^1)\neq 0 
  \]
as $a$ sufficiently small. So we can update the third equation into
\[
\bar{I}_{ii}(\eta^+,\xi^+)=\bar{I}_{ii}(\eta,\xi)+\alpha(\eta,\xi,i\pi+a\tau_{i}^1)\tau_{ii}^2+\mathcal{O}(a\delta)
\]
with $\alpha$ defined in previous section. 
Since $\eta $ belongs to a compact region $\bf K$, $a$ 
can be chosen sufficiently small and
\[
\begin{aligned}
\bar{I}_{ii}(\eta^+,\xi^+)&=\bar{I}_{ii}(\eta,\xi^+)+\mathcal{O}(a\epsilon)\\
&=\bar{I}_{ii}(\eta,\xi+2\pi\{n\eta\}-\frac{\eta}{\lambda}\ln(1+a{\bar{I}_{ii}(\eta^+,\xi^+)}))+\mathbf{O}_1\\
&=\bar{I}_{ii}(\eta,\xi+2\pi\{n\eta\}-\frac{\eta}{\lambda}\ln(1+a{\bar{I}_{ii}(\eta,\xi+2\pi\{n\eta\})}))+\mathcal{O}(a^2),
\end{aligned}
\]

so we solve the fourth equation of (\ref{1-order estimate}) and get
%
%
\[
\bar{I}_{ii}(\eta,\sigma(\eta,\xi))=\lambda[\tau_{i}^1(\eta,\sigma(\eta,\xi))-\tau_{i}^1(\eta,\xi)]+\mathcal{O}(a)
\]
where $\sigma(\eta,\xi)=\xi+2\pi\{n\eta\}$. Take this equation back into the third equation of (\ref{1-order estimate}) and get
\[
\tau_{ii}^2(\eta,\xi)=\frac{\bar{I}_{ii}(\eta,\sigma(\eta,\xi))-\bar{I}_{ii}(\eta,\xi)}{\alpha(i\pi+a \tau_{i}^1)}+\mathcal{O}(a\delta).
\]

\end{itemize}
\end{proof}

\begin{lem} \label{lem:period-two-centers}
Suppose condition (\ref{eq:Melnikov-deriv-perturb}) 
holds for a sufficiently small $a>0$. Fix $1\geq\varpi>1/4>\rho>0$. Then for 
a sufficiently small $\eps>0$ and $\dt \in (\eps^{\varpi},\eps^{\rho})$
there are functions $\tau_{ij}$ and $I_{ij},\ \{i,j\}=\{0,1\}$  such that 
$I_{ij}=\mathcal O( \dt)$ and these functions satisfy 
\be \label{eq:nhil-permuted-legs}
\beal 
\eta^+&=\eta -  \eps M^+_{\xi}(\eta,\xi,\tau_{ij}(\xi,\eta,\eps)) 
\\
\xi^+&=\xi + \eps M^+_{\eta^+}(\eta,\xi,\tau_{ij}(\xi,\eta,\eps))-
\dfrac{\eta^+}{\lambda}
\log \left| \dfrac{\kappa^+ w_{ji}(\xi^+,\eta^+,\eps)}{\lambda}\right|
\\
 \left| \ \tau_{ij}\right .& \left.(\eta,\xi,\eps) + 2n\pi + \frac 1 \lambda\ 
\log \left| \dfrac{\kappa^+ w_{ij}(\xi^+,\eta^+,\eps)}{\lambda}\right|-
\tau_{ji}(\eta^+,\xi^+,\eps) \right| \le \mathcal O(a\,\dt^2)
\\
\left| I_{ij}\right .& \left.(\eta,\xi,\eps) - 
 \left[M^+_{\tau}(\eta,\xi,\tau_{ij}(\eta,\xi,\eps)) -\eta
M^+_{\xi}(\eta,\xi,\tau_{ij}(\eta,\xi,\eps)) \right] \right. 
\\ & \left. \qquad \qquad \qquad  \qquad \qquad \qquad 
 - I_{ji}(\eta^+,\xi^+,\eps) \right| \le \mathcal O( \dt^2),
\enal
\ee
where $I_{ij}(\eta,\xi,\eps)>0$ for all $(\eta,\xi)\in {\bf K}\times \T$. 

Moreover, these solutions satisfy 
\be \label{eq:wij}
I_{ij}(\eta,\xi,\eps)=\dt e^{-\lambda\pi} (1+
a \bar I_{ij}(\eta,\xi,a))
\ee
\[
\tau_{ij}(\eta,\xi,\eps)= {i\pi}+a\tau_{i}^1(\eta,\xi,a)+
a\dt\, \tau_{ij}^2(\eta,\xi,a)
\]
for some smooth functions $\bar{I}_{ij},\ \tau_{i}^1,\ \tau_{ij}^2$
with $i\pi+a\tau_{i}^1$ solving the first implicit equation in {\bf [M1]}. 
\end{lem}
Neglecting the error terms in 
the square of separatrix map $\SM_\eps^2$ from Corollary \ref{separatrix-for-arnold-example}
we get two weak invariant cylinders
\[
\Lb_{ij}=
\{(\eta,\xi,h_{ij}(\eta,\xi),\tau_{ij}(\eta,\xi)):\ (\eta,\xi) \in {\bf K}\times \T\}
\] 
with $\omega_{ij}=I_{ij}(\eta,\xi,\epsilon)+\Delta(\eta,\xi,\tau_{ij}(\eta,\xi,\epsilon))$
and $h_{ij}=\epsilon I_{ij}+\eta^2/2$. Denote by $\Lb_{ij}^*$ the invariant cylinder obtained by extending 
$h_{ij}$ and $\tau_{ij}$ for an $O(\eps)$-neighbour\-hood of $\bf K$. 
Then up to error terms $\SM_\eps(\Lb_{ij})\subset \Lb^*_{ji}.$

\begin{proof}

We use almost the same procedure as previous Lemma. Start by proving existence of $w_{ij},\ \tau_{ij}$'s solving 
functional equations (\ref{eq:nhil-fixed-legs}) for $\{i,j\}=\{0,1\}$. 

\begin{itemize}
\item Recall that we have already solved the $\tau_{i}^1(\eta,\xi)$ by
\[
\tau_{i}^1(\eta,\xi,a)=\frac{\partial_t\overline{M}(\eta,\xi,i\pi)-\eta\partial_{\xi}\overline{M}(\eta,\xi,i\pi)}{2\pi\cos i\pi}+\mathcal{O}(a),
\]
which satisifes
\[
M^+_{\tau}(\eta,\xi,i\pi+a\tau_{i}^1(\eta,\xi)) -\eta
M^+_{\xi}(\eta,\xi,i\pi+a\tau_{i}^1(\eta,\xi))  =0.
 \]

\item Take the formal solution (\ref{eq:wij}) into the separatrix map, which should satisfies
\begin{equation} 
\left\{
\begin{aligned}
\eta^{+} & = \eta-a\epsilon \partial_{\xi}\overline{M}^s(\eta,\xi,\tau_{01}(\eta,\xi,a))+\mathbf{O}_2,\\
\xi^{+} & = \xi+2\pi\{(n-\frac{1}{2})\eta\}-\frac{\eta}{\lambda}\ln(1+a{\bar{I}_{10}(\eta^+,\xi^+)})+\mathbf{O}_1,\\
\bar{I}_{10}(\eta^+,\xi^+,a)& =\bar{I}_{01}(\eta,\xi,a)+\alpha(\eta,\xi,i\pi+a\tau_{0}^1(\eta,\xi,a)) \tau_{01}^2+\mathcal{O}(a\delta)\\
a\tau_{1}^1(\eta^+,\xi^+,a)& = a\tau_{0}^1(\eta,\xi)+\frac{1}{\lambda}\ln(1+{a\bar{I}_{10}(\eta^+,\xi^+,a)})+\mathcal{O}(a\delta),
\end{aligned}
\right.
\end{equation}
and
\begin{equation} 
\left\{
\begin{aligned}
\eta^{+} & = \eta-a\epsilon \partial_{\xi}\overline{M}^s(\eta,\xi,\tau_{10}(\eta,\xi,a))+\mathbf{O}_2,\\
\xi^{+} & = \xi+2\pi\{(n-\frac{1}{2})\eta\}-\frac{\eta}{\lambda}\ln(1+a{\bar{I}_{01}(\eta^+,\xi^+)})+\mathbf{O}_1,\\
\bar{I}_{01}(\eta^+,\xi^+,a)& =\bar{I}_{10}(\eta,\xi,0)+\alpha(\eta,\xi,i\pi+a\tau_{1}^1)\tau_{10}^2+\mathcal{O}(a\delta)\\ 
a\tau_{0}^1(\eta^+,\xi^+,a)& = a\tau_{1}^1(\eta,\xi)+\frac{1}{\lambda}\ln(1+{a\bar{I}_{01}(\eta^+,\xi^+,a)})+\mathcal{O}(a\delta).
\end{aligned}
\right.
\end{equation}
Recall that 
\[
\alpha(\eta,\xi,i\pi+a\tau_{i}^1)\neq 0 ,\quad i=0,1,
\]
uniformly hold for sufficient small $a$ and
\[
\xi^+=\xi+2\pi\{(n-\frac{1}{2})\eta\}+\mathcal{O}(a)
\]
because $\eta $ belongs to a compact region $\bf K$ and $a$ 
can be chosen sufficiently small. So we can solve the solution by
\begin{equation}
\left\{
\begin{aligned}
\bar{I}_{10}(\eta,\sigma) & =\lambda(\tau_{1}^1(\eta,\sigma)-\tau_{0}^1(\eta,\xi))+\mathcal{O}(a),\\
\bar{I}_{01}(\eta,\sigma) & =\lambda(\tau_{0}^1(\eta,\sigma)-\tau_{1}^1(\eta,\xi))+\mathcal{O}(a),\\\end{aligned}
\right.
\end{equation}
and
\begin{equation}
\left\{
\begin{aligned}
\tau_{01}^2(\eta,\xi) & =\frac{\bar{I}_{10}(\eta,\sigma)-\bar{I}_{01}(\eta,\xi)}{\alpha(\eta,\xi,i\pi+a\tau_{ii}^1)}+\mathcal{O}(a\delta),\\
\tau_{10}^2(\eta,\xi) & =\frac{\bar{I}_{01}(\eta,\sigma)-\bar{I}_{10}(\eta,\xi)}{\alpha(\eta,\xi,i\pi+a\tau_{ii}^1)}+\mathcal{O}(a\delta),\\
\end{aligned}
\right.
\end{equation}
where $\sigma(\eta,\xi)=\xi+2\pi\{(n-\frac{1}{2})\eta\}$ and $\delta=\frac{\lambda}{\kappa\epsilon}\exp(2n\lambda\pi)$.

\end{itemize}
\end{proof}

\subsection{Verification of isolating block conditions [C1-C5]}
\label{sec:cone-cond}

In this section we isolating 
blocks $\Pi^{u,s,\kappa}_{ij}=\Pi_{ij}, i,j=\{0,1\}$ and verify {\bf C1-C3} for them. Then we define cone field over each point in isolating blocks and verify {\bf C4-C5}. Recall that In Lemma \ref{lem:eigenvalues} we compute 
eigenvectors $e_j(x)$ and eigenvalues $\lb_j(x),\ j=1,\dots,4$ of 
the $\text{D}\SM_\eps(x)$.

Since the map $\SM_\eps$ is symplectic, eigenvalues come in pairs: 
one pair of eigenvalues $\lb_{1,2}$ is close to one, the other pair is 
$\lb_3 \sim c\dt$ and  $\lb_4 \sim (c\dt)^{-1}$. Besides, from Lemma \ref{lem:fixed-centers} and Lemma \ref{lem:period-two-centers} we get the eigenvectors based on the centers by 
$$
v_{s}^{ij}(\eta,\xi,\eps)=
v_{s}^{ij}(\eta,\xi,I_{ij}(\eta,\xi,\eps),\tau_{ij}(\eta,\xi,\eps)), \quad 
s=3,4. 
$$
For small $\dt>0$ and properly large $\kappa>0$ 
we define the following four sets: 
{\small
\be \label{isolating block}
\beal
\Pi^{s,\kappa}_{ij}:=\left\{(\eta,\xi,I,\tau): \text{there is }
(\eta_0,\xi_0) \in {\bf K}\times \T,\qquad \qquad \qquad  
\qquad \qquad \qquad \qquad \right. \\ 
|c|\le \kappa_1 \dt^2,|d|\le \kappa_2 \dt\ 
\text{ such that } \qquad \qquad \qquad \qquad 
\qquad \qquad \qquad \qquad \\ \left. (\eta,\xi,I,\tau)=
(\eta_0,\xi_0,I_{ij}(\eta_0,\xi_0,\eps),\tau_{ij}(\eta_0,\xi_0,\eps))
+c v_{3}^{ij}(\eta_0,\xi_0,\eps) +
d v_{4}^{ij}(\eta_0,\xi_0,\eps)\right\}.
\enal 
\ee}

We drop $\eps$-dependence for brievity. 
These sets $\Pi^{s,\kappa}_{ij}, \ i,j\in\{0,1\}$ can be viewed 
as the union of parallelograms centered at  $(\eta_0,\xi_0,
I^{ij}_\eps(\eta_0,\xi_0), \tau^{ij}_\eps(\eta_0,\xi_0))$ 
with $(\eta_0,\xi_0)$ varying inside ${\bf K}\times \T$.

By Lemma \ref{lem:eigenvalues} we derive that 
these eigenvectors of $D\SM_\eps$ have the following form:
$$
e_3(x)=\frac{(0,\eta,0,-1)}
{\sqrt{1+\eta^2}}+\mathcal O(\dt)
$$
$$
e_4(x)=\frac{(0,\eta,
\Delta M,-1)
}{\sqrt{1+\eta^2+\Delta M^2}}
+\mathcal O(\dt^2)
$$
$$
e_{1,2}(x)=
\frac{(0,-M_{\tau\tau}+\eta M_{\xi\tau},0,M_{\xi\tau}-\eta M_{\xi\xi})}{\sqrt{(-M_{\tau\tau}+\eta M_{\xi\tau})^2+(-M_{\xi\tau}+\eta M_{\xi\xi})^2}}+\mathcal{O}(\epsilon^{1/8}\log\epsilon).
$$
%
%
%

Define 
\[
\psi:=\min \{\measuredangle(v_3(x),v_4(x)):
\ x\in \Pi_{ij},\ i,j\in\{0,1\}\}.
\]

\begin{lem}\label{angle}
If condition {\bf [M1]} holds, then $\psi>0$ uniformly holds. 
\end{lem}

\begin{proof}
By definition $v_3(x)\parallel\approx (0,\eta,0,-1)$ 
and $v_4(x)\parallel\approx (0,\eta,\alpha-\eta\gamma,-1)$. By condition {\bf [M1]} we have 
\[
\Delta M=\alpha-\eta \gamma\ne 0
\]
uniformly for $(\eta,\xi)\in\mathbf{K}\times\mathbb{T}$.
\end{proof}

\begin{figure}[h]
  \begin{center}
  \includegraphics[width=9.5cm]{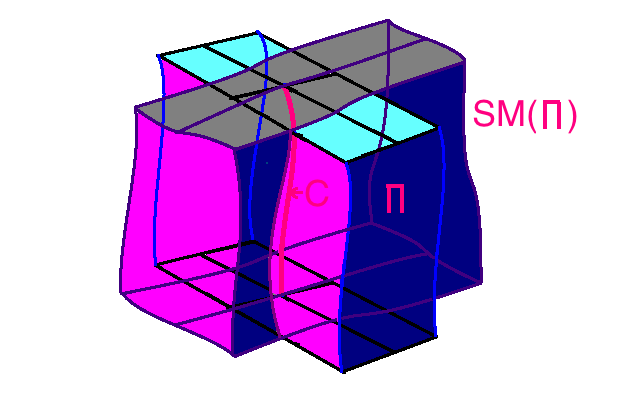}
  \end{center}
  \caption{Isolating Block}
  \label{fig:isolating-block}
\end{figure}
For any point in the isolating block $\Pi_{ij}^{s,\kappa}$, we can define a local transformation by:
\[
\Phi_{ij}:(\eta,\xi,I,\tau)\rightarrow(\eta_0,\xi_0,c,d),
\]
where $c$ and $d$ are the projections of $(\eta-\eta_0,\xi-\xi_0,I-I_{ij}(\eta_0,\xi_0),\tau-\tau_{ij}(\eta_0,\xi_0))$ to $e_3$ and $e_4$ with $(\eta_0,\xi_0)$ the corrsponding point in the center. Under this new coordinate, we have
\[
\mathfrak{SM}_{\epsilon}:=\Phi_{ji}\circ\mathcal{SM}_{\epsilon}\circ\Phi_{ij}^{-1}
\]
defined on the new straight grids
\[
\mathfrak{N}_{ij}^{s,\kappa}=\Phi_{ij}(\Pi_{ij}^{s,\kappa}).
\]
Then we can prove the following stronger conditions:
\begin{lem}
\begin{equation}\label{shrink}
\Big{|}\pi_4\circ\mathfrak{SM}_{\epsilon}\mathfrak{N}_{ij}^{s,\kappa}\Big{|}\leq\mathcal{O}(\delta^2),
\end{equation}
\begin{equation}\label{expand}
\Big{|}\pi_3\circ\mathfrak{SM}_{\epsilon}\mathfrak{N}_{ij}^{s,\kappa}\Big{|}\geq\mathcal{O}(\delta)
\end{equation}
and
\begin{equation}\label{homotopy}
\mathfrak{SM}_{\epsilon}\partial^u\mathfrak{N}_{ij}^{s,\kappa}\simeq\partial^u\mathfrak{N}_{ij}^{s,\kappa},
\end{equation}
for any $(\eta,\xi)\in\mathbf{K}\times\mathbb{T}$. Here $\partial^u$ means the boundary of the 3-rd component and $\simeq$ means homotopic equivalence.
\end{lem}
Recall that {\bf C3} is actually ensured by the weak invariance of centers in aforementioned section. So in the following we just need to prove {\bf C1 and C2}, which  can be derived from this Lemma.
\begin{proof}
We remove the $\epsilon$ dependence for convenience. $\forall ({\eta},\xi,I,\tau)\in \Pi_{ij}^{s,\kappa}$, which corresponds to a unique point $(\eta_0,\xi_0,c,d)\in\mathfrak{N}_{ij}^{s,\kappa}$, we get
\begin{equation}\label{1 jet}
\begin{aligned}
\mathfrak{SM}(\eta_0,\xi_0,c,d)-&(\eta_0^+,\xi^+_0,0,0)=\Phi_{ji}[\mathcal{SM}(\eta,\xi,I,\tau)-(\eta_0^+,\xi_0^+,I_{ji}(\eta_0^+,\xi_0^+),\tau_{ji}(\eta_0^+,\xi_0^+))]\\
&=\Phi_{ji}\int_{0}^{1}D\mathcal{SM}(s\eta+(1-s)\eta_0,s\xi+(1-s)\xi_0,sI+\\
&\quad\quad(1-s)I_{ij}(\eta_0,\xi_0), s\tau+(1-s)\tau_{ij}(\eta_0,\xi_0))\cdot(cv_3^{ij}+dv_4^{ij})ds\\
&=\Phi_{ji}\Big{[}D\mathcal{SM}(\eta_0,\xi_0,I_{ij}(\eta_0,\xi_0),\tau_{ij}(\eta_0,\xi_0))\cdot(cv_3^{ij}+dv_4^{ij})+\\
&\quad\quad\int_0^1\Upsilon(s)(cv_3^{ij}+dv_4^{ij})ds\Big{]}\\
&=\Phi_{ji}\Big{[}c\lambda_3v_3^{ij}+d\lambda_4v_4^{ij}+\int_0^1(\Upsilon'+\mathcal{O}(\kappa_2,\delta))(cv_3^{ij}+dv_4^{ij})ds\Big{]},
\end{aligned}
\end{equation}
where $\Upsilon(s)=D\mathcal{SM}\Big{|}_{0}^{s}$ is a variational matrix with $s\in(0,1)$. Formally it equals to
\[
\begin{pmatrix}
   0 & 0 & 0 & 0\\
   Osc(-\frac{1}{\lambda}\ln I^+-\eta^+\beta\zeta)   & Osc(-\eta^+\beta\gamma) & Osc(-\eta^+\beta) &Osc(-\eta^+\beta\alpha)  \\
   Osc(\zeta) & Osc(\gamma) & 0 & Osc(\alpha)\\

        Osc(\zeta\beta) & Osc(\beta\gamma) & Osc(\beta) & Osc(\beta\alpha)   \\
\end{pmatrix}
+\delta\mathbf{O}_1,
\]
where the `Osc' means the variation between $0$ and $s$. Recall that we take $\epsilon^{\varpi}\lesssim\delta\lesssim\epsilon^{\rho}$, $v_{3}^{ij}$ is $\delta-$parallel to $(0,\eta_0,0,-1)$ and $v_4^{ij}$ is $\delta^2-$parallel to $(0,\eta_0,\alpha(\eta_0,\xi_0)-\eta_0\gamma(\eta_0,\xi_0),-1)$. By removing the $\mathcal{O}(\kappa_2,\delta)$ order error, we can simplify $\Upsilon$ by
\[
\Upsilon'=
\begin{pmatrix}
   0 & 0 & 0 & 0\\
-\eta Osc(\beta\zeta)-\frac{1}{\lambda}Osc(\ln I^+)  & -\eta Osc(\beta\gamma) & -\eta Osc(\beta) & -\eta Osc(\beta\alpha)\\
   0 & 0 & 0 & 0\\
        Osc(\zeta\beta) & Osc(\beta\gamma) & Osc(\beta) & Osc(\beta\alpha)   \\
\end{pmatrix}.
\]
Here the $\mathcal{O}(\kappa_2,\delta)$ implies the error term depends on $\kappa_2$. This is because both $(0,\eta_0,0,-1)$ and $(0,\eta_0,\alpha(\eta_0,\xi_0)-\eta_0\gamma(\eta_0,\xi_0),-1)$ have a degenerate first component. Another observation is that for any vector $V$ linearly composed by $v_3^{ij}$ and $v_4^{ij}$, $\Upsilon\cdot V$ is $\delta-$parallel to $(0,\eta_0,0,-1)$. Besides, we get the norm estimate $|\Upsilon'v_4^{ij}|\sim\mathcal{O}(1)$.

Due to these observations, (\ref{shrink}) and (\ref{expand}) are almost obvious now:
\begin{equation}
\begin{aligned}
|\pi_4\circ\mathfrak{SM}\mathfrak{N}_{ij}^{s,\kappa}| =&|\pi_4(\mathfrak{SM}(\eta_0,\xi_0,c,d)-(\eta_0^+,\xi^+_0,0,0))|\\
\leq & |\lambda_4d|+\delta(|\pi_4\Upsilon' c v_3^{ij}|+|dv_4^{ij}|)\\
 \leq & C(\kappa_1,\kappa_2)\delta^2,
\end{aligned}
\end{equation}

\begin{equation}\label{C2}
\begin{aligned}
|\pi_3\circ\mathfrak{SM}\mathfrak{N}_{ij}^{s,\kappa}| =&|\pi_3(\mathfrak{SM}(\eta_0,\xi_0,c,d)-(\eta_0^+,\xi^+_0,0,0))|\\
\geq & |\lambda_3c|-|\Upsilon'dv^{ij}_4|-\mathcal{O}(\kappa_2,\delta^2)\\
 \geq & (C(\kappa_1)-C(\kappa_2))\delta\geq C(\kappa_1)\delta/2,
\end{aligned}
\end{equation}
where $C(\kappa_i)$ is $\mathcal{O}(1)$ constants depending on $\kappa_i$, $i=1,2$,  so we can always take $\kappa_1$ properly greater than $\kappa_2$ such that previous inequalities hold.

As for the homotopy equivalence of (\ref{homotopy}), we can use the same approach as in \cite{KZ14} by lifting $\mathcal{SM}$ by $\overline{\mathcal{SM}}$ in the covering space $(\eta,\xi,I,\tau)\in\mathbf{K}\times\mathbb{R}\times(-C\delta,C\delta)\times\mathbb{T}$ and the isolating blocks $\mathfrak{N}_{ij}^{s,\kappa}$ by $\overline{\mathfrak{N}}_{ij}^{s,\kappa}$. The benefit of doing this is that $\partial^u\overline{\mathfrak{N}}_{ij}^{s,\kappa}$ becomes single connected. So the boundary corresponds to a fixed $|c|=\kappa_1\delta^2$ into (\ref{1 jet}), of which we can always take a properly small $\kappa_2$ and get a slightly deformed new boundary $\mathfrak{SM}\partial^u\mathfrak{G}_{ij}^{s}$ which is also single connected. 
\end{proof}
%
\vspace{10pt}

Use almost the same approach we can prove {\bf C1'-C3'} for $\mathcal{SM}_{\epsilon}^{-1}$. We drop $\eps$-dependence for brievity. Suppose $x^+=\mathcal{SM}(x)$ for $x=(\eta,\xi,I,\tau)$, then we get
\[
D\mathcal{SM}^{-1}(x^+)=(D\mathcal{SM}(x))^{-1}
\]
and
\[
D\mathcal{SM}^{-1}(x^+)e_i(x)=\frac{1}{\lambda_i}e_i(x),\quad i=3,4.
\]
Notice that $e_i(x)\in T_{x^+}\mathbb{R}^4$ is a parallel shift from $T_x\mathbb{R}^4$ of Euclid metric. For small $\dt>0$ and properly large $\kappa>0$ 
we define the following four sets: 
{\small
\be \label{isolating block}
\beal
\Pi^{u,\kappa}_{ij}:=\left\{(\eta^+,\xi^+,I^+,\tau^+): \text{there is }
(\eta^+_0,\xi^+_0) \in {\bf K}\times \T,\qquad \qquad \qquad  
\qquad \qquad \qquad \qquad \right. \\ 
|c|\le \kappa_3 \dt,|d|\le \kappa_4 \dt^2\ 
\text{ such that } \qquad \qquad \qquad \qquad 
\qquad \qquad \qquad \qquad \\ \left. (\eta^+,\xi^+,I^+,\tau^+)=
(\eta_0^+,\xi_0^+,I_{ij}(\eta_0^+,\xi_0^+,\eps),\tau_{ij}(\eta_0^+,\xi_0^+,\eps))
+c v_{3}^{ji}(\eta_0,\xi_0,\eps) +
d v_{4}^{ji}(\eta_0,\xi_0,\eps)\right\}
\enal 
\ee}
with $(\eta_0,\xi_0,I_{ji}(\eta_0,\xi_0),\tau_{ji}(\eta_0,\xi_0))=\mathcal{SM}^{-1}(\eta_0^+,\xi_0^+,I_{ij}(\eta_0^+,\xi_0^+),\tau_{ij}(\eta_0^+,\xi_0^+))$. Via the transformation $\Phi_{ij}$ we can define 
\[
\mathfrak{SM}_{\epsilon}^{-1}:=\Phi_{ji}\circ\mathcal{SM}_{\epsilon}^{-1}\circ\Phi_{ij}^{-1}
\]
on the new straight grids
\[
\mathfrak{N}_{ij}^{u,\kappa}=\Phi_{ij}(\Pi_{ij}^{u,\kappa}).
\]
For later use, we write down $D\mathcal{SM}^{-1}(x^+)$ here:
\[
\begin{pmatrix}
    1&0&0&0\\
    \frac{1}{\lambda}\ln\frac{\kappa\epsilon I^+}{\lambda}  & 1&\eta\beta &0 \\
   -\frac{\gamma}{\lambda}\ln\frac{\kappa\epsilon I^+}{\lambda}-\zeta  &-\gamma &1+\beta(\alpha-\eta\gamma) &-\alpha\\
     0 &0  & -\beta&1    \\
\end{pmatrix}+\mathbf{O}_1.
\]

Now we can prove the following stronger conditions:
\begin{lem}
\begin{equation}\label{shrink'}
\Big{|}\pi_3\circ\mathfrak{SM}_{\epsilon}^{-1}\mathfrak{N}_{ij}^{u,\kappa}\Big{|}\leq\mathcal{O}(\delta^2),
\end{equation}
\begin{equation}\label{expand'}
\Big{|}\pi_4\circ\mathfrak{SM}_{\epsilon}^{-1}\mathfrak{N}_{ij}^{u,\kappa}\Big{|}\geq\mathcal{O}(\delta)
\end{equation}
and
\begin{equation}\label{homotopy'}
\mathfrak{SM}_{\epsilon}^{-1}\partial^s\mathfrak{N}_{ij}^{u,\kappa}\simeq\partial^s\mathfrak{N}_{ij}^{u,\kappa},
\end{equation}
for any $(\eta,\xi)\in\mathbf{K}\times\mathbb{T}$. Here $\partial^s$ means the boundary of the 4-th component and $\simeq$ means homotopic equivalence.
\end{lem}
(\ref{shrink'}), (\ref{expand'}) and (\ref{homotopy'}) are sufficient to {\bf C1', C2'} and {\bf C3'}. 
\begin{proof}
$\forall ({\eta}^+,\xi^+,I^+,\tau^+)\in \Pi_{ij}^{u,\kappa}$, which corresponds to a unique point $(\eta_0^+,\xi_0^+,c,d)\in\mathfrak{N}_{ij}^{u,\kappa}$, the following holds:
\begin{equation}\label{1 jet}
\begin{aligned}
\mathfrak{SM}^{-1}(\eta_0^+,\xi_0^+,c,d)-&(\eta_0,\xi_0,0,0)=\Phi_{ji}[\mathcal{SM}^{-1}(\eta^+,\xi^+,I^+,\tau^+)-\\
&\quad\quad\quad\quad \mathcal{SM}^{-1}(\eta_0^+,\xi_0^+,I_{ji}(\eta_0^+,\xi_0^+),\tau_{ji}(\eta_0^+,\xi_0^+))]\\
&=\Phi_{ji}\int_{0}^{1}D\mathcal{SM}^{-1}(s\eta^++(1-s)\eta^+_0,s\xi^++(1-s)\xi^+_0,sI^++\\
&\quad\quad(1-s)I_{ij}(\eta^+_0,\xi^+_0), s\tau^++(1-s)\tau_{ij}(\eta^+_0,\xi^+_0))\cdot(cv_3^{ji}+dv_4^{ji})ds\\
&=\Phi_{ji}\Big{[}D\mathcal{SM}^{-1}(\eta_0^+,\xi_0^+,I_{ij}(\eta^+_0,\xi^+_0),\tau_{ij}(\eta^+_0,\xi^+_0))\cdot(cv_3^{ji}+dv_4^{ji})+\\
&\quad\quad\int_0^1\Upsilon(s)(cv_3^{ji}+dv_4^{ji})ds\Big{]}\\
&=\Phi_{ji}\Big{[}cv_3^{ji}/\lambda_3+dv_4^{ji}/\lambda_4+\int_0^1(\Upsilon'+\mathcal{O}(\kappa,\delta))(cv_3^{ji}+dv_4^{ji})ds\Big{]},
\end{aligned}
\end{equation}
where $\Upsilon(s)=\mathcal{SM}^{-1}\Big{|}_{0}^{s}$ is a variational matrix with $s\in(0,1)$. Formally it equals to
\[
\begin{pmatrix}
   0 & 0 & 0 & 0\\
   Osc(\frac{1}{\lambda}\ln I^+)   &0 & Osc(\eta\beta) &0 \\
   -Osc(\frac{\gamma}{\lambda}\ln\frac{\kappa\epsilon I^+}{\lambda}+\zeta) & -Osc(\gamma) & Osc(\beta(\alpha-\eta\gamma)) & -Osc(\alpha)\\

0& 0 & -Osc(\beta) & 0   \\
\end{pmatrix}
+\delta\mathbf{O}_1.
\]
 Recall that we take $\epsilon^{\varpi}\lesssim\delta\lesssim\epsilon^{\rho}$, $v_{3}^{ji}$ is $\delta-$parallel to $(0,\eta_0,0,-1)$ and $v_4^{ji}$ is $\delta^2-$parallel to $(0,\eta_0,\alpha(\eta_0,\xi_0)-\eta_0\gamma(\eta_0,\xi_0),-1)$, so $Osc(\alpha-\eta\gamma)\sim\mathcal{O}(\kappa_4,\delta^2)$ and $Osc(\eta)\sim\mathcal{O}(\kappa_3,\delta^2)$. Here the $\mathcal{O}(\kappa_i,\delta)$ implies the error term is dependent of $\kappa_i$, $i=3,4$. By removing the $\mathcal{O}(\delta)$ order error, we can simplify $\Upsilon$ by
\[
\Upsilon'=
\begin{pmatrix}
   0 & 0 & 0 & 0\\
   Osc(\frac{1}{\lambda}\ln I^+)   &0 & \eta Osc(\beta) &0 \\
   -Osc(\frac{\gamma}{\lambda}\ln\frac{\kappa\epsilon I^+}{\lambda}+\zeta) & 0 & (\alpha-\eta\gamma)Osc(\beta) & 0\\

0& 0 & -Osc(\beta) & 0   \\
\end{pmatrix}.
\]
Also we have the following observations: (1) both $(0,\eta_0,0,-1)$ and $(0,\eta_0,\alpha(\eta_0,\xi_0)-\eta_0\gamma(\eta_0,\xi_0),-1)$ have a degenerate first component; (2) For any vector $V$ linearly composed by $v_3^{ji}$ and $v_4^{ji}$, $\Upsilon\cdot V$ is $\delta-$parallel to $(0,\eta_0,\alpha(\eta_0,\xi_0)-\eta_0\gamma(\eta_0,\xi_0),-1)$. Besides, $|\Upsilon'v_3^{ji}|\sim\mathcal{O}(\delta)$.

Due to these observations, (\ref{shrink'}) and (\ref{expand'}) are almost obvious now:
\begin{equation}
\begin{aligned}
|\pi_4\circ\mathfrak{SM}^{-1}\mathfrak{N}_{ij}^{u,\kappa}| \geq & |d/\lambda_4|-\delta(|\pi_4\Upsilon' cv_3^{ji}|+|dv_4^{ji}|)\\
 \geq & C(\kappa_4)\delta-C(\kappa_3)\delta\geq C(\kappa_4)\delta/2,
\end{aligned}
\end{equation}

\begin{equation}\label{C2}
\begin{aligned}
|\pi_3\circ\mathfrak{SM}^{-1}\mathfrak{N}_{ij}^{u,\kappa}|  \leq & |c/\lambda_3|+|\Upsilon'dv_4^{ji}|+\mathcal{O}(\kappa_4,\delta^2)\\
 \leq & C(\kappa_3,\kappa_4)\delta^2,
\end{aligned}
\end{equation}
where $C(\kappa_i)$ is $\mathcal{O}(1)$ constants depending on $\kappa_i$, $i=3,4$,  so we can always take $\kappa_4$ properly greater than $\kappa_3$ such that previous inequalities hold.

As for the homotopy equivalence of (\ref{homotopy'}), we still lift $\mathcal{SM}^{-1}$ by $\overline{\mathcal{SM}}^{-1}$ in the covering space $(\eta,\xi,I,\tau)\in\mathbf{K}\times\mathbb{R}\times(-C\delta,C\delta)\times\mathbb{T}$ and the isolating blocks $\mathfrak{N}_{ij}^{u,\kappa}$ by $\overline{\mathfrak{N}}_{ij}^{u,\kappa}$. Then $\partial^u\overline{\mathfrak{N}}_{ij}^{s,\kappa}$ becomes single connecte, which corresponds to $|d|=\kappa_4\delta^2$. Take $|d|=\kappa_4\delta^2$ into (\ref{1 jet}) and get a slightly deformed new boundary $\mathfrak{SM}^{-1}\partial^s\mathfrak{N}_{ij}^{u,\kappa}$ which is also single connected. 
\end{proof}

The following Fig. \ref{fig:pen-iso} is a projection graph for the isolating blocks, which can give the readers a clear geometric explanation of previous Lemmas.

\begin{figure}[h]
  \def\svgwidth{6in}
  \centering
  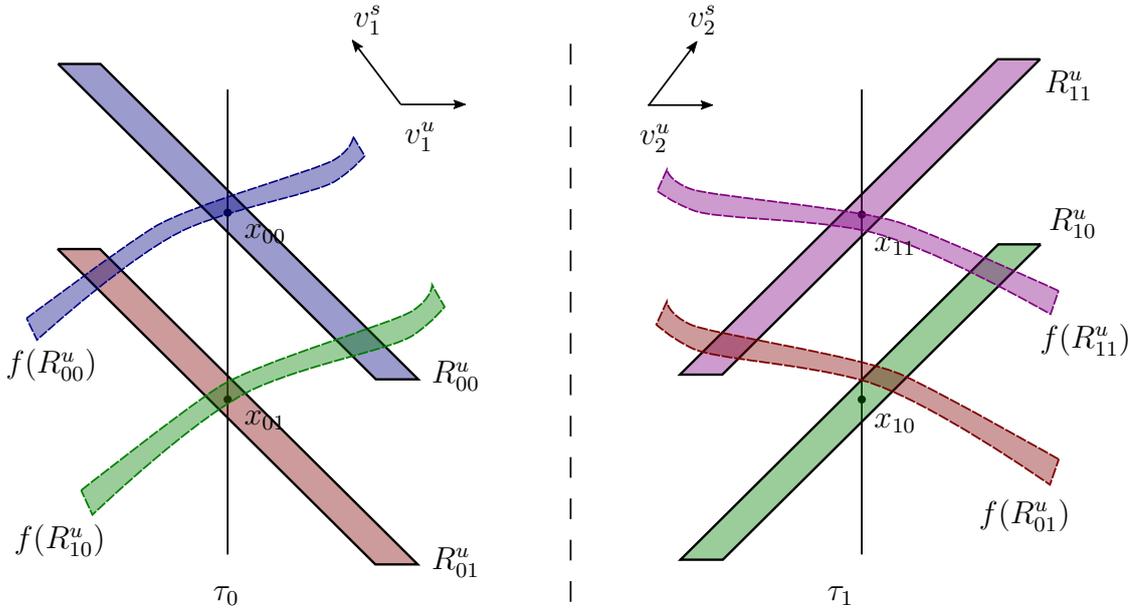
  \caption{Isolating blocks}
  \label{fig:pen-iso}
\end{figure}
\vspace{20pt}
From Appendix \ref{sec:suff-nhil} we can now get a topological invariant set in the intersectional parts of $\Pi_{ij}^{u,\kappa}\bigcap\Pi_{lk}^{s,\kappa}$, $i,j,k,l=0,1$, which is shown in Fig. \ref{fig:NHIL}. But we still need to prove the cone conditions for it, i.e. {\bf C4, C5}.\\

Recall that our invariant set lies in a domain $\mathbf{K}\times\mathbb{T}\times[-C\delta,C\delta]\times\mathbb{T}$, which is denoted by the original manifold $M$ and can be embedded into $\mathbb{R}^4$. On the other side, we can take a group of 
base vectors of $TM$ by
\[
E_1^c=(0,1,0,0)^t,
\]
\[
E_2^c=(1,0,0,0)^t,
\]
\[
E^u=(0,-\eta,0,1)^t,
\]
\[
E^s=(0,-\eta,-\alpha+\eta\gamma,1)^t.
\]
Notice that $TM=span\{E_1^c,E_2^c,E^u,E^s\}$, 
so every vector $v\in T_xM$ corresponds a unique coordinate $(a,b,c,d)\in\mathbb{R}^4$ such that
\[
v=a{E_1^c}{\mathcal{X}}+b{E_2^c}{\mathcal{X}}+c E^u+dE^s,
\]
where $\mathcal{X}$ is a rescale constant decided later on. Besides, we can take the following metric on $TM$:
\[
\|v\|_{\mathcal{X}}:=\|(a,b,c,d)\|_e,
\]
with $\|\cdot\|_e$ is the typical Euclid metric. Define the unstable cones in the bundle of isolating blocks $T_x\mathbb{R}^4\Big{|}_{x\in\Pi_{ij}^{s,\kappa}}$ by
\[
C_{ij}^u(x)=\{v\in T_xM: \ \measuredangle (v,E^u(x))\le \theta^u \},\quad i,j=0,1
\]
and on $T_x\mathbb{R}^4\Big{|}_{x\in\Pi_{ij}^{u,\kappa}}$ the stable cones 
\[
\qquad 
C_{ij}^s(x)=\{v\in T_xM: \ \measuredangle (v,E^s(x))\le  \theta^s \},\quad i,j=0,1.
\]

\begin{lem} \label{lem:cone}
For any $x\in \Pi^{s,\kappa}_{ij}$ and any $v\in C_{ij}^u(x)$ we have 
\[
D\SM_\eps(x) v \in C_{ji}^u(\SM_\eps(x)) \quad 
\text{ and }\quad\|D\SM_\eps (x)v\|_{\mathcal{X}}\ge \frac{m^u}{4\dt}\,\|v\|_{\mathcal{X}}
\]
Similarly, for any $x^+\in\Pi^{u,\kappa}_{ij}$ and 
any $v\in C_{ij}^s(x^+)$ we have 
\[
D\SM^{-1}_\eps(x^+) v \in C_{ji}^s(\SM^{-1}_\eps(x^+)) \quad 
\text{ and }\quad \|D\SM^{-1}_\eps (x^+)v\|_{\mathcal{X}}\ge \frac{m^s}{4\dt}\,\|v\|_{\mathcal{X}}
\]
with 
\[
\theta_u=\arctan\min\{\mathcal{O}(\frac{\mathcal{X}}{\gamma}),\mathcal{O}(\frac{1}{\delta}),\mathcal{O}(\frac{1}{\gamma\delta\mathcal{X}})\},
\]
\[
\theta_s=\arctan\min\{\mathcal{O}(\frac{\mathcal{X}}{\delta\ln\epsilon}),\mathcal{O}(\frac{1}{\delta}),\mathcal{O}(\frac{1}{\delta^2\ln\epsilon\mathcal{X}})\}
\]
for
\[
\mathcal{O}(\delta^2\ln\epsilon)\leq\mathcal{X}\leq\mathcal{O}(\frac{1}{\gamma},\frac{1}{\zeta})
\]
 and $m^{u,s}$ are $\mathcal{O}(1)$ constants depending on them.
\end{lem}
\begin{proof}
 $\forall x\in \Pi_{ij}^{s,\kappa}$, $C_{ij}^u(x)$ can be converted into
\begin{eqnarray*}
C_{ij}^u(x)=\Big{\{}v=a {E_1^c}{\mathcal{X}}+b {E_2^c}{\mathcal{X}}+cE^u(x)+dE^s(x)\Big{|}a^2+b^2+d^2\leq k_u^2c^2\Big{\}}
\end{eqnarray*}
with $\theta^u=\arctan k_u$. We should remind the readers $\mathcal{X}$ only influences the length of vectors in the cone but not the direction, so the angle $\theta^u$ keeps invariant.

Suppose $(a',b',c',d')$ is the coordinate of $D\mathcal{SM}(x)v$ of base vectors 
\[
{E_1^c(x^+)}{\mathcal{X}},{E_2^c(x^+)}{\mathcal{X}}, E^u(x^+), E^s(x^+),
\] 
then
\[
(a',b',c',d')^t=X^{-1}\cdot \Xi_{+}^{-1}\cdot D\mathcal{SM}(x)\cdot\Xi\cdot X(a,b,c,d)^t
\]
with $\Xi:=[E^c_1,E^c_2,E^u,E^s]_{4\times4}$ and
\[
X=\begin{pmatrix}
     \mathcal{X} &0 &0 &0    \\
  0    &       \mathcal{X} &0&0\\
  0&0&1&0\\
  0&0&0&1
\end{pmatrix}. 
\]
By calculation 
\[
\Xi_{+}^{-1}=
\begin{pmatrix}
    0 & 0 & -\frac{1}{\alpha^+-\eta^+\gamma^+} & 0  \\
     1 & 0 & 0 &0 \\ 
    0 & 0 & \frac{1}{\alpha^+-\eta^+\gamma^+} & 1 \\
     0   & 1 & 0 & \eta^+  
\end{pmatrix}
\]
and
\begin{equation}\label{differ-matrix}
\Xi_{+}^{-1}\cdot D\mathcal{SM}\cdot\Xi=
\begin{pmatrix}
   1   & -\frac{1}{\lambda}\ln\frac{\kappa\epsilon I^+}{\lambda}+\mathbf{O}_1/\delta &  \mathbf{O}_1/\delta & \mathbf{O}_1  \\
  \mathbf{O}_1   & 1 & \mathbf{O}_1  & \mathbf{O}_1 \\
   \beta\gamma+\frac{\gamma}{\alpha^+-\eta^+\gamma^+} & \beta\zeta+ \frac{\zeta}{\alpha^+-\eta^+\gamma^+} & 1+\beta(\alpha-\gamma\eta)+\frac{\alpha-\eta\gamma}{\alpha^+-\eta^+\gamma^+} & 1 \\
   -\frac{\gamma}{\alpha^+-\eta^+\gamma^+} & -\frac{\zeta}{\alpha^+-\eta^+\gamma^+} & -\frac{\alpha-\eta\gamma}{\alpha^+-\eta^+\gamma^+} & \mathbf{O}_1 
\end{pmatrix},
\end{equation}
so the rescaled matrix should be
{\small \[
X^{-1}\cdot \Xi_{+}^{-1}\cdot D\mathcal{SM}\cdot\Xi\cdot X=
\begin{pmatrix}
   1   & -\frac{1}{\lambda}\ln\frac{\kappa\epsilon I^+}{\lambda}+\mathbf{O}_1/\delta &  \mathbf{O}_1/\delta\mathcal{X} & \mathbf{O}_1/\mathcal{X}  \\
  \mathbf{O}_1   & 1 & \mathbf{O}_1/\mathcal{X}  & \mathbf{O}_1/\mathcal{X} \\
\mathcal{O}(\mathcal{X}/\delta) & \mathcal{O}(\mathcal{X}/\delta) & 1+\beta(\alpha-\eta\gamma)+\frac{\alpha-\eta\gamma}{\alpha^+-\eta^+\gamma^+} & 1 \\
\mathcal{O}(\mathcal{X}) & \mathcal{O}(\mathcal{X}) & -\frac{\alpha-\eta\gamma}{\alpha^+-\eta^+\gamma^+} & \mathbf{O}_1 
\end{pmatrix}.
\]}
An advantage of involving $\mathcal{X}$ is now the diagonal terms of aforementioned matrix are much greater than the rest.  
To make ${a}'^2+{b}'^2+{d}'^2\leq k_u^2{c}'^2$, we need  
\[
k_u^2\leq\min\{\mathcal{O}(\frac{\mathcal{X}}{\gamma}),\mathcal{O}(\frac{1}{\delta}),\mathcal{O}(\frac{1}{\gamma\delta\mathcal{X}})\}
\]
for
\[
\mathcal{O}(\gamma\delta^2,\zeta\delta^2)\leq\mathcal{X}\leq\mathcal{O}(\frac{1}{\gamma},\frac{1}{\zeta}).
\]
Recall that $\alpha-\eta\gamma\neq0$ and $\beta\sim\mathcal{O}(1/\delta)$ for any $x\in\Pi_{ij}^{u,s,\kappa}$, then we also get
\begin{eqnarray*}
\|D\mathcal{SM}(x)v\|_{\mathcal{X}}\geq|c'|&\geq& \frac{\beta(\alpha-\eta\gamma)}{2}|c|,\\
&\geq&\frac{\beta(\alpha-\eta\gamma)}{2}\frac{\|v\|_{\mathcal{X}}}{\sqrt{1+k_u^2}}.
\end{eqnarray*}
Taking $m^u={2(\alpha-\eta\gamma)}/{\sqrt{1+k_u^2}}$ we proved the first part.
\vspace{10pt}

Similarly,  $\forall x^+\in \Pi_{ij}^{u,\kappa}$, $C_{ij}^s(x^+)$ can be converted into
\begin{eqnarray*}
C_{ij}^s(x^+)=\Big{\{}v=a {E_1^c}{\mathcal{X}}+b {E_2^c}{\mathcal{X}}+cE^u(x^+)+dE^s(x^+)\Big{|}a^2+b^2+c^2\leq k_s^2d^2\Big{\}}
\end{eqnarray*}
with $\theta^s=\arctan k_s$. Also for this case $\mathcal{X}$ only influences the length of vectors in the cone but not the direction, so the angle $\theta^s$ keeps invariant.

Now we have
{\small
\begin{equation}\label{inverse matrix}
\begin{aligned}
\Xi^{-1}\cdot D\mathcal{SM}(x^+)^{-1}\cdot\Xi_+  &=[\Xi_+^{-1}\cdot D\mathcal{SM}(x)\cdot \Xi]^{-1}\\
& =\begin{pmatrix}
       1   & \frac{1}{\lambda}\ln\frac{\kappa\epsilon I^+}{\lambda} & 0 & 0 \\
           0  & 1 & 0 & 0 \\
-\frac{\gamma}{\alpha-\eta\gamma} & -\frac{1}{\alpha-\eta\gamma}(\frac{\gamma}{\lambda}\ln\frac{\kappa\epsilon I^+}{\lambda}+\zeta) & 0 &  -\frac{\alpha^+-\eta^+\gamma^+}{\alpha-\eta\gamma} \\
            \frac{\gamma}{\alpha-\eta\gamma}& \frac{1}{\alpha-\eta\gamma}(\frac{\gamma}{\lambda}\ln\frac{\kappa\epsilon I^+}{\lambda}+\zeta) & 1 &       (1+\alpha\beta-\eta\beta\gamma)\frac{\alpha^+-\eta^+\gamma^+}{\alpha-\eta\gamma} +1
\end{pmatrix}\\
&+\mathbf{O}_1/\delta^2.
\end{aligned}
\end{equation}}

Suppose $(a',b',c',d')$ is the coordinate of $D\mathcal{SM}^{-1}(x^+)v$, 
then
\[
(a',b',c',d')^t=X^{-1}\cdot \Xi^{-1}\cdot D\mathcal{SM}^{-1}(x^+)\cdot\Xi_+\cdot X(a,b,c,d)^t.
\]

To make ${a}'^2+{b}'^2+{c}'^2\leq k_s^2{d}'^2$, we need 

\[
k_s^2\leq\min\{\mathcal{O}(\frac{\mathcal{X}}{\delta\ln\epsilon}),\mathcal{O}(\frac{1}{\delta}),\mathcal{O}(\frac{1}{\delta^2\ln\epsilon\mathcal{X}})\}
\]
for
\[
\mathcal{O}(\delta^2\ln\epsilon)\leq\mathcal{X}\leq\mathcal{O}(\frac{1}{\delta\ln\epsilon}).
\]
Based on these
\begin{eqnarray*}
\|D\mathcal{SM}^{-1}(x^+)v\|_{\mathcal{X}}\geq|d'|&\geq& \frac{\beta(\alpha-\eta\gamma)}{2}|d|,\\
&\geq&\frac{\beta(\alpha-\eta\gamma)}{2}\frac{\|v\|_{\mathcal{X}}}{\sqrt{1+k_s^2}}.
\end{eqnarray*}
Taking $m^s={2(\alpha-\eta\gamma)}/{\sqrt{1+k_s^2}}$ we proved the second part.

\end{proof}
\begin{rmk}
{\bf C4, C5, C4', C5'} can be deduced from this Lemma with 
\[
\nu^{u,s}=\frac{m^{u,s}}{4\delta}
\]
and aforementioned $\theta^u$, $\theta^s$.
\end{rmk}

\subsection{$C^r$ smoothness and H\"older continuity of NHIL}
\label{sec:cr-regularity}

Based on previous analysis, we have proved {\bf C1 to C5} for the separatrix map, which lead to first part of Theorem \ref{thm:nhil}, i.e. we get two collections of Lipschitz graphs by $W^{uc}_{ij}$ and $W^{sc}_{ij}$, which corresponds to the invariant set in $\Pi_{ij}^{s,\kappa}\cap\Pi_{lk}^{u,\kappa}$, $i,j,l,k=0,1$ (see Fig. \ref{fig:NHIL}).  Actually, $W_{ij}^c:=W_{ij}^{uc}\pitchfork W_{ij}^{sc}$ is the normally hyperbolic invariant lamination and $W_{ij}^{uc}$, $W_{ij}^{sc}$ are the unstable, stable manifold of it. 

$\forall x\in W_{ij}^c$, $\{\mathcal{SM}_{\epsilon}^n(x)\}_{n\in\mathbb{Z}}$ will decide a unique bilateral sequence 
\[
\omega=(\omega_k),\quad k\in\mathbb{Z}, \;\omega_k\in\{0,1\},
\]
where $(\omega_k,\omega_{k+1})$ is the index of the isolating block where $\mathcal{SM}_{\epsilon}^k(x)$ lies. 

If we take the rescaled metric $\|\cdot\|_{\mathcal{X}}$ and base vectors $\{E_1(x)^c,E_2^c(x),E^u(x),E^s(x)\}$ on $T_xM\Big{|}_{x\in W_{ij}^c}$, we can get {\bf C6} with

\[
\lambda_{sc}^+=\ln\epsilon,\;\lambda_s^-\sim\mathcal{O}(1/\delta),
\]
\[
\lambda_{uc}^-\sim\mathcal{O}(\sqrt{1+a^2\mathcal{X}^2}\ln\epsilon),\;\lambda_u^+\sim\mathcal{O}(1/\delta),
\]
\[
m=\max\Big{\{}\sqrt{1+2\Big{(}\frac{\epsilon}{\delta^2\mathcal{X}}\Big{)}^2},1,a\mathcal{X}\ln\epsilon\Big{\}}.
\]
whereas $\delta\in[\epsilon^{\varpi},\epsilon^{\rho}]$ with $0<\rho<1/4$. 

Besides, the bundle $T\mathbb{R}^4$ restricted on it has a continuous splitting by 
$$
T_x\mathbb{R}^4\Big{|}_{x\in{W}_{ij}^c}=E_{ij}^u(x)\oplus E_{ij}^c(x)\oplus E_{ij}^s(x)
$$ 
and
\[
D\mathcal{SM}E_{ij}^*(x)//E_{jk}^*(\mathcal{SM}(x)),\;\forall x\in{W}^c_{ij},\;i,j,k\in\{0,1\},
\]
where $*$ can be any of $s,c,u$, $(i,j)=(\omega_0,\omega_1)$ and $(j,k)=(\omega_1,\omega_2)$. Notice that
$E_{ij}^{c,u,s}$ are different from aforementioned base vectors $\{E_1^c,E_2^c,E^u, E^s\}$ , but they still inherit the spectral estimate, i.e. the following inequalities hold:
\[
\max_{v\in E_{\omega}^c}\{\frac{||D\mathcal{SM}v||_{\mathcal{X}}}{||v||_{\mathcal{X}}},(\frac{||D\mathcal{SM}^{-1}v||_{\mathcal{X}}}{||v||_{\mathcal{X}}})\}\leq m,
\]
\[
\max_{v\in E_{\omega}^s}\{\frac{||D\mathcal{SM}v||_{\mathcal{X}}}{||v||_{\mathcal{X}}},(\frac{||D\mathcal{SM}^{-1}v||_{\mathcal{X}}}{||v||_{\mathcal{X}}})^{-1}\}\leq\frac{1}{\lambda_s^-}<1
\]
and
\[
\min_{v\in E_{\omega}^u}\{\frac{||D\mathcal{SM}v||_{\mathcal{X}}}{||v||_{\mathcal{X}}},(\frac{||D\mathcal{SM}^{-1}v||_{\mathcal{X}}}{||v||_{\mathcal{X}}})^{-1}\}\geq\lambda_u^+>1
\]
with $\lambda_s^-, \lambda_u^+\sim\mathcal{O}(1/\delta)$ and $m\leq\mathcal{O}(\ln\epsilon)$ due to {\bf C6}. \\

Now we make the following convention: recall that $\forall x\in W^c_{ij}$, there exists a corresponding bilateral sequence $\omega\in\Sigma$, conversely we can define a leaf of $W^c$ by
\[
\mathcal{L}_{\omega}=\{x\in W^c|\mathcal{SM}^n(x)\text{ corresponds to a fixed }\omega\in\Sigma, n\in\mathbb{Z}\}.
\]
So it's an one to one correspondence between $\omega\in\Sigma$ and $\mathcal{L}_{\omega}\subset W^c$. Besides, we can see that $\mathcal{L}_{\omega}$ is a collection of countably many 2-dimensional submanifolds, i.e.
\[
\mathcal{L}_{\omega}=\{(\eta,\xi,I_{\omega}(\eta,\xi,\epsilon),\tau_{\omega}(\eta,\xi,\epsilon))|\omega\in\Sigma, (\eta,\xi)\in\mathbf{K}\times\mathbb{T}\}.
\]
Usually, we can define the Bernoulli metric $\|\cdot\|_{\varrho}$ on $\Sigma$ by
\[
\|\omega-\omega'\|_{\varrho}:=\sum_{i\in\mathbb{Z}}\frac{|\omega_i-\omega'_i|}{\varrho^{|i+1|}},\quad\forall \omega=(\omega_i),\;\omega'=(\omega'_i),
\]
where $\varrho$ is a positive constant. In this article we take $\varrho=1/\delta$, and we explain why in the following.\\

From {\bf C1 to C5}, for any two bilateral sequences $\omega$ and $\omega'$ satisfying
$(\omega)_i=(\omega')_i$ for $-m\leq i\leq n$, $m,n\in\mathbb{N}$, 
\begin{equation}\label{norm left}
\|\pi_u(x'-x)\|\leq\mathcal{O}(\delta^{n+1})
\end{equation}
and
\begin{equation}\label{norm right}
\|\pi_s(x'-x)\|\leq\mathcal{O}(\delta^{m}),
\end{equation}
hold for any $x=(\eta,\xi,I_{\omega}(\eta,\xi),\tau_{\omega}(\eta,\xi))\in\mathcal{L}_{\omega}$, $x'=(\eta,\xi,I_{\omega'}(\eta,\xi),\tau_{\omega'}(\eta,\xi))\in\mathcal{L}_{\omega'}$. $\forall v\in T_xM$, $\pi_u v$, $\pi_s v$ is the unstable, stable component, i.e. if
\[
v=aE_1^c(x)+bE_2^c(x)+cE^u(x)+dE^s(x),
\]
$\pi_u v=cE^u(x)$ and $\pi_s v=dE^s(x)$.

On the other side, we can define the $\|\cdot\|_{C^r}$ norm on different leaves by:
\begin{equation}\label{norm leaf}
\|\mathcal{L}_{\omega}-\mathcal{L}_{\omega'}\|_{C^r}\doteq\min_{(\eta,\xi)\in\mathbf{K}\times\mathbb{T}}\|I_{\omega}(\eta,\xi)-I_{\omega'}(\eta,\xi)\|_{C^r}+\|\tau_{\omega}(\eta,\xi)-\tau_{\omega'}(\eta,\xi)\|_{C^r}.
\end{equation}
Recall that $(\eta,\xi)\in\mathbf{K}\times\mathbb{T}$ is compact, and 
\[
E^u(\eta,\xi,I,\tau)=(0,-\eta,0,1)^t,
\]
\[
E^s(\eta,\xi,I,\tau)=(0,-\eta,-\alpha+\eta\gamma,1)^t
\]
have a uniform angle away from zero due to Lemma \ref{angle}, so there exists an $\mathcal{O}(1)$ constant $C'$ such that 
\begin{equation}\label{norm symbol}
\frac{1}{C'}\|\omega-\omega'\|_{\delta}\leq\|\mathcal{L}_{\omega}-\mathcal{L}_{\omega'}\|_{C^0}\leq C'\|\omega-\omega'\|_{\delta},
\end{equation}
due to (\ref{norm left}), (\ref{norm right}) and (\ref{norm symbol}).

\begin{itemize}
\item {\bf the smoothness of each leaf's stable (unstable) manifolds:}
\end{itemize}

To fix the setting of Theorem \ref{cr}, we can take $X_0$ by $W^c_{\omega}$, $X$ by $\mathbf{K}\times\mathbb{T}\times(-C\delta,C\delta)\times\mathbb{T}$ and $f$ by $\mathcal{SM}_{\epsilon}$. We take the admissible metric by $||\cdot||_{\mathcal{X}}$. $\forall x\in W^c_{\omega}$, the exponential map will pull back $W^{uc}(x)$ into a unique backward invariant graph Lipschitz graph $g^{uc}_{inv}(x)\in\mathfrak{L}_{1/k_s}(E_{ij}^{uc}(x),E_{ij}^s(x))$, where $1/k_s$ is achieved due to Proposition \ref{pro} and {\bf C5}.  

Take $x$ on a certain leaf $\mathcal{L}_{\omega}$, i.e. $\omega$ is fixed, then $\rho_{uc}=m$, $\nu_{uc}^{-1}={m}/\lambda_s^-$ due to {\bf C6} condition. Then
\[
\rho_{uc}^r\nu_{uc}^{-1}\sim\mathcal{O}(\delta\cdot\ln\epsilon^{r+1}),
\]
 where the right side far less than $1$ for arbitrary $r>1$ as long as $\epsilon$ sufficiently small. Due to Theorem \ref{cr} we get the $C^r-$smoothness of the unstable manifold ${W}_{\omega}^{uc}(x)=\exp(g_{inv}^{uc}(x))$, $\forall x\in\mathcal{L}_{\omega}$.

Similarly, we can get the $C^r-$smoothness of the stable manifold ${W}_{\omega}^{sc}(x)=\exp(g_{inv}^{sc}(x))$ for $x\in\mathcal{L}_{\omega}$.

As $\mathcal{L}_{\omega}={W}_{\omega}^{sc}\pitchfork{W}_{\omega}^{uc}$, we get the $C^r-$smoothness of each leaf.
\begin{itemize}
\item \bf{H\"older continuity on $\omega\in\Sigma$:}
\end{itemize}

We take $f$ by $\mathcal{SM}_{\epsilon}$ and $\Lambda$ by $W^c_{\omega}$ in Theorem \ref{holder}, and $\lambda_+^{sc}\leq\ln\epsilon$, $\lambda_+^u\sim\mathcal{O}(1/\delta)$ due to {\bf C6}. Then $E_{ij}^{sc}(x)|_{x\in\Lambda}$ is H\"older with exponent $\varphi^{s,c}_1=\ln(1/\delta\ln\epsilon)/\ln(b_1/\ln\epsilon)$, which is greater than $1/(4^2+1)$ because $b_1\sim\mathcal{O}(1/\delta^{16})$. Similarly, we get $\lambda_-^{uc}\lesssim\ln\epsilon$ and $\lambda_-^s\sim\mathcal{O}(1/\delta)$. Then $E_{ij}^{uc}(x)|_{x\in\Lambda}$ is H\"older with exponent $\varphi^{u,c}_1=\ln(1/\delta\ln\epsilon)/\ln(1/\delta^{16}\ln\epsilon)$.

As $E_{ij}^c(x)|_{x\in\Lambda}=E_{ij}^{uc}(x)\cap E_{ij}^{sc}(x)|_{x\in\Lambda}$, we actually get the $\frac{1}{17}-$H\"older continuity of $E^c(x)$, i.e. 
\[
\|E_{ij}^c(x)-E_{ij}^c(y)\|\leq C_1\|x-y\|^{\varphi^{c}_1},\quad\forall x,y\in\Lambda,
\]
where the distance between two linear spaces is defined by
\[
dist(A,B)=\max\{\max_{v\in A,\;\|v\|=1}dist(v,B),\;\max_{w\in B,\;\|w\|=1}dist(w,A)\}.
\]
Recall that $\Lambda=W^c=\cup_{\omega\in\Sigma}\mathcal{L}_\omega=\{(\xi,\eta, t_{\omega}(\xi,\eta),I_{\omega}(\xi,\eta))|(\xi,\eta)\in\mathbb{T}\times\mathcal{D},\omega\in\Sigma\}$, for $x\in\Lambda$, 
\[
E_{ij}^c(x)=T_x\Lambda=span\{\partial_\xi\Lambda(x),\partial_\eta\Lambda(x)\}
\]
with 
\[
\partial_\xi\Lambda(x)=(1,0,\partial_\xi t_\omega(\xi,\eta),\partial_\xi I_\omega(\xi,\eta))
\]
and
\[
\partial_\eta\Lambda(x)=(0,1,\partial_\eta t_\omega(\xi,\eta),\partial_\eta I_\omega(\xi,\eta)).
\]
$\forall x,y\in\Lambda$ (unnecessarily on the same leaf) we have
\begin{eqnarray*}
\|\partial_\xi\Lambda(x)-\partial_\xi\Lambda(y)\|&\leq&\|\partial_\xi\Lambda(x)\|\cdot\|\frac{\partial_\xi\Lambda(x)}{|\partial_\xi\Lambda(x)|}-\frac{\partial_\xi\Lambda(y)}{|\partial_\xi\Lambda(x)|}\|\\
&\leq&\tilde{C}_1\|\partial_\xi\Lambda(x)\|\cdot \max_{v\in E_{ij}^c(x),\|v\|=1}dist(v,E_{ij}^c(y))\\
&\leq& \tilde{C}_1\|\partial_\xi\Lambda(x)\|dist(E_{ij}^c(x),E_{ij}^c(y))\\
&\leq& \tilde{C}_1\|\partial_\xi\Lambda(x)\|C_1\|x-y\|^{\varphi_1^c}\\
&\leq& C'_1\|x-y\|^{\varphi_1^c},
\end{eqnarray*}
where $\tilde{C}_1$ depends on $\max\{\|\partial_\xi\Lambda(x)\|,\|\partial_\xi\Lambda(y)\|\}$ and we can absorb it into $C'_1$. Use the same way we get 
\[
\|\partial_\eta\Lambda(x)-\partial_\eta\Lambda(y)\|\leq C'_1\|x-y\|^{\varphi_1^c}
\]
and then
\begin{eqnarray}
\|\mathcal{L}_{\omega}-\mathcal{L}_{\omega'}\|_{C^1}&\leq& 2C'_1 \|\mathcal{L}_{\omega}-\mathcal{L}_{\omega'}\|_{C^0}^{\varphi_1^c}\nonumber\\
&\leq&  2C'_1C'\|\omega-\omega'\|_{\delta}^{\varphi^{c}_1},
\end{eqnarray}
where $\omega,\omega'\in\Sigma$ uniquely decided by $x$,$y$.\\

In the following we remove the subscript `$ij$' for brevity. By induction of Theorem \ref{holder}, we get $E_{i+1}^c(x_i,v_i)|_{ \underbrace{T\cdots T }_i\Lambda}=E_{i+1}^{uc}(x_i,v_i)\cap E_{i+1}^{sc}(x_i,v_i)|_{\underbrace{T\cdots T }_i\Lambda}=\underbrace{T\cdots T}_{i+1}\Lambda$, $x_k=(x_{k-1},v_{k-1})$ and for $\|v_k\|\leq1$, $1\leq k\leq i$,
\[
\varphi_{i+1}^{sc}=(\ln\lambda_{i+1,+}^u-\ln\lambda_{i+1,+}^{s,c})/(\ln b_{i+1}-\ln\lambda_{i+1,+}^{s,c})>1/(i+3)^2
\]
and
\[
\varphi_{i+1}^{uc}=(\ln\lambda_{i+1,-}^s-\ln\lambda_{i+1,-}^{u,c})/(\ln b_{i+1}-\ln\lambda_{i+1,-}^{u,c})>1/(i+3)^2
\]
with $\lambda_{i+1,+}^{u}\sim\mathcal{O}(1/\delta)$, $\lambda_{i+1,+}^{sc}\lesssim\ln\epsilon$, $\lambda_{i+1,-}^s\sim\mathcal{O}(1/\delta)$ and $\lambda_{i+1,-}^{uc}\lesssim\ln\epsilon$ due to (\ref{matrix}). That means 
\begin{eqnarray*}
dist(E_{i+1}^c(x_{i},v_{i}),E_{i+1}^c(y_{i},w_{i}))&\leq& C_{i+1}(\|x_{i}-y_{i}\|^{\varphi^c_{i+1}}+\|v_{i}-w_{i}\|^{\varphi^c_{i+1}})\\
&\leq&C_{i+1}(\|x_{i}-y_{i}\|^{\varphi^c_{i+1}}+dist^{\varphi^c_{i+1}}(E_i^c(x_i), E_i^c(y_i)))\\
&\leq&C_{i+1}\|x_{i}-y_{i}\|^{\varphi^c_{i+1}}+C_{i+1}(C_i\|x_i-y_i\|^{\varphi_i^c})^{\varphi_{i+1}^c}\\
&\leq&C_{i+1}(\|x_{i-1}-y_{i-1}\|+\|v_{i-1}-w_{i-1}\|)^{\varphi_{i+1}^c}+\\
& & C_{i+1}C_i^{\varphi_{i+1}^c}(\|x_{i-1}-y_{i-1}\|+\|v_{i-1}-w_{i-1}\|)^{\varphi_{i+1}^c\varphi_{i}^c}\\
&\leq&\cdots\\
&\leq& \hat{C}_{i+1}\|x_1-y_1\|^{\prod_{k=1}^{i+1}\varphi_k^c},
\end{eqnarray*}
where $\hat{C}_{i+1}$ is a constant depending on $C_k$, $1\leq k\leq i+1$. \\

Recall that $E_{i+1}^c(x_i,v_i)=T_{(x_i,v_i)}(\underbrace{T\cdots T}_i\Lambda)$ and each leaf is sufficiently smooth,  $\partial_\xi\partial_\eta\Lambda=\partial_\eta\partial_\xi\Lambda$ holds.
\begin{eqnarray}\label{holder inequality}
\|\mathcal{L}_{\omega}-\mathcal{L}_{\omega'}\|_{C^{i+1}}
&=&\sum_{j=-1}^{i}\sum_{k=0}^{j+1}\|\partial_\xi^k\partial_\eta^{j+1-k}\Lambda(x)-\partial_\xi^k\partial_\eta^{j+1-k}\Lambda(y)\|_{C^0}\nonumber\\
&\leq&\sum_{j=-1}^{i}\sum_{k=0}^{j+1}|\partial_\xi^{k}\partial_\eta^{j+1-k}\Lambda(x)|\cdot\|\frac{\partial_\xi^{k}\partial_\eta^{j+1-k}\Lambda(x)}{|\partial_\xi^{k}\partial_\eta^{j+1-k}\Lambda(x)|}-\frac{\partial_\xi^{k}\partial_\eta^{j+1-k}\Lambda(y)}{|\partial_\xi^{k}\partial_\eta^{j+1-k}\Lambda(x)|}\|_{C^0}\nonumber\\
&\leq&\tilde{C}_{i+1}\sum_{j=-1}^{i}\sum_{k=0}^{j+1}|\partial_\xi^{k}\partial_\eta^{j+1-k}\Lambda(x)|\cdot \max_{v_{i+1}\in E_{i+1}^c(x_i),\atop \|v_{i+1}\|=1}dist(v_{i+1},E_{i+1}^c(y_i,w_i))\nonumber\\
&\leq& \tilde{C}_{i+1}\sum_{j=-1}^{i}\sum_{k=0}^{j+1}|\partial_\xi^{k}\partial_\eta^{j+1-k}\Lambda(x)|dist(E_{i+1}^c(x_i,v_i),E_{i+1}^c(y_i,w_i))\nonumber\\
&\leq& \tilde{C}_{i+1}\sum_{j=-1}^{i}\sum_{k=0}^{j+1}|\partial_\xi^{k}\partial_\eta^{j+1-k}\Lambda(x)|\hat{C}_{i+1}\|x_1-y_1\|^{\prod_{k=1}^{i+1}\varphi_k^c},\nonumber\\
&\leq& C'_{i+1}(\eta,\xi)\|\omega-\omega'\|^{\prod_{k=1}^{i+1}\varphi_k^c}.
\end{eqnarray}
where $x_1=x$ and $y_1=y$ have the same $\eta$ and $\xi$ components but belong to different leaves. We can always assume $|\partial_\xi^{k}\partial_\eta^{j+1-k}\Lambda(x)|>|\partial_\xi^{k}\partial_\eta^{j+1-k}\Lambda(y)|$, so the second line of aforementioned inequalities holds. If we specially take $(x_i,v_i)$ satisfying $v_k=0$ and $(y_i,w_i)$ satisfying $w_k=0$ for all $1\leq k\leq i$, the third line holds with $\tilde{C}_{i+1}$ depending on $x$, $y$. In the last line, we absorb all these constants and assume a new constant $C'_{i+1}$ which depends on $\eta$, $\xi$ and $i$. Recall that $(\eta,\xi)\in\mathbf{K}\times\mathbb{T}$ is compact, so $C'_{i+1}$ is of $\mathcal{O}(1)$ comparing to sufficiently small $\epsilon$.\\

Now we gather all these materials and construct the following commutative diagram:
\be
\begin{array}[c]{ccc} 
\quad\ {\bf \Lb}_\eps&
\quad  \stackrel{\SM_\eps} {\longrightarrow} & 
\quad {\bf \Lb}_\eps \\
C\ \uparrow &  & C\ \uparrow \\ 
\ \ \ \ | &  & \ \ \ \ | \\ 
\mathbb A_0 \times \Sigma
& \quad  \stackrel{F} {\longrightarrow} & \ \ 
\mathbb A_0 \times \Sigma
\end{array}
\ee
via
\be
\begin{array}[c]{ccc} 
\quad\ (\eta,\xi,I_{\omega}(\eta,\xi),\tau_{\omega}(\eta,\xi))&
\quad  \stackrel{\SM_\eps} {\longrightarrow} & 
\quad (\eta^+,\xi^+,I_{\sigma\omega}(\eta^+,\xi^+),\tau_{\sigma\omega}(\eta^+,\xi^+)) \\
C\ \uparrow &  & C\ \uparrow \\ 
\ \ \ \ | &  & \ \ \ \ | \\ 
\mathbb (\eta,\xi,\omega) 
& \quad  \stackrel{F} {\longrightarrow} & \ \ 
\mathbb (\eta^+,\xi^+,\sigma\omega)
\end{array},
\ee
where ${\bf \Lb}_\eps=W^c=\cup_{\omega\in\Sigma}\mathcal{L}_{\omega}$, $\mathbb{A}_0=\mathbf{K}\times\mathbb{T}$, $C$ is the standard projection and  $\sigma$ is the typical Bernoulli shift. As $C$ is a smooth diffeomorphism and ${\bf \Lb}_\eps$ is a collection of 2-dimensional graphs, $F$ is uniquely defined by $F=C^{-1}\circ\mathcal{SM}_{\epsilon}\circ C$. Actually, $F(\eta,\xi)=(\eta^+,\xi^+)$ should obey 
\begin{eqnarray*}
\eta^+&=&\eta - \eps M^+_\xi (\eta^+,\xi,\tau_{\om}(\eta,\xi))+
{\bf O}_2
\\
\xi^+&=&\xi +\epsilon M_{\eta^+}^+(\eta^+,\xi,\tau_{\omega}(\eta,\xi))
-\dfrac{\eta^+}{\lambda} \log 
\left| \dfrac{\kappa^s\,I_{\sigma\om}(\eta^+,\xi^+)}{\lambda}\right|
+{\bf O}_1
\end{eqnarray*}
due to Corollary \ref{separatrix-for-arnold-example}.

Now we can see aforementioned diffeomorphism is quite similar to (\ref{eq:separatrix-for-arnold-example-model}), but we need to deduce the dependence of $\omega$ further. \\

From (\ref{holder inequality}), we know 
\[
\|\mathcal{L}_{\omega}-\mathcal{L}_{\omega'}\|_{C^r}\leq C'_r\|\omega-\omega'\|^{\prod_{k=1}^{r}\varphi_k^c}.
\]
So $\forall\; r>1$ and $\hbar\gg1$, $\exists\; K_{r,\hbar}\in\mathbb{N}$ such that if $\omega$, $\omega'\in\Sigma$ with 
\[
\omega_i=\omega'_i,\quad\forall -K_{r,\hbar}\leq i\leq K_{r,\hbar},
\]
we have
\[
\|\mathcal{L}_{\omega}-\mathcal{L}_{\omega'}\|_{C^r}\ll\mathcal{O} (\epsilon^{\hbar}).
\]
 Recall that $\mathbf{K}\times\mathbb{T}$ is compact and $\mathcal{SM}_{\epsilon}$ is smooth enough, so
we can take a truncation by 
\[
[\omega]_{K_{r,\hbar}}=(\omega_{-K_{r,\hbar}},\cdots,\omega_0,\cdots,\omega_{K_{r,\hbar}}),\quad\forall \omega\in\Sigma
\]
and take $2^{2K_{r,\hbar}+1}$ leaves 
\[
\mathcal{L}_{\omega}=\{(\eta,\xi,I_{\omega}(\eta,\xi),\tau_{\omega}(\eta,\xi))|(\eta,\xi)\in\mathbb{A}_0 \}
\]
of different $[\omega]_{K_{r,\hbar}}$, such that (\ref{eq:separatrix-for-arnold-example-model}) holds.
\section{Derivation of a skew product model}
\label{sec:derivation-random-model}

In this section we derive a skew-product of cylinder maps 
model (\ref{eq:skew-shift-model}). It requires a second 
order expansion of the separatrix map from 
Theorem \ref{thm:separatrix-map} and 
a new ``concervative'' system of coordinates on each of 
cylinder leave.\\

In last section we get a skew product satisfying (\ref{eq:separatrix-for-arnold-example-model}). We rewrite it here for later use:
\begin{equation}\label{eq:skew-product}
\begin{aligned}
\eta^+ & = \eta &{}-  \epsilon M_\xi (\eta, \xi, \tau_{\om}) & &{} + 
{\bf O}_2  &,\\
\xi^+ & = \xi &{} +  \epsilon M_\eta (\eta, \xi, \tau_\om) & - \frac{\eta^+}{\lambda} \log \left(  \frac{\epsilon \kappa I^+_{\sigma\om}}{\lambda}  \right) &{} + {\bf O}_1 &,\\
I_{\sigma\om}^+ & = I_{\om}& {} -  \bigl(M_\tau - E'(\eta)M_\xi)(\eta, \xi, \tau\bigr) & &{} + \frac{1}{\epsilon}{\bf O}_2 &, \\
\tau_{\sigma\om}^+ & = \tau_{\om} &   + \frac{1}{\lambda} \log \left(  \frac{\epsilon \kappa I_{\sigma\om}^+}{\lambda}  \right) &\mod 2\pi &+\;  {\bf O}_1  &,
\end{aligned}
\end{equation}
where any leaf of the lamination can be expressed by 
\[
\mathcal{L}_{\omega}=\{(\eta,\xi,I_{\omega}(\eta,\xi,\epsilon),\tau_{\omega}(\eta,\xi,\epsilon))|(\eta,\xi)\in\mathbf{K}\times\mathbb{T}\}
\]
with $I_{\omega}$ and $\tau_\omega$ $C^r$ smooth, $r\geq12$. Besides, due to Lemma \ref{lem:fixed-centers} and Lemma \ref{lem:period-two-centers}, we know that the lamination lies in the $\mathcal{O}(\delta)$ neighborhood of the isolating centers, i.e. 
\begin{eqnarray}\label{expansion of epsilon}
I_{\omega}(\eta,\xi,\epsilon)&=&\delta I_{\omega_0\om_1}^1(\eta,\xi,a)+\delta^2 I_\omega^2(\eta,\xi,a)+O_1^{(1)}(\delta^3)\\
\tau_{\omega}(\eta,\xi,\epsilon)&=&\omega_0\pi+\tau_{\omega_0}^1(\eta,\xi,a)+a\delta\tau_{\om}^2(\eta,\xi)+{O}_1^{(1)}(a^2\delta^2)+\cdots\nonumber
\end{eqnarray}
hold with $\|I_\om^1\|_{C^1}$, $\|\tau_{\om_0}^1\|_{C^1}$ and $\|\tau_\om^2\|_{C^1}$ being $\mathcal{O}(a)-$bounded. Due to {\bf M1} condition $\tau_{\omega_0}^1$ ia non-degenerate, i.e. $\exists\; C_0,\;C_1>0$ such that $\|\tau^1_{\omega_0}\|_{C^0}>C_0$ and $\|\tau^1_{\omega_0}\|_{C^1}>C_1$. Recall that $\epsilon^{\varpi}<\delta<\epsilon^{\rho}$ with $1\geq\varpi>1/4>\rho>0$, so we can take $\varpi=1$ and let $\delta\sim\mathcal{O}(\epsilon)$ throughout this section. \\

Another observation is the following: As $\mathcal{SM}_{\epsilon}$ is an exact symplectic diffeomorphism (see Remark \ref{exact symplectic}), $d\eta\wedge d\xi+dh\wedge d\tau$ is invariant and we can pull it back onto the NHIL  and get an area form $d\mu_{\omega}(\eta,\xi,\omega)=\rho_\omega(\eta,\xi,a)d\eta\wedge d\xi$. Actually, we have
\begin{equation}\label{area form}
\rho_\om(\eta,\xi)=\Big{(}1+\epsilon\Big{|}\frac{\partial(I_\om,\tau_\om)}{\partial(\eta,\xi)}\Big{|}+\eta\frac{\partial\tau_\om}{\partial\xi}\Big{)}d\eta\wedge d\xi,\quad\om\in\Sigma.
\end{equation}

So if we take $0<\epsilon\ll a\ll1$, $\rho_\omega-1\sim\mathcal{O}(a)$ is uniformly bounded because $\|I_{\om}\|_{C^1}$ and $\|\tau_{\om}\|_{C^1}$ is $\mathcal{O}(a)$ bounded due to the cone condition. For this we just need to take $\mathcal{X}\sim\mathcal{O}(1)$ in Lemma \ref{lem:cone} and the corresponding $1/k_u\leq\mathcal{O}(a)$ and $1/k_s\leq\mathcal{O}(\delta\log\eps)$. Later in section \ref{sec:symplectic-coord} we will further transform (\ref{eq:separatrix-for-arnold-example-model}) into the form (\ref{eq:skew-shift-model}) with the standard symplectic 2-form $dr\wedge d\theta$ on it (see section \ref{sec:symplectic-coord}), which totally fits into the framework of \cite{CK}.

\subsection{The second order of the separatrix map for 
trigonometric perturbations in the single resonance regime}
\label{sec:separatrix-map-SR}

Here we give formulas for the separatrix map for 
the trigonometric perturbations expanded to the second 
order. They are obtained in \cite{GKZ}. 

First fix some notation. Take a function
$f:\T^n\times\R^n\times\R^2\times\T\longrightarrow \R$ 
with Fourier series 
\[
f=\sum_{k\in\Z^{n+1}}\,f^k(I,p,q)e^{2\pi i k\cdot (\varphi,t)}.
\]
Define $\mathcal N$ as 
\[
 \mathcal N(f)=\{k\in\Z^{n+1}: f^k\neq 0\}
\]
and 
\[
 \mathcal N^{(2)}(f)=\{k\in\Z^{n+1}: k=k_1+k_2,\ 
k_1,k_2\in \mathcal N(f)\}.
\]

Consider the \emph{non-resonant region}, which 
stays away from resonances created by 
the harmonics in $\mathcal N^{(2)}( H_1)$.

Define
\begin{equation}\label{def:SR}
 \text{Non}_\beta=\left\{I:  \forall k\in \mathcal N^{(2)}( H_1),\,\,  
|k\cdot (\nu(I),1)|\geq\beta\right\}.
\end{equation}
for a fixed parameter $\beta$. The complement of 
the non-resonant zone is build up by the different resonant 
zones associated to the harmonics in $\mathcal N^{(2)}( H_1)$. 
Fix $k\in\mathcal N^{(2)}( H_1)$, then we define the 
resonant zone
\begin{equation}\label{def:DR}
 \text{Res}^k_\beta=\left\{I: |k\cdot
(\nu(I),1)|\leq\beta\right\}.
\end{equation}
The parameter $\beta$ in both regions will be chosen 
differently, so that the different zones overlap.

We abuse notation and we redefine the norms  in
\eqref{anisotropic-norm} as
\[
\|\cdot \|_r^*=\|\cdot \|_r^{(\beta)},  \ \ \ 
O^{(b)} =O^{(b)}_1, \ \ \  O^*_k=O^{(\beta)}_k.
\] 
Now we can give formulas for the separatrix map in 
both regions.

The main result of this section is Theorem 
\ref{thm:SM-higher-order:SR} 
which gives refined formulas for the separatrix map in 
the single resonance zone (see \eqref{def:SR}). 
To state it we need to define an auxiliary function $w$. 
This $w$ is a slight modification of the function $w_0$ 
given in ( \ref{dist-to-separatrix}).

Consider a function $g(\eta, r)$. It is obtained 
in Section 4.1\cite{GKZ} by applying Moser's normal 
form to $H_0$. This $g$ satisfies 
$g(\eta,r)=\lb(\eta) r+\mathcal O(r^2)$, where $\lb$ is
the positive eigenvalue of the matrix \eqref{def:MatrixLambda}.
Therefore, $g$ is invertible with respect to the second variable 
for small $r$. Somewhat abusing notation, call $g_r^{-1}$ 
the inverse of $g$ with respect to the second variable
\footnote{the subindex is to emphasize that the inverse is 
performed with respect to the variable $r$}. Then define 
the function $w$  by
\begin{equation}\label{def:omega}
 w(\eta,h)=g_r^{-1} (\eta, h-E(\eta)).
\end{equation}

\begin{thm} \label{thm:SM-higher-order:SR}
Fix $\beta>0$ and $1\ge a>0$. For $\eps$ sufficiently small
there exist $c>0$ independent of $\eps$ and a canonical 
system of coordinates $(\eta,\xi,h,\tau)$ such that in 
the non-resonant zone $\textrm{Non}_\beta$ we have 
\[
\eta = I +\mathcal O^*_{1}(\eps)+\mathcal O^*_{2}(H_0-E(I)), \ 
\xi + \nu(\eta)\tau=\varphi+f, \ 
h=H_0+\mathcal O^*_{1}(\eps)+\mathcal O^*_{2}(H_0-E(I)),
\]
where $f$ denotes a function depending only on 
$(I,p,q,\eps)$ and such that $f(I,0,0,0)=0$ and 
$f=\mathcal O(w^\sg_0+\eps)$. In these coordinates 
the separatrix map has the following form. For any 
$\sigma\in \{-,+\}$ and $(\eta^*,h^*)$ such that 
\[
 c^{-1}\eps^{1+a}<|w(\eta^*,h^*)|<c\eps,\qquad |\tau|<c^{-1}, \qquad
c<|w(\eta^*,h^*)|\,e^{\lb(\eta^*)\bar t}<c^{-1},
\]
the separatrix map $(\eta^*, \xi^*,h^*,\tau^*)=
\SM(\eta, \xi,h,\tau)$ is defined implicitly as follows
 \[
  \begin{split}
 \eta^*=&\ \eta- \ \ \eps
M_1^{\sigma,\eta}+\ \ \eps^2 M_2^{\sigma,\eta}+\ 
\ \mathcal O_3^*(\eps,|w|)|\log|w||\\
   \xi^*=&\ \xi+\ \partial_1 \Phi^\sg (\eta,w (\eta^*,h^*))
+\partial_{\eta^*} w(\eta^*,h^*) 
\left[\log |w(\eta^*,h^*)| + \partial_2 \Phi^\sg (\eta^*, w (\eta^*,h^*))
\right] 
\\
&+
\mathcal O_1^*(\eps+|w|)\,( |\log\eps|+|\log|w||)\\
 h^*=&\ h - \ \ \eps M_1^{\sigma,\tau}
+\ \ \eps^2 M_2^{\sigma,h}+\ \ \mathcal O_3^*(\eps,|w|)\\
 \tau^*=&\ \tau+\ \qquad \qquad  \bar t +\qquad \quad 
\partial_{h^*} w(\eta^*,h^*) 
\left[\log| w(\eta^*,h^*) |+ \partial_2 \Phi^\sg (\eta^*, w (\eta^*,h^*))
\right] 
\\
&+ \mathcal O_1^*(\eps+|w|)\,( |\log\eps|+|\log|w||),
  \end{split}
 \]
where $w$ is the function defined in \eqref{def:omega}, 
$M^*_i$ and $\Phi^\pm$ are $C^2$ functions and 
$\bar t$ is an integer satisfying 
\be \label{eq:time-interval}
 \left|\tau+\bar t+\frac{\partial_{h^*} w_0^\sigma}{\lb}
\log \left|\frac{\kappa^\sigma w_0^\sigma}{\lb}\right|\right|<c^{-1}
\ee
The functions $M_i^*$ 
are evaluated at  $(\eta^*, \xi,h^*,\tau)$. 
\end{thm}
Corollary \ref{separatrix-for-arnold-example} is a special case of this theorem with $\mu^{\sigma}$, $\kappa^\sigma$, $\lambda$ and $\nu$ independent of $\eta$. And the function $g$ is also independent of $\eta$.
\begin{rmk}\label{rm:on-SM-higher-order}
The change of coordinates in the above Theorem is 
$\eps$-close (in the $C^2$-norm) to the system of 
coordinates obtained in Theorem \ref{thm:separatrix-map}.

The functions $\Phi^\sigma$ are the generalizations of 
the functions $\mu^\sigma$ and $\kappa^\sigma$. Indeed, 
they satisfy
\[
\partial_\eta\Phi^\sigma(\eta,r)=\mu^\sigma(\eta)+\mathcal O^*_2(r)\,
\text{ and
}\,e^{\partial_r \Phi^\sigma(\eta,r)}=\kappa^\sigma(\eta)+\mathcal O^*_2(r). 
\]
Moreover, the functions $M_i^{\sigma, i}$ satisfy
\[
 M_1^{\sigma, 1}=\partial_\xi M^\sigma +\mathcal O^*_2(w),
\qquad M_1^{\sigma, 2}=\partial_\tau M^\sigma+\mathcal O^*_2(w), 
\]
where $M^\sigma$ is the (Melnikov) split potential 
given  in Proposition \ref{poincare-melnikov}. 

\end{rmk}

\subsection{Conservative structure and normalized 
coordinates for the skew-shift}\label{sec:symplectic-coord}

Arguments in this section are important for the proof and 
arose from envigorating discussion with L. Polterovich in 
Minneapolis in November 2014. 

Consider a normally hyperbolic lamination consisting of  
cylinderic leaves 
$$
C:\mathcal D_0\times \T \times \Sigma \to
\mathcal D \times \T \times \R \times \T
$$
$$
C(\eta,\xi,\omega)=(\eta,\xi,h(\eta,\xi,\om),\tau(\eta,\xi,\om)),
$$
where $h(\eta,\xi,\omega)=\epsilon I(\eta,\xi,\omega)+\eta^2/2$. Consider the area form $d\mu(\eta,\xi,\om)$ on a leave 
(the cylinder) $C(\mathcal D_0\times \T,\om)$
induced by the canonical form 
$$
\om=d\eta\wedge d\xi+dh\wedge d\tau.
$$ 
Denote by 
$$
d\mu_\om (\eta,\xi,\om)= \rho_\om(\eta,\xi)d\eta \wedge d\xi. 
$$
the corresponding density of this measure, which is 
$C^r$ smooth. Recall that $\rho_\om$ satisfies (\ref{area form}). Since each leave (cylinder) is a graph 
over $(\eta,\xi)$-component and $(\eta,\xi)$ are conjugate 
variables, this restriction is nondegenerate. 

\begin{lem} \label{lem:canonical-coord}
There is a map 
$$
\cM:  \mathcal D_0\times \T \times \Sigma
\to \R \times \T \times \Sigma 
$$
$$
\cM(\eta,\xi,\om) \to (r,\theta,\om) 
$$
\begin{eqnarray*}
\cM(\eta,\xi,\om)&=&(\cM^r(\eta,\om),
\cM^\theta(\eta,\xi,\om),\om)\\
&=&(\cN(\eta,1,\om),
\frac{\cN_\eta(\eta,\xi,\om)}{\cN_\eta(\eta,1,\om)},\om)
\end{eqnarray*}
such that for each $\om \in \Sigma$  the induced 
area-form 
$$
d_{(\eta,\xi)}\cM^* \,d\mu_\om(\eta,\xi)=
dr \wedge d\theta.
$$
Moreover, for each $\om \in \Sigma$ the $r$-component 
of this map satisfies 
$$
\cN(\eta,1,\om)=R_\om(\eta)
$$
for some family of smooth strictly monotone 
functions $R_\om(\cdot)$. 
\end{lem}

\begin{proof}
Fix $\om \in \Sigma$. Let $\cN_\xi(\eta,0,\om)=0$. 
Let 
$$
S(\eta,\xi,\eps,\om)=\{(\eta',\xi',h(\eta',\xi',\om),\tau(\eta',\xi',\om)):
\eta'\in (\eta,\eta+\eps),\ \xi' \in (0,\xi)\}.
$$
Define the $\mu$-area
$$
A(\eta,\xi,\eps,\om):=\mu(S(\eta,\xi,\eps,\om)).
$$
Define 
$$
A(\eta,\xi,\om):= \lim_{\eps \to 0}
\dfrac{A(\eta,\xi,\eps,\om)}{\eps}. 
$$
Fix $\eta>0$ too. Define 
$$
\cN_\eta(\eta,\xi,\om):= A(\eta,\xi,\om).
$$
For $\eta<0$ one can give a similar definition. 
\end{proof}
\begin{rmk}\label{r,theta}
Actually, previous formal transformation can be explicitly evaluated by
\[
r=\int_0^\eta\int_0^{1}\rho_\om(\vartheta,\xi)d\xi d\vartheta,\quad\theta=\frac{1}{ r_\eta}\int_0^\xi\rho_\om(\eta,\zeta) d\zeta.
\]
On the other side, $\rho_\om$ obeys (\ref{area form}), so $| r_\eta-1|\leq\mathcal{O}(\epsilon)$.
\end{rmk}
\begin{lem} \label{lem:skew-shift-normal-form}
Let $\mathcal F:\R \times \T \times \Sigma
\to \R \times \T \times \Sigma$ be a skew shift 
$$
\cF:(r,\theta,\om) \to (f_\om(r,\theta),\sigma \om)
$$
such that the following diagram commutes 
\be \label{eq:skew-shift-normal-form}
\begin{array}[c]{ccc} 
\mathcal D_0 \times \T \times \Sigma
& \quad  \stackrel{F} {\longrightarrow} & \ \ 
\mathcal \R \times \T \times \Sigma \\
\mathcal N\ \uparrow &  & \mathcal N\ \uparrow \\ 
\ \ \ \ \, | &  & \ \ \ \ \,| \\ 
\R \times \T \times \Sigma
& \quad  \stackrel{\mathcal F} {\longrightarrow} & \ \ 
\R \times \T \times \Sigma
\end{array}
\ee
then $\cF$ has the following form
\be \label{eq:skew-shift-normal-form} 
\beal 
r^*         & = & r + \eps M^{[\om]_{k+1}}_1(\theta,r)+ 
\eps^2 M^{[\om]_{k+1}}_2(\theta,r)+\mathcal O (\eps^3 )|\log \eps|
\\
\theta^*&=& \theta + 
\frac{\mathcal R(r)}{\lb } (\log  \eps \dt +\log \kappa^\sg \lb^{-1})
+\mathcal O(\eps\log \eps),\quad 
\enal
\ee
where $\mathcal R$ is a smooth strictly monotone function, 
$\om_0=i$ or $1$, $\om=(\dots,\om_0,\dots)\in \{0,1\}^\Z$ and $[\omega]_{k+1}$ is the $(k+1)-$truncation introduced in Corollary \ref{separatrix-for-arnold-example}.  
\end{lem}

Denote $\Delta=(\log  \eps \dt +\log \kappa^\sg \lb^{-1})/\lb$. 
Recall that both $\kappa^sg$ and  $\lb$ are constants 
for Arnold's example. Notice also that we study the regime 
$\dt\in (\eps^{1/4},\eps^1)$. 
Therefore, $\Delta \sim \log \eps$. 
Let $R:=\Delta \cdot \mathcal R(r)$ and 
$\overline {\mathcal R}:R \to r$ be the inverse map, i.e. 
$\overline {\mathcal R}( \mathcal R(r))\equiv r.$ 

\begin{cor}\label{cor:random-model}
Let $\Phi:(\theta,r)\mapsto (\theta,\mathcal R(r))$ be a smooth 
diffeomorphism and $\Phi \circ \mathcal F \circ \Phi^{-1}$
be the map $\mathcal F$ written in $(\theta,R)$-coordinates.
Then it has the following form  
\be \label{eq:skew-shift-normal-form-normalized} 
\beal 
R^*         & = & R + \eps \Delta \widetilde M^{[\om]_{k+1}}_1
(\theta,R/\Delta)&+ \eps^2 \Delta  \widetilde M^{[\om]_{k+1}}_2
(\theta,R/\Delta)  
+ \mathcal O (\eps^3)|\log \eps|
\\
\theta^*&=&\theta + R +\mathcal O(\eps\log \eps).&
\enal
\ee
where $\om_0=0$ or $1$, and 
$\om=(\dots,\om_0,\dots)\in \{0,1\}^\Z,$ and 
$ \widetilde M^{[\om]_{k+1}}_i,\ i=1,2$ are smooth functions.  
\end{cor}
\begin{rmk} 
Let $\eps' = \eps \log \eps$. Notice that 
$\eps \Delta \sim \eps'$. Non-homogenenous random walks 
with step $\sim \eps'$, generically, have a drift of order 
$\eps'=\eps^2\log^2\eps\gg \eps^2 \Delta$. Therefore, 
the dominant contribution to diffusive behaviour comes from 
the term $\eps \Delta \widetilde M^{[\om]_{k+1}}_1$. 

During this diffeomorphism the Bernoulli shift $\sigma$ will be involved, that's why in Corollary \ref{cor:random-model} the function $ \widetilde M^{[\om]_{k+1}}_2$ depends on $[\om]_{k+1}$. So sill finitely many cases should be considered. 
\end{rmk}
Before we prove Lemma \ref{lem:skew-shift-normal-form} we 
derive this Corollary.

\begin{proof} Consider the direct substitution $R^*=\mathcal R(r^*)$.
Apply Taylor formula of order $2$ and get 
\be \nonumber 
\beal 
R^* & = & \Delta \mathcal R(r) + 
\eps \Delta \mathcal R'(r) M^{[\om]_{k+1}}_1(\theta,r)+ \\ 
&&
\eps^2 \Delta 
\left(\mathcal R'(r) M^{[\om]_{k+1}}_2(\theta,r)+ \frac 12 \mathcal R''(r)
(M^{[\om]_{k+1}}_1)^2(\theta,r) \right) & 
+ \mathcal O (\eps^3 )|\log \eps|
\\
\theta^*&=& \theta + R
+\mathcal O(\eps\log \eps).\qquad \qquad \qquad &
\enal
\ee
Notice that $r=\overline {\mathcal R}(R/\Delta)$ is a smooth function of $R$.
Therefore, substituting instead of $r$ and using 
$R=\Delta \mathcal R(r)$ we obtain the required expression.  
\end{proof}

\begin{proof}
As we have improved the separatrix with second order estimate, we can apply (\ref{modified SM}) for the NHIL. 
So we get the improved skew product by:
\begin{eqnarray}
 \eta^+=&\ \eta- \ \ \eps
\partial_{\xi}M^{\sigma}(\eta^+,\xi,\tau_{[\omega]_k})+\ \ \eps^2 M_2^{\sigma,\eta}(\eta^+,\xi,h^+,\tau_{[\omega]_k})+\ 
\ \mathcal O_3^*(\eps)|\log\epsilon|\nonumber\\
   \xi^+=&\ \xi+\epsilon\partial_{\eta^+}M^{\sigma}(\eta^+,\xi,\tau_{[\omega]_k})
-\frac{\eta^+}{\lambda}
\log \Big{|}\frac{\kappa^{\sigma}(h^+-E(\eta^+))}{\lambda}\Big{|}+
\mathcal O_1^*(\eps)|\log\eps|.
\end{eqnarray}
Recall that $\delta\sim\mathcal{O}(\epsilon)$, $h=\epsilon I+\eta^2/2$ and (\ref{expansion of epsilon}) holds, 
\begin{eqnarray*}
 \eta^+&=&\eta- \eps
M_\xi^\sigma(\eta - \eps M_\xi^\sigma(\eta,\xi,\tau_{[\omega]_k}),\xi,\tau_{[\omega]_k})\\
& &+\eps^2 M_2^{\sigma,\eta}(\eta,\xi,\epsilon I_{[\sigma\omega]_k}(\eta^+,\xi^+)+\eta^{+2}/2,\tau_{[\omega]_k})+ 
\mathcal O_3^*(\eps)|\log \eps|\\
&=&\eta- \eps
M_\xi^\sigma(\eta ,\xi,\omega_0\pi+a\tau_{\omega_0}^1(\eta,\xi))+\eps^2M_{\xi\eta}^\sigma(\eta ,\xi,\omega_0\pi+a\tau_{\omega_0}^1(\eta,\xi))\cdot\\
& &M_\xi^\sigma(\eta ,\xi,\omega_0\pi+a\tau_{\omega_0}^1(\eta,\xi))-\eps a\delta M_{\xi\tau}^\sigma(\eta ,\xi,\omega_0\pi+a\tau_{\omega_0}^1(\eta,\xi))\cdot\tau_{[\omega]_k}^2(\eta,\xi)\\
& &+\eps^2 M_2^{\sigma,\eta}(\eta,\xi, \eta^2/2,\omega_0\pi+a\tau_{\omega_0}^1(\eta,\xi))+\mathcal O_3^*(\eps)|\log \eps|.
\end{eqnarray*}
If we assume $M_1$ and $M_2$ by the $\mathcal{O}(\eps)$ and $\mathcal{O}(\eps^2)$ functions, formally we can get 
\[
\eta^+=\eta+\eps M_1^{[\omega]_k}(\eta,\xi)+\eps^2M_2^{[\omega]_{k}}(\eta,\xi)+\mathcal O_3^*(\eps)|\log \eps|.
\]
The angular component $\xi$ satisfies 
\begin{eqnarray}\label{cocycle form}
\xi^+&=& \xi+\eps M_{\eta^+}^\sigma(\eta,\xi,\omega_0\pi+a\tau_{\omega_0}^1(\eta,\xi))-\frac{\eta}{\lb}
\log \left| \frac{\kappa^\sigma\eps I_{[\sigma\omega]_k}(\eta^+,\xi^+)}{\lb} \right|
+\mathcal O_1^*(\eps)|\log\eps|\nonumber\\
&=&\xi+\frac{\eta}{\lambda}\log\frac{\kappa\epsilon\delta}{\lambda}-\Big{(}\eta^+\tau_{[\sigma\om]_k}^+-\eta\tau_{[\om]_k}\Big{)}+\mathcal O_1^*(\eps)|\log\eps|\\
&=&\xi+\frac{\eta}{\lambda}\log\frac{\kappa\epsilon\delta}{\lambda}-\Big{(}\eta^+\tau^1_{\om_1}(\eta^+,\xi^+,a)-\eta\tau^1_{\om_0}(\eta,\xi,a)\Big{)}+\mathcal O_1^*(\eps)|\log\eps|\nonumber
\end{eqnarray}
where $\eps M_{\eta^+}^\sigma(\eta,\xi,\omega_0\pi+a\tau_{\omega_0}^1(\eta,\xi))$ is an invalid term and can be absorbed into the reminder, and the term within the brackets of the second line is due to (\ref{eq:skew-product}). Recall that due to the cone condition, $\|\eta^+\tau^1_{\om_1}(\eta^+,\xi^+,a)-\eta\tau^1_{\om_0}(\eta,\xi,a)\|_{C^1}\leq\mathcal{O}(a)$ for all $(\eta,\xi)\in\mathbf{K}\times\mathbb{T}$.

In Remark \ref{rm:on-SM-higher-order} we state that that 
change of coordinate from Theorem \ref{thm:separatrix-map} 
to Theorem \ref{thm:SM-higher-order:SR} is 
$\mathcal O(\eps)$-close to the identity. Therefore, the bound on 
the error terms of the $\xi$-component stays  unchanged. Besides, by taking $\tilde{\xi}=\xi+\eta\tau_{\om_0}^1$ we can simplify $\xi-$equation into:
\begin{equation}\label{cocycle}
\tilde{\xi}^+=\tilde{\xi}+\frac{\eta}{\lambda}\log\frac{\kappa\epsilon\delta}{\lambda}+\mathcal O_1^*(\eps)|\log\eps|
\end{equation}
which is independent of $\om\in\Sigma$ by removing the reminder. Obviously $\tilde{\xi}(\xi+1)=\tilde{\xi}(\xi)+1$ and the transformation $(\eta,\xi)\rightarrow(\eta,\tilde{\xi})$ is nondegenerate by taking $a$ properly small. Recall that any leaf of the lamination is invariant, so the approximate rotation number of the $\tilde{\xi}$-component only depends on the term 
\[
\frac{\eta}{\lambda}\log\frac{\kappa\epsilon\delta}{\lambda}+\mathcal O_1^*(\eps)|\log\eps|,
\] 
which is independent of $\om\in\Sigma$ except the reminder.

Now we transform this map 
$F|_{(\eta,\xi,\omega)\in \mathbf{K}\times\mathbb{T}\times\Sigma}$ 
into the $(r,\theta)$-coordinates. 
By Lemma \ref{lem:canonical-coord} the map 
$\mathcal M(\eta,\xi,\om)=(r,\theta)$ has 
the $r$-component being a function of $\eta$ 
only, which we denote by $r=R_{\om}(\eta)$. 
Thus, for the new action variable $r$ we have 
\[
r^*=r+ \eps N_1(r,\theta,[\om]_{k+1})+\eps^2 N_2(r,\theta,[\om]_{k+1})+ 
\mathcal O_3^*(\eps)|\log \eps|,
\]
where $N_1$ and $N_2$ are smooth functions. This is because $|{r_\eta}-1|\leq\mathcal{O}(\epsilon)$ due to Remark \ref{r,theta}.

Consider the $\theta$-component. 
Denote by  $\mathcal R_{\om}(r)$ the inverse of 
$r=R_{\om}(\eta)$, i.e. $R_{\om}(\mathcal R_{\om}(r))\equiv r$ 
and by $\xi=\Theta(r,\theta)$ the inverse of 
$\mathcal M^\theta(\eta,\xi,\om)=\theta$, i.e. 
$$
\mathcal M^\theta
(\mathcal R_{\om}(r),\Theta(r,\theta),\om)
\equiv \theta. 
$$
Then we rewrite the $\tilde{\xi}-$equation into $\theta$ equation
as follows 
\[
\theta^*= \theta+\frac{\mathcal R_{\om}(r)}{\lb } 
  \log\frac{ \eps \dt\kappa^\sigma}{\lb} 
+\Delta(r,\theta,a,\om_0\om_1)
+\mathcal O_1^*(\eps\log \eps).
\]
Actually, from our special form of Remark \ref{r,theta} we know $|\mathcal R_{\om}(r)-r|\leq\mathcal{O}(\eps)$ and $\Delta$ depends only on $\om_0\om_1$ since we have (\ref{cocycle form}). 
Besides, the map is exact area-preserving. Therefore, 
the rigidity makes the function $\Delta$ be
$\mathcal O(\eps)$-close to constant functions in $\theta$. 
Benefit from this we can rewrite in the form 
\[
\theta^*= \theta+\frac{r}{\lb } 
(\log  \eps \dt +\log \kappa^\sg \lb^{-1}) 
+N_3(r,a,\om_0\om_1)+\mathcal O_1^*(\eps\log \eps).
\]
Recall that the $\tilde{\xi}-$equation (\ref{cocycle}) is of the cocycle type, i.e. the approximate rotation number doesn't depend on $\omega\in\Sigma$. This property can be saved under $(r,\theta)-$coordinate, so actually $r$ can be updated by $\mathcal{R}(r)$ independent of $\om$ and $N_3(r,a,\om_0\om_1)$ can be also absorbed. On the other side, $\Delta\rightarrow0$ as $a\rightarrow0$. So $\mathcal R(r)$ is still strictly monotone due to Lemma \ref{lem:canonical-coord}. Finally we get the skew product:
\[
r^*=r+ \eps N_1(r,\theta,[\om]_{k+1})+\eps^2 N_2(r,\theta,[\om]_{k+1})+ 
\mathcal O_3^*(\eps)|\log \eps|,
\]
\[
\theta^*= \theta + 
\frac{\mathcal R(r)}{\lb } (\log  \eps \dt +\log \kappa^\sg \lb^{-1})
+\mathcal O_1^*(\eps\log \eps).
\]
\end{proof}
\subsection{A generalization of random iterations}\label{comment}

In previous subsection we deduce the skew product (\ref{eq:skew-shift-model}) from Corollary \ref{cor:random-model}. It's exactly of the standard form as \cite{CK}. The fifth remark under Theorem 2.2 in \cite{CK} can be applied and we get Theorem \ref{thm:main-thm}.

\appendix

\section{Sufficient condition for existence of NHIL}
\label{sec:suff-nhil}
We set the following notations: $x\in \mathbb{R}^{s}, y \in \mathbb{R}^{u},
z\in M$, where $M$ is a smooth Rimannian manifold, possibly
with the boundary $\partial M$, $s$ and $u$ are dimensions 
of the corresponding Euclidean spaces. In the proof we need 
a local linear structure on $M$ given as follows. For a point 
$z\in M$ define a map from $T_zM$ to its neighborhood 
$U\subset M$ by considering the exponential map 
$\exp_z(v)\to M$.
By definition $\exp_z(0)=z$ and for a unit vector $v$ we have 
$\exp_z(tv)$ to be the position of geodesic starting at $z$ and 
for a unit vector $v$ after time $t$. For the Euclidean components 
we assume that the metric is flat and the corresponding exponential
map is linear, i.e. $Z=(x,y,z)$ and $v=(v_x,v_y,v_z)$ resp. 
we have $\exp_Z(v)=(x+v_x,y+v_y,\exp_z(v_z))$.
Denote $\pi_{sc}(x,y,z)=(x,z),\ \pi_u(x,y,z)=y,\ \pi_{uc}(x,y,z)=(y,z),\ 
\pi_{s}(x,y,z)=x$ the respective natural projections.

Fix a positive integer $N$. Let $j=1,\dots,N$, 
$B^s_j\subset \R^{s}$ and $B^u_j \subset \R^{u}$ be 
the unit balls of dimensions $s$ and $u$, $M_j$ 
be a smooth manifold diffeomorphic to $M$. 
Denote 
$D_j^{sc} = B_j^s\times M_j$ and $D_j = D^{sc}_j \times B_j^u$ 
the corresponding manifolds with boundary. By analogy denote 
$D_j^{uc} = B^u \times M_j$ and $D_j = D_j^{sc} \times B^u_j$.

Consider the domain 
\[
\Pi:=\cup_{j=1}^N \Pi_j, \text{ where } \Pi_j = 
B^u_j \times B^s_j \times M_j,
\]
Consider a $C^1$ smooth embedding map $f=(f_s,f_u,f_c):\Pi \to \R^n$, 
given by its components. Consider a subshift of finite type 
$\sigma_A:\Sigma_A \to \Sigma_A$ with a transition $N\times N$ 
matrix $A$. Denote by Ad the set of admissible pairs $ij$. 

\vskip 0.1in 

Suppose for each admissible $ij$ we have 
nonempty sets 
$$
\Pi_{ij}= f^{-1}(\Pi_j)\cap \Pi_i = B^s_i \times B^u_{ij} \times  M_i \subset \Pi_i
$$
for some connected simply connected open sets 
$B^u_{ij}$ and such that 
\begin{itemize}
\item [\bf C1]
$\pi_{sc}f(B^s_i \times B^u_{ij} \times  M_i)\subset 
B^s_j \times M_j$, 

\item  [\bf C2]
$f(B^s_i \times \partial  B^u_{ij} \times  M_i) 
\subset B^s_j \times (\R^{u} \setminus B^u_{j}) 
\times M_j$ maps into and is a homotopy equivalence.

\item  [\bf C3] $f(B^s_i \times B^u_{ij} \times  
\partial  M_i) \subset B^s_j \times \R^{u} 
\times \partial  M_j$.
\end{itemize}

The first two condition means that $f$ contracts along 
the stable direction $s$ and the image of $\Pi_i$ does not 
intersect blocks other than $\Pi_j$. The second condition 
says that $f$ stretches along the unstable direction $u$ so 
that the image of $\Pi_{ij}$ goes across $\Pi_j$. The third 
condition says that orbits can't escape from $\Pi$ through 
the central component. Presence of central directions complicated
analysis of existence  of stable and unstable manifolds. 
To resolve this we assume that the boundary condition [C3]. 

Parallel conditions can be raised for $f^{-1}$:
\begin{itemize}
\item  [\bf C1'] $\pi_{uc}f^{-1}(B^s_{ij} \times B^u_{j} \times  M_j)\subset 
B^u_i \times M_i$,

\item  [\bf C2'] 
$f^{-1}(\partial B^s_{ij} \times  B^u_{j} \times  M_j) 
\subset  (\R^{s} \setminus B^s_{i}) 
\times B_i^u\times M_j$ maps into and is a homotopy equivalence.

\item  [\bf C3'] 
$f^{-1}(B^s_{ij} \times B^u_{j} \times  
\partial  M_j) \subset B^s_i \times \R^{u} 
\times \partial  M_i$.
\end{itemize}

For an admissible $ij\in $Ad denote $f_{ij}:=f|_{\Pi_{ij}}$.
For $\mu>0$ denote the unstable cone
\[
C^u_{\mu,Z}:=\{v=(v_s,v_u,v_c)\in T_ZD: \ 
\mu^2 \,\|v_u\|^2\ge \|v_c\|^2+\|v_s\|^2\},
\] 
where $\|\cdot\|$ is the Riemannian metric of $TD$. Similarly, one can define $C^c_\mu(Z)$ and 
$C^s_{\mu,Z}$. 
Now we state {\it cone conditions}. 

Assume that there are $\mu >1$ and $\nu>1$ with 
the property that for any admissible $ij\in $Ad and any 
$Z_1,Z_2\in D$ such that $Z_2 \in \exp_{Z_1}(C^u_{\mu,Z_1})$ 
we have 
\begin{itemize}
\item[{\bf C4}] $f_{ij}(Z_2)\in 
\exp_{f_{ij}(Z_1)}(C^u_{\mu,f_{ij}(Z_1)}).$
\item[{\bf C5}]  $\|\pi_u(f_{ij}(Z_2)-f_{ij}(Z_1))\|
\ge \nu^u \|\pi_u(Z_2-Z_1)\|.$
\end{itemize}
One can define a set of $\mu$'s and $\nu$'s depending on 
an admissible $ij\in $Ad. Similarly, for $f^{-1}$, $\forall Z_1,Z_2\in D$ such that $Z_2 \in \exp_{Z_1}(C^s_{\mu,Z_1})$

\begin{itemize}
\item[{\bf C4'}] $f_{ij}^{-1}(Z_2)\in 
\exp_{f_{ij}^{-1}(Z_1)}(C^s_{\mu,f_{ij}^{-1}(Z_1)}).$
\item[{\bf C5'}]  $\|\pi_s(f_{ij}^{-1}(Z_2)-f_{ij}^{-1}(Z_1))\|
\ge \nu^s \|\pi_s(Z_2-Z_1)\|.$
\end{itemize}

In order to obtain more refine properties of the unstable
and stable manifolds we introduce additional conditions.  
Denote the linearization matrix $df_{ij}(x)$ and by $T^s,\ T^u,\ T^c$ 
the subspaces tangent to $B_{ij}^s,\ B_{ij}^u,\ M_i^c$ 
respectively. \\

{\bf C6} Assume that there are $0<\lambda_{sc}^+ < \lambda_{u}^+,\ 
0< \lambda_{uc}^- < \lambda_{s}^-,\ m>0$ such  that for each 
$x\in \Pi_{ij}$ we have 
\[
\|\pi _{sc} df_{ij}(x) v_{sc}\|\le \lambda^+_{sc}\|v_{sc}\|,
\qquad 
\|\pi _{u} df_{ij}(x) v_{u}\|\ge \lambda^+_{u}\|v_{u}\|,
\]
\[
\|\pi _{u} df_{ij}(x) v_{sc}\|\le m \|v_{sc}\|,\ \ 
\qquad 
\|\pi _{sc} df_{ij}(x) v_{u}\|\le m \|v_{u}\|,
\]
and 
\[
\|\pi _{uc} df_{ij}^{-1}(x) v_{uc}\|\le \lambda^-_{uc}\|v_{uc}\|,
\qquad 
\|\pi _{s} df_{ij}^{-1}(x) v_{s}\|\ge \lambda^-_{s}\|v_{s}\|,
\]
\[
\|\pi_{s} df_{ij}^{-1}(x) v_{uc}\|\le m \|v_{uc}\|,\ 
\qquad 
\|\pi_{uc} df_{ij}^{-1}(x) v_{s}\|\le m \|v_{s}\|.
\]

%
%
%
%
%
%

  
Denote by $W^{sc}$ the set of points whose positive 
orbits remain inside $\Pi$. Similarly, denote by $W^{uc}$ 
the set of points whose negative orbits remain inside $\Pi$. 
Each of these sets naturally decomposes into $N$ 
components $W^{sc}_i:=W^{sc}\cap \Pi_i, \ i=1\dots,N$.  
\begin{thm} 
\label{thm:nhil}
 Assume that conditions {\bf [C1-C5]} hold, then the set 
$W^{sc}=\cup_{i=1}^N W^{sc}_i$ 
is a collection of graphs of Lipschitz functions, i.e. for any 
$\omega^+\in  \Sigma_A^+$ and any $i=1,\dots,N$ we have that 
$W^{sc}_i(\cdot\,,\omega^+)$ is a graph of a Lipschitz function 
\[
W^{sc}_i:B_i^s\times M_i \times \omega^+ \to B_i^u
\]
and the set $W^{uc}=\cup_{i=1}^N W^{uc}_i$ 
is a collection graphs of Lipschitz functions with
\[
W^{uc}_i:B^u_i\times M_i \times \omega^- \to B_i^s.
\]
Therefore, the set $W^{c}=\cup_{i=1}^N W^{c}_ i$ 
is a collection graphs of Lipschitz functions, i.e. 
for any $\omega\in  \Sigma_A$ and any $i=1,\dots,N$ 
we have that $W^{c}_i(\cdot,\omega)$ is a graph of 
a Lipschitz function 
\[
W^c_i=(W^{sc}_i,W^{uc}_i):M_i \times \omega \to B_i^s
\times B_i^u.
\]
Moreover, {\bf C6} implies
\[
\rho_-=\max\{m,\lambda_{uc}^-\}\text{\;,\;}\rho_+=\max\{m,\lambda_{sc}^+\}
\]
and
\[
\nu_-=\lambda_s^-\cdot\lambda_{uc}^+\text{\;,\;}\nu_+=\lambda_u^+\cdot\lambda_{sc}^-.
\]
Once 
$$
\rho_{\pm}^k \nu_{\pm}^{-1}<1
$$
is satisfied for an integer $k\geq1$ and all parameters on an admissible 
$ij$, the $W_{ij}^c$ is $C^r$ smooth for $ij\in\Sigma_{A}$.

\end{thm}

Recall that $\rho_{\pm}$ and $\mu_{\pm}$ are dependent of all parameters on an admissible 
$ij$, then this condition can be formalized to 
$$
\max_{ij\in \text{Ad}}
\rho_{ij}^k \nu_{ij}^{-1}<1.
$$

In the case $\Sigma^\pm_A$ is a single point the result can be 
deduced from known results see \cite{Fe, HPS,McG,BKZ}. In large 
part we follows the proof from the book of Shub \cite{Shu}. 

\begin{proof} We start by proving that each $W^{sc}_i$ is
a Lipschitz manifold. We reply on Proposition D.1 \cite{KZ14}.  
Since it is short we reproduce it here. Fix $\omega\in \Sigma^+_A$.
 
Let $\mathcal V_{i}$ be the set  
$\Gamma_{i} \subset \Pi_{i}$ satisfying 
the following conditions: 
for each admissible $ij\in $ Ad we have
\begin{itemize}
\item (a)  $\pi_{u}\Gamma_{i} \supset B_{ij}^u$, 
\item (b) $Z_2 \in \exp_{Z_1}(C_{Z_1}^u)$ for all 
$Z_1, Z_2 \in \Gamma_{ij}:=\Gamma_{i}\cap f^{-1}(\Pi_{j})$, 
\end{itemize}
where $\pi_u$ is the projection to the unstable component. 
These conditions ensures $\pi_u: \Gamma_{ij} \to B_{ij}^u$ is 
one-to-one and onto, therefore, $\Gamma_{ij}$ is a graph 
over $B_{ij}^u$. Moreover, condition (b) further implies that 
the graph is Lipschitz. In particular, each $\Gamma_{ij} \in \mathcal V_{ij}$  
is a topological disk. 

\begin{lem}\label{lem:inv-unstable-disk}
Let $\Gamma_{i} \in \mathcal V_{i}$, then 
$f_{ij}(\Gamma_{ij})\cap D\in \mathcal V_{j}$. 
\end{lem}
\begin{proof}
By [C4] for any $Z_1$ and $Z_2$ we have that 
$f_{ij}(Z_2)$ belongs to the cone $C^u_{f_{ij}(Z_1)}$ of $f_{ij}(Z_1)$.
Thus, it suffices to show that $B_j^{u}\subset \pi_{u}
(f_{ij}(\Gamma)\cap D)$. The proof is by contradiction. 
Suppose there is $Z_* \in B_j^{u}$ such that 
$Z_*\not\in \pi_{u}(f_{ij}(\Gamma_i))$. 

We have the following commutative diagram
\begin{equation}
\begin{aligned}
\partial \Gamma_{ij} \qquad & \quad \longrightarrow{i_1} & \Gamma_{ij} \ \ \qquad \\ 
\downarrow \pi_{u}\circ f_{ij} & & \downarrow \pi_{u}\circ f_{ij} \\
\mathbb{R}^u\setminus B_{ij}^u& \quad \longrightarrow{i_2} & \mathbb{R}^u \setminus\{Z_*\}
\end{aligned}
\end{equation}
and by [C2] and standard topology, both 
$\pi_u \circ f_{ij} |_{ \partial \Gamma_{ij}}$ and $i_2$ are 
homotopy equivalences. Also $\pi_u \circ f_{ij} |_{ \Gamma_{ij}}$ 
is a homeomorphism onto its image. 
Since the diagram commutes, $\Gamma_{ij}$ is homotopic to 
$\partial \Gamma_{ij}$, which is a contradiction. 
\end{proof}


The first part of Theorem \ref{thm:nhil} 
follows from the next statement.
\begin{prop}\label{pro}
The mapping $\pi_{sc}: W^{sc}_{i}\to D_i^{sc}$ is one-to-one 
and onto, therefore, it is the graph of a function $W_i^{sc}$. 
Moreover, $W^{sc}_i$ is Lipschitz and 
\[
	T_Z W^{sc}_i \in (C_\mu^{u}(Z))^c = 
	C_{\mu^{-1}}^{sc}(Z), \quad Z \in W^{sc}_i. 
\] 
\end{prop}
\begin{proof}
For each $X \in D_i^{sc}$, we define $\Gamma_X = (\pi_{sc})^{-1} X$, 
clearly $\Gamma_X \in \mathcal V_i$. We first show 
$\Gamma_X \cap W^{sc}_i$ is nonempty and consists of 
a single point. Assume first that $\Gamma_X \cap W^{sc}_i$ 
is empty. Then by definition of $W^{sc}_i$, there is $n \in \Z_+$ 
and a composition of $n$ admissible maps $f_{i_0i_1},
f_{i_1i_2}, \dots, f_{i_{n-1}i_n}$ such that  
$$
 f_{i_{n-1}i_n}f_{i_1i_2} \dots f_{i_0i_1}
(\Gamma_X) \cap D_{i_n} = \varnothing.
$$
However, by Lemma~\ref{lem:inv-unstable-disk}, 
$$\bigcap_{i=1}^n 
f_{i_{n-1}i_n}f_{i_1i_2} \dots f_{i_0i_1}(\Gamma_X) \cap D \in \mathcal V_{i_n}
$$ 
is always nonempty, a contradiction. We now consider 
two points $Z_1, Z_2 \in W^{sc}_i$ with $\pi_u Z_1 = \pi_u Z_2$. Since 
$Z_2 \in \exp_{Z_1}C_\mu^u(Z_1)$, by [C5] we have 
\[
2\ge \|\pi_u(f^k(Z_1)-f^k(Z_2))\|\ge \nu^k\|\pi_u(Z_1-Z_2)\| 
\]
for all $k$, which implies $Z_1=Z_2$.

The last argument actually shows $Z_2 \notin \exp_{Z_1}C_\mu^u(Z_1)$ for all $Z_1, Z_2 \in W^{sc}_i$. 
For any $\epsilon>0$, for $Z_1=(X_1, Y_1), Z_2=(X_2, Y_2) 
\in W^{sc}_i$ with dist$(X_1, X_2)$ small, we have  
$\|Y_1 - Y_2\| \le (\mu^{-\frac12} + \epsilon)$  
dist$(X_1, X_2)$. This implies both the Lipschitz and the cone properties in our proposition.
\end{proof}

The second part of Theorem \ref{thm:nhil} is due to the following $C^r$ section theorem. For the consistency of our paper we rewrite it under our symbol system, but the original version can be found in \cite{Shu} or \cite{HPS}. Now we finish the proof of Theorem \ref{thm:nhil}.
\end{proof}
\begin{thm}[$C^r$ Section]\label{cr}
Let $\Pi:E\rightarrow X$ be a vector bundle over the metric space $X$, where $E$ has a splitting by $E^u\bigoplus E^s\bigoplus E^c$. Let $X_0$ is an invariant subset of $X$ and $D$ be the disc bundle of radius $C$ in $E$, where $C>0$ is a finite constant. Let $D_0$ be the restriction of $D$ over $X_0$, i.e. $D_0=D\bigcap\Pi^{-1}(X_0)$. 

Suppose $F=(f,Df):D_0\rightarrow D$ be the covering function of $f$. $\forall x\in X_0$, there exists a Lipschitz  invariant graph in the bundle space $E_x$ which can be locally formed by
\[
id\times g_{inv}^{uc}(x,\cdot):X_0\times E^{uc}_x\rightarrow E^{uc}_x\times E_x^s,
\]
with the Lipschitz constant bounded by $k_s$. We can define a couple of functions 
\[
h^{uc}_x=\pi_{uc}\circ{Df}(x)\cdot(id,g_{inv}^{uc}):E_x^{uc}\rightarrow E_{f(x)}^{uc},\quad\forall x\in X_0
\]
and
\[
F^{uc}_x:E_x^{uc}\times\mathbb{L}(E_x^{uc},E_x^s)\rightarrow E_{f(x)}^{uc}\times\mathbb{L}(E_{f(x)}^{uc},E_{f(x)}^s)
\]
via 
\[
(\xi,\eta,z^u,\sigma_x^{uc}(\xi,\eta,z^u))\rightarrow(h_x^{uc}(\xi,\eta,z^u),\sigma_{f(x)}^{uc}(h^{uc}_x(\xi,\eta,z^u))),
\] 
i.e. the following diagram commutes:
\begin{equation} 
\begin{array}[c]{ccc} 
\quad\ E_x^{uc}&
\quad  \stackrel{h^{uc}} {\longrightarrow} & 
\quad E_{f(x)}^{uc} \\
\pi_{uc}\ \uparrow &  & \pi_{uc}\ \uparrow \\ 
\ \ \ \ \ \ | &  & \ \ \ \ \ \ | \\ 
 E_x^{uc} \times \mathbb{L}(E_x^{uc},E_x^s)
& \quad  \stackrel{F^{uc}} {\longrightarrow} & \ \ 
 E_{f(x)}^{uc} \times \mathbb{L}(E_{f(x)}^{uc},E_{f(x)}^s),
\end{array}
\end{equation}
where $\mathbb{L}(E_x^{uc},E_x^s)$ is the linear transformation space and $\sigma_x^{uc}:E_x^{uc}\rightarrow\mathbb{L}(E_x^{uc},E_x^s)$. $\forall x\in X_0$, $F_{uc}(x,\cdot)$ is Lipschitz with constant at most $\nu_{uc}^{-1}$. 
\begin{itemize}
\item There exits a unique section map $\sigma^{uc}_{inv}(x,\cdot):E_x^{uc}\rightarrow\mathbb{L}(E_x^{uc},E_x^s)$ such that 
\[
\sigma_{inv}^{uc}(f(x),h^{uc}_{x}(\xi,\eta,z^u))=\pi_2F^{uc}_{x}(\sigma^{uc}_{inv}(x,\xi,\eta,z^u)),\forall x\in X_0,\  (\xi,\eta,z^u)\in E_x^{uc};
\]
\item If $F^{uc}$ is continuous, so is $\sigma^{uc}_{inv}$;
\item If moreover, $h_{uc}^{-1}$ is Lipschitz with $Lip(h_{uc}^{-1})=\rho_{uc}$, $F_{uc}$ is $\alpha-$H\"older, and $\nu_{uc}^{-1}\rho_{uc}^{\alpha}<1$, then $\sigma^{uc}_{inv}$ is $\alpha-$H\"older; In particular, when $\alpha=1$, $\sigma^{uc}_{inv}$ is Lipschitz;
\item If moreover, $X$, $X_0$ and $E$ are $C^r$ manifolds ($r\geq1$), $h_{uc}$ and $F_{uc}$ are $C^r$, $j-$th order derivatives of $h_{uc}^{-1}$ and $F_{uc}$ are bounded for $1\leq j\leq r$ and Lipschitz for $1\leq j<r$, there exists a $r\geq1$, such that $\rho_{uc}=Lip(h^{-1}_{uc})$ and $\nu_{uc}^{-1}=Lip(F_{uc})$, and $\rho_{uc}^r\nu_{uc}^{-1}<1$, then the backward invariant graph $\sigma_{inv}^{uc}$ is $C^r$.
\end{itemize}
Similarly, $\forall x\in X_0$, there exists a Lipschitz  invariant graph in the bundle space $E_x$ which can be locally formed by
\[
id\times g^{sc}_{inv}(x,\cdot):E^{sc}_x\rightarrow E^{sc}_x\times E_x^u,
\]
with the Lipschitz constant bounded by $k_u$. We can define 
\[
h^{sc}_x=\pi_{sc}\circ{Df^{-1}(x)}\cdot(id,g_{inv}^{sc}):E_{x}^{sc}\rightarrow E_{f^{-1}(x)}^{sc},\quad \forall x\in X_0
\]
and
\[
F^{sc}_x=:E_x^{sc}\times\mathbb{L}(E_x^{sc},E_x^u)\rightarrow E_{f^{-1}(x)}^{sc}\times\mathbb{L}(E_{f^{-1}(x)}^{sc},E_{f^{-1}(x)}^u)
\]
via
\[
(\xi,\eta,z^s,\sigma_x^{sc}(\xi,\eta,z^s))\rightarrow(h_x^{sc}(\xi,\eta,z^s),\sigma_{f^{-1}(x)}^{sc}(h_x^{sc}(\xi,\eta,z^s)),
\]
i.e. the following diagram commutes:
\begin{equation} 
\begin{array}[c]{ccc} 
\quad\ E_x^{sc}&
\quad  \stackrel{h^{sc}} {\longrightarrow} & 
\quad E_{f^{-1}(x)}^{sc} \\
\pi_{sc}\ \uparrow &  & \pi_{sc}\ \uparrow \\ 
\ \ \ \ \ | &  & \ \ \ \ \ | \\ 
 E_x^{sc} \times \mathbb{L}(E_x^{sc},E_x^u)
& \quad  \stackrel{F^{sc}} {\longrightarrow} & \ \ 
 E_{f^{-1}(x)}^{sc} \times \mathbb{L}(E_{f^{-1}(x)}^{sc},E_{f^{-1}(x)}^u).
\end{array}
\end{equation}
$\forall x\in X_0$, $F_{sc}(x,\cdot)$ is Lipschitz with constant at most $\nu_{sc}^{-1}$. 

\begin{itemize}
\item There exits a unique section map $\sigma_{inv}^{sc}(x,\cdot):E_x^{sc}\rightarrow\mathbb{L}(E_x^{sc},E_x^u)$ such that
\[
\sigma_{inv}^{sc}(f^{-1}(x),h_x^{sc}(\xi,\eta,z^s))=\pi_2F_x^{sc}(\sigma^{sc}_{inv}(x,\xi,\eta,z^s))\quad\forall x\in X_0, (\xi,\eta,z^s)\in E_x^{sc};
\]
\item If $F_{sc}$ is continuous, so is $\sigma^{sc}_{inv}$;
\item If moreover, $h_{sc}^{-1}$ is Lipschitz with $Lip(h_{sc}^{-1})=\rho_{sc}$, $F^{sc}$ is $\alpha-$H\"older, and $\nu_{sc}^{-1}\rho_{sc}^{\alpha}<1$, then $\sigma^{sc}_{inv}$ is $\alpha-$H\"older; In particular, when $\alpha=1$, $\sigma^{sc}_{inv}$ is Lipschitz;
\item If moreover, $X$, $X_0$ and $E$ are $C^r$ manifolds ($r\geq1$), $h_{uc}$ and $F_{sc}$ are $C^r$, $j-$th order derivatives of $h_{sc}^{-1}$ and $F_{sc}$ are bounded for $1\leq j\leq r$ and Lipschitz for $1\leq j<r$, there exists a $r\geq1$, such that $\rho_{sc}=Lip(h^{-1}_{sc})$ and $\nu_{sc}^{-1}=Lip(F_{sc})$, and $\rho_{sc}^r\nu_{sc}^{-1}<1$, then the forward invariant graph $\sigma_{inv}^{sc}$ is $C^r$.
\end{itemize}

\end{thm}

\begin{rmk}
Theorem \ref{cr} allows us to prove the smoothness of unstable (stable) manifold by induction: Actually, the exponential map will send $g_{inv}^{uc}$ ($g_{inv}^{sc}$) into manifolds $W_{inv}^{uc}$ ($W_{inv}^{sc}$). We already know that restricted on each leaf, $g_{inv}^{uc}$ ($g_{inv}^{sc}$) is $C^1$ whenever $h,F$ are. This is because $\sigma_{inv}^{uc}$($\sigma_{inv}^{sc}$) is continuous and is actually the $1-$jet of $g_{inv}^{uc}$($g_{inv}^{sc}$) due to former two bullets of aforementioned Theorem. Then suppose $g_{inv}^{uc}$ ($g_{inv}^{sc}$) is already $C^{s-1}$, $s\geq2$, use the last bullet and we can get $\sigma_{inv}^{uc}$($\sigma_{inv}^{sc}$) is also $C^{s-1}$ hence $g_{inv}^{uc}$($g_{inv}^{sc}$) is $C^s$. So the induction can be repeated until $s=r$.

\end{rmk}

\begin{rmk}
For the setting of Theorem \ref{thm:nhil}, we just need to take $X_0$ by $W^c$, $X$ by $\mathbb{R}^n$ and the splitting $E^u\bigoplus E^s\bigoplus E^c$ by the invariant splitting of $W^c$. We already know that $W^c$ is invariant due to {\bf C1} to {\bf C5}, so such a splitting does exist.
\end{rmk}

\section{H\"older continuity of jet space for hyperbolic invariant set}
\label{holder for jet}
Notice that we need to get an available normal form (\ref{eq:skew-shift-normal-form}), of which \cite{CK} can be used to get our main conclusion (see Appendix \ref{sec:oriol-result}). So we still need to prove the regularity of $W^c$ in $\omega$, for which at least some $\varphi-$H\"older regularity should be ensured, $\varphi>0$. The crucial idea for this is the following Theorem, which is translated to adapt our setting from Theorem 6.1.3. of \cite{BrSt}.

\begin{thm}\label{holder}
Let $\Lambda\hookrightarrow M$ be a compact invariant embedding set of a $C^{\infty}$ diffeomorphism $f:M\rightarrow M$. Suppose there exists a splitting on the tangent bundle by:
\[
T_{\Lambda}M=E_1^c\oplus E_1^u\oplus E_1^s
\]
and $0<\lambda_+^{sc}<\lambda_+^u$ such that $\|df^n(x_1)v_1^{s,c}\|\leq C(\lambda_+^{s,c})^n\|v_1^{s,c}\|$, $\|df^n(x_1)v_1^{u}\|\geq C(\lambda_+^u)^n\|v_1^{u}\|$ hold for all $x_1\in\Lambda$, $v_1^{s,c}\in E_1^{s,c}(x_1)$ and $v_1^u\in E_1^u(x_1)$, where $C$ is a proper constant and $n\in\mathbb{N}$. Let 
\[
f_i(x_i,v_i)\doteq(f_{i-1}(x_{i}),Df_{i-1}(x_{i})v_i),\quad i\in\mathbb{N},
\]
be the $i$th-jet map with $x_i=(x_{i-1},v_{i-1})$, $(x_{i},v_i)\in {\underbrace{T\cdots T}_{i}}\;_\Lambda M=T_{T^{i-1}_{\Lambda}M}(T^{i-1}M)$ and $f_0(x_1)=f(x_1)$ for $x_1\in\Lambda$. Suppose Theorem \ref{cr} holds for $f$ and $\Lambda$, and $W^c(x)$ is the center manifold. There exists a $i$th-jet splitting by 
\[
{\underbrace{T\cdots T}_{i}}\;_{\Lambda}M=E^c_i\bigoplus E^u_i\bigoplus E^s_i
\]
with 
\[
E^{uc}_i(x_i)\Big{|}_{x_1\in\Lambda}=\underbrace{T\cdots T}_iW^{uc}(x_1)\Big{|}_{x_1\in\Lambda}
\]
\[
E^{sc}_i(x_i)\Big{|}_{x_1\in\Lambda}=\underbrace{T\cdots T}_iW^{sc}(x_1)\Big{|}_{x_1\in\Lambda}
\]
and
\[
E^{uc}_i(x_i)\pitchfork E^{sc}_i(x_i)\Big{|}_{x_1\in\Lambda}=E^c_i(x_i)\Big{|}_{x_1\in\Lambda}=\underbrace{T\cdots T}_i\Lambda(x_1)\Big{|}_{x_1\in\Lambda}.
\]
Besides, if we assume
\[
b_i^+=\max_{x\in\Lambda}3^{4i}(i+1)!2^{(i+2)(i+3)/2}\|f\|_{C^{i+3}}^{i+3},
\]
 then the stable/center distribution $E_i^{s,c}(x_{i})$ is H\"older continuous with exponent $\varphi_i^{s,c}=(\ln\lambda_+^u-\ln\lambda_+^{s,c})/(\ln b_i^+-\ln\lambda_+^{s,c})$, where $\|v_j\|\leq1$, $0\leq j\leq i-1$.

Similarly, if there exists $0<\lambda_-^{uc}<\lambda_-^s$ and a proper constant $C$ such that $\|df^{-n}(x_1)v_1^{u,c}\|\leq C(\lambda_-^{u,c})^n\|v_1^{u,c}\|$, $\|df^{-n}(x_1)v_1^{s}\|\geq C(\lambda_-^s)^n\|v_1^{s}\|$ hold for all $x_1\in\Lambda$, $v_1^{u,c}\in E_1^{u,c}(x_1)$ and $v_1^s\in E_1^s(x_1)$, $n\in\mathbb{N}$. Let 
\[
b_i^-=\max_{x\in\Lambda}3^{4i}(i+1)!2^{(i+2)(i+3)/2}\|f^{-1}\|_{C^{i+3}}^{i+3}.
\]
Then the unstable/center distribution $E_i^{u,c}(x_{i})$ is H\"older continuous with exponent $\varphi_i^{u,c}=(\ln\lambda_-^s-\ln\lambda_-^{u,c})/(\ln b_i^--\ln\lambda_-^{u,c})$, where $\|v_j\|\leq1$, $0\leq j\leq i-1$.
\end{thm}
\begin{proof}
Without loss of generality, we can assume that $M$ is embedded in $\mathbb{R}^N$. As 
\begin{equation}\label{matrix}
Df_{i}(x_i,v_i)=\begin{pmatrix}
   Df_{i-1}(x_{i})   & 0   \\
  D^2f_{i-1}(x_i)v_i    & Df_{i-1}(x_{i}) 
\end{pmatrix},\quad\forall (x_i,v_i)\in \underbrace{T\cdots T}_i\;_{\Lambda}M,
\end{equation}
that means $Df_{i}(x_i,v_i)$ has the same eigenvalues with $Df_{i-1}(x_{i-1},v_{i-1})$. Besides, we have
\[
E^{uc}_i(x_i)\Big{|}_{x_1\in\Lambda}=\underbrace{T\cdots T}_iW^{uc}(x_1)\Big{|}_{x_1\in\Lambda}
\]
\[
E^{sc}_i(x_i)\Big{|}_{x_1\in\Lambda}=\underbrace{T\cdots T}_iW^{sc}(x_1)\Big{|}_{x_1\in\Lambda}
\]
and
\[
E^{uc}_i(x_i)\pitchfork E^{sc}_i(x_i)\Big{|}_{x_1\in\Lambda}=E^c_i(x_i)\Big{|}_{x_1\in\Lambda}=\underbrace{T\cdots T}_i\Lambda(x_1)\Big{|}_{x_1\in\Lambda}
\]
due to the backward invariance of $W^{uc}$ (forward invariance of $W^{sc}$). The $C^r-$section theorem ensures the smoothness of $W^{uc}(x)$ and $W^{sc}(x)$ for $x$ in a certain leaf. 

Now we use induction to prove the H\"older continuity. From (\ref{matrix}), we know that there exists a constant ${C}_i>1$ such that for any $(x_i,v_i)\in\underbrace{ T\cdots T}_iM$ with $x_1\in\Lambda$ and $\|v_j\|\leq1$ for $1\leq j\leq i$, it's tangent space $(x_{i+1},v_{i+1})\in\underbrace{ T\cdots T}_{i+1}M$ satisfies:
\[
\|Df_i^n(x_i,v_i)v_{i+1}\|\geq C_i^{-1}(\lambda_+^u)^n\|v_{i+1}\|,
\]
if $v_{i+1}\bot E_{i+1}^{sc}(x_{i+1})$.

We can extend $Df_i(x_{i+1})$ to a linear map $L_i(x_{i+1}):\underbrace{ T\cdots T}_{i+1}M\rightarrow\underbrace{ T\cdots T}_{i+1}M$ by setting $L_i(x_{i+1})\Big{|}_{E^{sc^{\bot}}_{i+1}(x_{i+1})}=0$, and 
\[
L_{i,n}(x_{i+1})=L_i(f_i^{n-1}(x_i,v_i))\circ\cdots\circ L_i(f_i(x_i,v_i))\circ L_i(x_i,v_i).
\]
Note that $L_{i,n}(x_i,v_i)\Big{|}_{T_{(x_i,v_i)}\underbrace{ T\cdots T}_i\Lambda}=Df_i^n(x_i,v_i)$.

Fix two points $x_{i+1,1}$ and $x_{i+1,2}$ of $\underbrace{ T\cdots T}_iM$ with $\|x_{i+1,1}-x_{i+1,2}\|\leq1$. The following Lemma \ref{operator} and Lemma \ref{induction} are satisfied with $L_{i,n}^k=L_{i,n}(x_{i+1,k})$ and $E^k_{i+1}=E^{sc}_{i+1}(x_{i+1,k})$, $k=1,2$. Then the first part of theorem follows. Similar way for $E^{uc}_{i+1}(x_{i+1,k})$ and $f^{-1}$ we get the second part.
\end{proof}
\begin{lem}\label{operator}
Let $L_{n,i}^k:\mathbb{R}^K\rightarrow\mathbb{R}^K$, $k=1,2$, $n\in\mathbb{N}$ be two sequences of linear maps. Assume that for some $b_i>0$ and $\delta_i\in(0,1)$
\[
\|L_{n,i}^1-L_{n,i}^2\|\leq\delta b_i^n,\quad i\in\mathbb{N}
\]
and there exist two subspaces $E_i^1$, $E_i^2$ and positive constants $C_i>1$ and $\lambda_i<\mu_i$ with $\lambda_i<b_i$ such that
\[
\|L_{n,i}^kv_{i+1}\|\leq C_i\lambda_i^n\|v_{i+1}\|,\quad\forall v_{i+1}\in E_{i+1}^k,
\]
\[
\|L_{n,i}^kw_{i+1}\|\geq C_i^{-1}\mu_i^n\|w_{i+1}\|,\quad\forall w_{i+1}\bot E_{i+1}^k.
\]
Then $dist(E_i^1,E_i^2)\leq 3C_i^2\frac{\mu_i}{\lambda_i}\delta^{\frac{\ln\mu_i/\lambda_i}{\ln b_i/\lambda_i}}$. Here the distance of two linear spaces is defined by $dist(A,B)=\max\{\max_{v\in A,\;\|v\|=1}dist(v,B),\;\max_{w\in B,\;\|w\|=1}dist(w,A)\}$.
\end{lem}
\begin{proof}

Set $K_{n,i}^k=\{v_{i+1}\in\underbrace{T\cdots T}_{i+1}\;_{\Lambda}M\Big{|}\|L_{n,i}^kv_{i+1}\|\leq2C_i\lambda_i^n\|v_{i+1}\|\}$, $k=1,2$. Let $v_{i+1}\in K_{n,i}^1$. Write $v_{i+1}=v_{i+1}^1+v_{i+1,\bot}^1$, where $v_{i+1}^1\in E_{i+1}^1$ and $v_{i+1,\bot}^1\bot E_{i+1}^1$. Then
\begin{eqnarray*}
\|L_{n,i}^1v_{i+1}\|=\|L_{n,i}^1(v_{i+1}^1+v_{i+1,\bot}^1)\|&\geq&\|L_{n,i}^1v_{i+1,\bot}^1\|-\|L_{n,i}^1v_{i+1}^1\|\\
&\geq&C_i^{-1}\mu_i^n\|v_{i+1,\bot}^1\|-C_i\lambda_i^n\|v_{i+1}^1\|,
\end{eqnarray*}
and hence
\[
\|v_{i+1,\bot}^1\|\leq C_i\mu_i^{-n}(\|L_{n,i}^1v_{i+1}\|+C_i\lambda_i^n\|v_{i+1}^1\|)\leq3C_i(\frac{\lambda_i}{\mu_i})^n\|v_{i+1}\|.
\]
It follows that 
\begin{equation}\label{distance}
dist(v_{i+1},E_{i+1}^1)\leq 3C_i^2(\frac{\lambda_i}{\mu_i})^n\|v_{i+1}\|.
\end{equation}
Set $\gamma=\lambda_i/b_i<1$. There is a unique non-negative integer $k$ such that $\gamma^{k+1}\leq\delta\leq\gamma^{k}$. Let $v_{i+1}^2\in E_{i+1}^2$, then 
\begin{eqnarray*}
\|L_{k,i}^1v_{i+1}^2\|&\leq&\|L_{k,i}^2v_{i+1}^2\|+\|L_{k,i}^1-L_{k,i}^2\|\cdot\|v_{i+1}^2\|\\
&\leq& C_i\lambda_i^k\|v_{i+1}^2\|+b_i^k\delta\|v_{i+1}^2\|\\
&\leq& (C_i\lambda_i^k+b_i^k\gamma^k)\|v_{i+1}^2\|\leq 2C_i\lambda_i^k\|v_{i+1}^2\|.
\end{eqnarray*}
It follows that $v_{i+1}^2\in K_{n,i}^1$ and hence $E_{i+1}^2\subset K_{n,i}^1$. By symmetry we get $E_{i+1}^1\subset K_{n,i}^2$. By (\ref{distance}) and the choice of $k$,
\[
dist(E_{i+1}^1,E_{i+1}^2)\leq 3C_i^2(\frac{\lambda_i}{\mu_i})^k\leq 3C_i^2\frac{\mu_i}{\lambda_i}\delta^{\frac{\ln\mu_i/\lambda_i}{\ln b_i/\lambda_i}}.
\]

\end{proof}
\begin{lem}\label{induction}
Let $f:\Lambda\rightarrow\Lambda$ be a $C^{\infty}$ diffeomorphism with $\Lambda\hookrightarrow M$ be a compact embedded manifolds and 
\[
f_i(x_i,v_i)\doteq(f_{i-1}(x_{i}),Df_{i-1}(x_{i})v_i),\quad i\in\mathbb{N},
\]
be the $i$th-jet map with $x_i=(x_{i-1},v_{i-1})$, $(x_{i},v_i)\in {\underbrace{T\cdots T}_{i}}M$ and $f_0(x_1)=f(x_1)$ for $x_1\in\Lambda$. Then for each $n, k\in\mathbb{N}$ and all $(x_{k},v_{k})$, $(y_k,w_k)\in \underbrace{T\cdots T}_kM$, $\|v_k\|,\|w_k\|\leq1$, we have
\begin{equation}\label{b constant}
\|Df_k^n(x_k,v_k)-Df_k^n(y_k,w_k)\|\leq\breve{C}_k b_k^n\|(x_k,v_k)-(y_k,w_k)\|
\end{equation}
where 
\[
b_k=\max_{x\in\Lambda}3^{4k}(k+1)!2^{(k+2)(k+3)/2}\|f\|_{C^{k+3}}^{k+3}
\]
and $\breve{C}_k$ is a constant only depending on $v_j$, $1\leq j\leq k$ and $1\leq k\leq i$.
\end{lem}
\begin{proof}
Here we use a multiple induction of index $k$ and $n$. Recall that the following 
\[
f_i^n(x_i,v_i)=(f_{i-1}^n(x_i), D(f_i^n(x_i))v_i)
\]
and 
\[
D(f_i^n)(x_i,v_i)=\begin{pmatrix}
   D(f_{i-1}^n)(x_{i})   & 0   \\
  D^2(f_{i-1}^n)(x_i)v_i    & D(f_{i-1}^n)(x_{i}) 
\end{pmatrix},\quad\forall (x_i,v_i)\in \underbrace{T\cdots T}_i\;_{\Lambda}M,
\]
hold due to (\ref{matrix}). Without loss of generality, we can assume $\|f_i\|_{C^0}\geq1$, $i\in\mathbb{N}$. For $k=1$, we already have
\begin{eqnarray*}
\|f_1(x_1,v_1)-f_1(y_1,w_1)\|&\leq&\|f(x_1)-f(y_1)\|+\|Df(x_1)v_1-Df(y_1)w_1\|\\
&\leq&\|Df\|\|x_1-y_1\|+\|Df\|\|v_1-w_1\|+\|D^2f\|\|w_1\|\|x_1-y_1\|\\
&\leq&2\|f\Big{|}_{\Lambda}\|_{C^2}(\|x_1-y_1\|+\|v_1-w_1\|)
\end{eqnarray*}
and
\begin{eqnarray*}
\|Df_1(x_1,v_1)-Df_1(y_1,w_1)\|&\leq&2\|Df(x_1)-Df(y_1)\|+\|D^2f(x_1)v_1-D^2f(y_1)w_1\|\\
&\leq& 2\|D^2f\|\|x_1-y_1\|+\|D^2f\|\|v_1-w_1\|+\|D^3f\|\|w_1\|\|x_1-y_1\|\\
&\leq& 3\|f\Big{|}_{\Lambda}\|_{C^3}(\|x_1-y_1\|+\|v_1-w_1\|).
\end{eqnarray*}
For the inductive step, we have
\begin{eqnarray*}
\|Df_1^{n+1}(x_1,v_1)&-&Df_1^{n+1}(y_1,w_1)\|=\|Df_1(f_1^n(x_1,v_1)Df_1^n(x_1,v_1)-\\
& &\quad\quad\quad\quad\quad\quad\quad\quad\quad Df_1(f_1^n(y_1,w_1)Df_1^n(y_1,w_1)\|\\
&\leq&\|Df_1(f_1^n(x_1,v_1))\|\|Df_1^n(x_1,v_1)-Df_1^n(y_1,w_1)\|+\\
& & \|Df_1^n(y_1,w_1)\|\|Df_1(f_1^n(x_1,v_1))-Df_1(f_1^n(y_1,w_1))\|\\
 &\leq& \max_{\theta\in[0,1]}\|D^2f_1^n(\theta x_1+(1-\theta)y_1,\theta v_1+(1-\theta)w_1)\|\cdot\nonumber\\
 & & \|Df_1(f^n(x_1),Df^n(x_1)v_1)\|(\|x_1-y_1\|+\|v_1-w_1\|) \nonumber\\
 & &+ \|Df_1(f^n(x_1),Df^n(x_1)v_1)-Df_1(f^n(y_1),Df^n(y_1)w_1)\|\cdot \\
 & & \|Df_1^n(y_1,w_1)\|\nonumber\\
 &\leq& (2\|Df(f^n(x_1))\|+\|D^2f(f^n(x_1))\|\cdot\nonumber\\
& & \|Df^n(x_1)\|)\cdot2\max_{\theta\in[0,1]}\|f^n(\theta x_1+(1-\theta)y_1)\|_{C^2}\cdot\nonumber\\
& & (\|x_1-y_1\|+\|v_1-w_1\|) +\Big{\{}2\|Df(f^n(x_1))-\nonumber\\
& & Df(f^n(y_1))\| +\|v_1\|\|Df^n(x_1)\|\|D^2f(f^n(x_1))-D^2f(f^n(y_1))\|\nonumber\\
 & & +\|D^2f(f^n(y_1))\|\|Df^n(x_1)\|\|v_1-w_1\|+\nonumber\\
 & & \|D^2f(f^n(y_1))\|\|w_1\|\|Df^n(x_1)-Df^n(y_1)\|\Big{\}}\|Df^n(y_1)\|_{C^2},\nonumber\\
&\leq&2\|f\|^n_{C^1}\|f\|_{C^2}4\|f\|_{C^3}^n(\|x_1-y_1\|+\|v_1-w_1\|)+\\
& & (2\|f\|_{C^1}^n+\|f\|_{C^2}\|f\|_{C^1}^{n-1})(3\|f\|_{C^2}\|f\|_{C^1}^n+\|f\|_{C^2}\|f\|_{C^1}^{n-1}+\\
& & \|f\|_{C^3}\|f\|_{C^1}^n)(\|x_1-y_1\|+\|v_1-w_1\|)\\
&\leq&\breve{C}_1b_1^{n+1}(\|x_1-y_1\|+\|v_1-w_1\|),
\end{eqnarray*}
then (\ref{b constant}) holds for case $k=1$ if $b_1=\max_{x\in\Lambda}3^{2}2!\|f\|_{C^3}^4$. Recall that $\|v_1\|,\|w_1\|\leq 1$ holds through this estimate.

Before we take the inductive step for $k>1$, we need to introduce two useful conclusions:
\[
\|Df_k^n(x_k,v_k)\|_{C^l}\leq3^k\|f\|_{C^{k+l}}^n
\]
and
\[
\|Df_k(f_k^n(x_k,v_k))\|_{C^l}\leq (l+k)!\prod_{i=1}^{k+1}3^i\|f\|_{C^{i+l}}^{n}.
\]
This is because 
\begin{eqnarray*}
\|Df_k^n(x_k,v_k)\|_{C^l}&\leq&2\|Df_{k-1}^n(x_{k-1},v_{k-1})\|_{C^{l}}+\|D^2f_{k-1}^n(x_{k-1},v_{k-1})v_k\|_{C^l}\\
&\leq&3\|Df_{k-1}^n(x_{k-1},v_{k-1})\|_{C^{l+1}}\\
&\leq&\cdots\\
&\leq&3^{k-1}\|Df_1^n(x_1,v_1)\|_{C^{l+k-1}}\\
&\leq&3^k\|f\|_{C^{l+k}}^n
\end{eqnarray*}
and
\begin{eqnarray*}
\|Df_k(f_k^n(x_k,v_k))\|_{C^l}&\leq&2\|Df_{k-1}(f_{k-1}^n(x_{k-1},v_{k-1}))\|_{C^{l}}+\|D^2f_{k-1}(f_{k-1}^n(x_{k-1},v_{k-1}))\cdot\\
& & Df_{k-1}^n(x_k)v_k\|_{C^l}\\
&\leq&2\|Df_{k-1}(f_{k-1}^n(x_{k-1},v_{k-1}))\|_{C^{l}}+(l+1)\|D^2f_{k-1}(\\
& & f_{k-1}^n(x_{k-1},v_{k-1}))\|_{C^l}\|Df_{k-1}^n(x_k)\|_{C^l}\\
&\leq&2\|Df_{k-1}(f_{k-1}^n(x_{k-1},v_{k-1}))\|_{C^{l}}+(l+1)\|Df_{k-1}(\\
& & f_{k-1}^n(x_{k-1},v_{k-1}))\|_{C^{l+1}}3^{k-1}\|f\|^n_{C^{l+k-1}}\\
&\leq&3^k(l+1)\|Df_{k-1}(f_{k-1}^n(x_{k-1},v_{k-1}))\|_{C^{l+1}}\|f\|^n_{C^{l+k-1}}\\
&\leq&\cdots\\
&\leq&(\prod_{i=1}^{k}3^{k-i+1}(l+i)\|f\|^n_{C^{l+k-i}})\cdot\|Df(f^n(x_0))\|_{C^{l+k}}\\
&\leq&(l+k)!\prod_{j=1}^{k+1}3^{j}\|f\|^n_{C^{l+j}}.
\end{eqnarray*}

Now the induction becomes:
\begin{eqnarray}\label{jet estimate}
\|Df^{n+1}_k(x_k,v_k)&-&Df^{n+1}_k(y_k,w_k)\|\leq\|Df_k(f^n_k(x_k,v_k))\|\|Df^n_k(x_k,v_k)-Df^n_k(y_k,w_k)\| \nonumber\\
 & &+ \|Df_k(f^n_k(x_k,v_k))-Df_k(f^n_k(y_k,w_k))\|\|Df_k^n(y_k,w_k)\| \nonumber\\
 &\leq& \max_{\theta\in[0,1]}\|D^2f_k^n(\theta x_k+(1-\theta)y_k,\theta v_k+(1-\theta)w_k)\|\cdot\nonumber\\
 & & \|Df_k(f_{k-1}^n(x_k),Df_{k-1}^n(x_k)v_k)\|(\|x_k-y_k\|+\|v_k-w_k\|) \nonumber\\
 & &+ \|Df_k(f_{k-1}^n(x_k),Df_{k-1}^n(x_k)v_k)-Df_k(f_{k-1}^n(y_k),Df_{k-1}^n(y_k)w_k)\|\cdot \nonumber\\
 & & \|Df_k^n(y_k,w_k)\|\\
 &\leq& (2\|Df_{k-1}(f_{k-1}^n(x_{k-1},v_{k-1}))\|+\|D^2f_{k-1}(f_{k-1}^n(x_{k-1},v_{k-1}))\|\cdot\nonumber\\
& & \|Df_{k-1}^n(x_{k-1},v_{k-1})\|)\cdot2\max_{\theta\in[0,1]}\|f_{k-1}^n(\theta x_k+(1-\theta)y_k)\|_{C^3}\cdot\nonumber\\
& & (\|x_k-y_k\|+\|v_k-w_k\|) +\Big{\{}2\|Df_{k-1}(f_{k-1}^n(x_k))-\nonumber\\
& & Df_{k-1}(f_{k-1}^n(y_k))\| +\|v_k\|\|Df_{k-1}^n(x_k)\|\cdot\nonumber\\
& &\|D^2f_{k-1}(f_{k-1}^n(x_k))-D^2f_{k-1}(f_{k-1}^n(y_k))\|+ \nonumber\\
 & & \|D^2f_{k-1}(f_{k-1}^n(y_k))\|\|Df_{k-1}^n(x_k)\|\|v_k-w_k\|+\nonumber\\
 & & \|D^2f_{k-1}(f_{k-1}^n(y_k))\|\|w_k\|\|Df_{k-1}^n(x_k)-Df_{k-1}^n(y_k)\|\Big{\}}\|Df_k^n(y_k,w_k)\|,\nonumber
\end{eqnarray}
which leads to 
\begin{eqnarray}\label{crucial}
2^{k+1}\|f\|_{C^{k+2}}^n3^kk!\prod_{i=1}^{k+1}3^i\|f\|_{C^{i}}^n\Big{|}_{\Lambda}&+&3^k\|f\|_{C^{k+1}}^n\Big{\{}2b_{k-1}3^{k-1}\|f\|_{C^k}^n
+(3^{k-1}\|f\|_{C^k}^n)^2(k+1)!\cdot\nonumber\\
& &3^{k-1}\prod_{i=1}^{k+2}3^i\|f\|_{C^i}^n+3^kk!\Big{(}\prod_{i=1}^{k}3^i\|f\|_{C^i}^n\Big{)}3^{k-1}\|f\|_{C^k}^n+\nonumber\\
& &3^kk!\Big{(}\prod_{i=1}^{k}3^i\|f\|_{C^i}^n\Big{)}3^{k-1}\|f\|_{C^{k+2}}^n\Big{\}}\Big{|}_{\Lambda}\leq b_{k}^{n+1}
\end{eqnarray}
whereas $\|v_l\|\leq1$, $1\leq l\leq k$. Then (\ref{crucial}) holds if we take 
\[
b_k=\max_{x\in\Lambda}3^{4k}(k+1)!2^{(k+2)(k+3)/2}\|f\|_{C^{k+3}}^{k+3}.
\]
\end{proof}
\begin{rmk}
In Lemma \ref{induction} we just give a very loose $b_k$ estimate. Besides, during the induction we might omit the $\mathcal{O}(1)$ constant, which can be absorbed into $C_i$ of Lemma \ref{operator} when proving Theorem \ref{holder}. So that won't influence our result.
\end{rmk}

\section{Normally Hyperbolic Invariant 
Laminations and skew-products}
\label{sec:nhil-skew}

Recall a definition of a normally hyperbolic invariant laminations
(see \cite{DL} for use of normally hyperbolic laminations to 
construct diffusing orbits).
 
\begin{definition} \label{def:nhl}
Let $\Sigma=\{0,1\}^\Z$ be the shift space, $\sigma$ be the shift on it.
Let $F$ be a $C^1$ map on a manifold $M$. Let $N$ be a manifold.
Let $h : \Sigma \times N \to M$ and $r : \Sigma \times N \to N$ be 
such that we have
the following properties:

a) For every $\sigma \in \Sigma,\ h_\sigma \in C^1(N,M)$ is an embedding,
$r \in C^0(N,N)$ is a homeomorphism. Denote 
$h_\sigma(x):= h(\sigma,x),\ r_\sigma (x):= r(\sigma, x) $. 

b) The maps from $\sigma \in \Sigma$ to $h_\sigma,\ r_\sigma$ are
$C^\al,\ \al > 0$ with the $\sigma$ given the natural topology 
and the maps $h,\ r$ given in the topology of embeddings.

We say that $h,\ r$ is a normally hyperbolic embedding of the shift 
$\Sigma$ if for every  $x \in N$ we can find a splitting
\[
 T_{h^\sigma}(x) = E^s_{h_\sigma}(x) \oplus E^s_{h^\sigma}(x) \oplus E^c_{h_\sigma}(x)
\]
 and numbers $0<C,\ 0<\lb<\mu<1$ such that
\[
v \in E^s_{h_\sigma}(x) \Longleftrightarrow |DF^n(\sg,x) v|\le C \lb^n |v|\quad 
n\ge 0
\]
\[
v \in E^u_{h_\sigma}(x) \Longleftrightarrow |DF^n(\sg,x)v|\le C \lb^{|n|} |v|\quad 
n\le 0\]
\[
v \in E^c_{h_\sigma}(x) \Longleftrightarrow |DF^n(\sg,x)v|\le C \mu^{-|n|} |v|\quad 
n\ge \mathbb Z.
\]
\end{definition}

If we replace $\Sigma$ by one point we get the definition 
of a normally hyperbolic invariant manifold. 
In our case we have 
$N=\mathbb A=\mathbb T \times \mathbb R$ is 
the cylinder, $M=\mathbb A\times U$ is the product 
of a cylinder and an open set $U$ in $\R^2$.

Now we turn to skew products.

Fix an integer $N>1$ and a matrix $A=(a_{ij})^N_{i,j=1}$, 
where $a_{ij} \in \{0, 1\}$. Denote by $\Sigma_A$ the set of 
all bilateral sequences $\om = (\om_n )_{n\in \Z}$ composed of 
symbols $1, \dots , N$ such that $a_{\om_n \om_{n+1}} = 1$ 
for any $n \in \Z$  (see, for instance, \cite{BrSt}).
Call $A$ a transition matrix. 
Suppose $\sigma : \Sigma_A \to \Sigma_A$  is a transitive 
subshift of finite type (a topological Markov chain) with
a finite set of states $\{1,\dots,N\}$ and 
the transition matrix $A$.

The map $\sigma$ shifts any sequence $\om$ one step to the left: 
$(\sg \om)_n = \om_{n+1}$ for any $n\in\Z$. 
By definition (cf. \cite{BrSt}), subshift 
is transitive iff there exists $n\in \Z_+$ such that 
\[
\text{for any }\  i,j\qquad (A^n)_{ij} >0.
\]
Transitivity implies the indecomposability of the subshift. Indeed, 
for any $m > 0$ the subshift $\sigma^m$ with the same states 
allows one to go from any state to any other in finitely many steps. 
Thus, for any $m > 0$ the subshift $\sigma^m$ cannot be split into 
two nontrivial subshifts of finite type.

As usual we endow $\Sigma$ with a metric defined by the formula
$$
d(\om^1,\om^2)=\begin{cases}
2^{-\min\{|n|:\ \om^1_n\neq \om^2_n\}} & \om^1\neq \om^2 \\
0 & \om^1=\om^2 
\end{cases} \qquad 
\om^1,\om^2\in \Sigma.
$$
Let $M$ be a smooth manifold with boundary. Denote by Diff$^{\,r}(M)$ 
the space of $C^r$-smooth maps from $M$ to itself which are 
diffeomorphisms to their images.

\begin{definition}  A skew product over a subshift of finite type 
$(\Sigma_A,\sigma)$ is a dynamical system 
$F : \Sigma_A \times M \to \Sigma_A \times M$ of the form
\[
(\om, x) \ \longrightarrow \  (\sigma \om, f_\om(x)),
\]
where $\om\in\Sigma_A, x\in M$ and the map 
$f_\om (x) \in$  \text{ Diff}$\,^r(M)$ and is continuous in $\om$. 
The phase space of the subshift is called the base of 
the skew product, the manifold $M$ is called the fiber, and 
the maps $f_\om$ are called the fiber maps. The fiber over 
$\om$ is the set 
$M_\om := \{\om \}\times M \subset \Sigma_A \times M$.
\end{definition}

In any argument about the geometry of the skew products 
we always assume that the base factor of $\Sigma_A \times M$ is ``horizontal'' and the fiber factor is ``vertical''. A skew product over 
a subshift of finite type is a step skew product if the fiber maps 
$f_\om$ depend only on the position $\om_0$ in the sequence $\om$. 
In our case $M$ is either the circle or the annulus $\A$. 

Let $\Sigma_A$ be the space of unilateral (infinite to the right) sequences 
$\om = (\om_n )^{+\infty}_0$ satisfying $a_{\om_n\om_{n+1}} = 1$ 
for all n. The left shift 
$$\sigma_+ : \Sigma_A^+ \to \Sigma_A^+,\ 
(\sigma_+\om)_n = \om_{n+1}
$$ 
defines a non-invertible 
dynamical system on $\Sigma_A^+$ . The system $(\Sigma_A^+,\sigma_+)$ 
is a factor of the system $(\Sigma_A, \sigma)$ under the ``forgetting the past'' map 
$$
\pi : (\om)_{-\infty}^{+\infty}\to  (\om)_{-\infty}^{+\infty}
\ \text{ so that }\ \pi \sigma \equiv \sigma_+\pi.  
$$
Similarly, one can define $(\Sigma^-_A,\sigma_-)$
the right shift 
$$
\sigma_- : \Sigma_A^- \to \Sigma_A^-,\ 
(\sigma_-\om)_n = \om_{n-1}.
$$

Let $\Pi = (\pi_{ij})^N_{i,j=1},\ \pi_{ij} \in [0,1]$ 
be a right stochastic matrix (i.e., for any $i$ we have $\sum_j \pi_{ij} = 1$) 
such  that $\pi_{ij} = 0$ iff $a_{ij} = 0$.  Let $p$ be its eigenvector 
with non-negative components that corresponds to the eigenvalue $1$:
for any i $p_i \ge 0$, and $\sum_i \pi_{ij}p_i = p_j$. 
We can always assume $\sum_i p_i = 1$. Using the distribution 
$p_i$ one defined a {\it Markov measure} $\nu$. 
Let $\nu$ be any ergodic Markov measure on $\Sigma$. From now 
on, the measure $\nu$ is fixed.

The {\it standard measure} $\bf s$ on $\Sigma_A \times M$ is 
the product of $\mu$ and the Lebesgue measure on the fiber
(which is either the cirle or the cylinder).

A skew product over a subshift of finite type is 
a {\it step skew product } if the fiber maps $f_\om$ depend 
only on the position $\om_0$ in the sequence $\om$. 
In this paper, we study skew-products close 
to step skew products.  

Suppose $\Theta$ is the set of all such skew products and 
$S \subset \Theta$ is the subset of all step skew products. 
Note that $S$ is the Cartesian product of $N$ copies of Diff$^r(M)$. 
We endow $\Theta$ with the metric
\[
\text{dist}_{\Theta} (F,G) := 
\sup_\om \ \text{dist}_{C^r}(f^{\pm 1}_\om,g^{\pm 1}_\om).
\]
This induces a metric of a product on $S$.  
We consider two cases:  $M$ is either the circle $\T$ 
or the cylinder $\A=\R \times \T$ and and each fiber  
$M_k := \{k\} \times \T$ is either the circle or the cylinder.

\section{A theorem from \cite{CK} on weak convergence to 
a diffusion process}\label{sec:oriol-result}

Let $\eps>0$ be a small parameter and $l\ge 12, s\ge 0$
be an integer. Denote by $\mathcal O_l(\eps)$  
a $\mathcal C^l$ function whose $\mathcal C^l$ norm is 
bounded by $C\eps$ with $C$ independent of $\eps$. 
Similar definition applies for a power of $\eps$. As before 
$\Sigma$ denotes $\{0,1\}^\Z$ and 
$\om=(\dots,\om_0,\dots)\in \Sigma$. 

Consider two nearly integrable maps:
\begin{eqnarray} \label{mapthetar}
f_\om:\mathbb{T}
\times \mathbb{R}
& \longrightarrow & 
\mathbb{T} \times \mathbb{R} \qquad \qquad 
\qquad \qquad \qquad \qquad
\nonumber\\
f_\om:
\left(\begin{array}{c}\theta\\r\end{array}\right) & 
\longmapsto &  
\left(\begin{array}{c}\theta+r+\eps u_{\om_0}(\theta,r)+
\mathcal O_s(\eps^{1+a},\om)
\\
r+\eps v_{\om_0}(\theta,r)+
\eps^2 w_{\om_0}(\theta,r)+\mathcal O_s(\eps^{2+a},\om)
\end{array}\right).
\end{eqnarray} 
for $\om_0\in \{-1,1\}$, where $u_{\om_0},\ v_{\om_0},$
and $w_{\om_0}$ are bounded $\mathcal{C}^l$ functions, 
$1$-periodic in $\theta$, $\mathcal O_s(\eps^{1+a},\om)$ 
and $\mathcal O_s(\eps^{2+a},\om)$ denote  remainders 
depending on $\om$ and uniformly $C^s$ bounded in $\om$, 
and $0<a\le 1/6$. Assume 
$$
\max |v_i(\theta,r)|\le 1,
$$
where maximum is taken over $i=-1,1$ and all 
$(\theta,r)\in \A$, otherwise, renormalize $\eps$. 

We study random iterations of the maps $f_1$ and $f_{-1}$, 
such that at each step the probability of performing either 
map is $1/2$. Importance of understanding iterations of 
several maps for problems of diffusion is well known
(see e.g. \cite{K,Mo}). 

Denote the expected potential and 
the difference of potentials by 
$$\Eu(\theta,r):=
\frac 12 (u_1(\theta,r)+u_{-1}(\theta,r)),\ \ \  \Ev(\theta,r):=\frac 12 (v_1(\theta,r)+v_{-1}(\theta,r)),$$
$$u(\theta,r):=\frac 12 (u_1(\theta,r)-u_{-1}(\theta,r)),\ \ \ 
v(\theta,r):=\frac 12 (v_1(\theta,r)-v_{-1}(\theta,r)).$$

Suppose the following assumptions hold:
\begin{itemize}
\item[{\bf [H0]}] ({\it zero average})
Let for each $r\in \R$ and $i=\pm 1$ we have  
$\int v_i(\theta,r)\,d\theta=0$.

\item[{\bf [H1]}] ({\it no common zeroes}) 
For each integer $n\in \Z$ potentials $v_{1}(\theta,n)$ and 
$v_{-1}(\theta,n)$ have no common zeroes and, equivalently, 
$f_1$ and $f_{-1}$ have no fixed points;

\item[{\bf [H2]}] 
for each $r\in \R$ we have $\int_0^1\ v^2(\theta,r)d\theta=:\sigma(r) \neq0$;

\item[{\bf [H3]}] The functions $v_i(\theta,r)$ are trigonometric polynomials 
in $\theta$, i.e. for some positive integer $d$ we have 
$$
v_i(\theta,r)=\sum_{k\in \Z,\ |k|\le d} v^{(k)}(r)\exp 2\pi ik\theta. 
$$
\end{itemize}

For $\omega\in\{-1,1\}^\Z$ we 
can rewrite the maps $f_{\om}$ in the following form:
\begin{equation*}
f_{\omega}
\left(\begin{array}{c}\theta\\r\end{array}\right)\longmapsto
\left(\begin{array}{c}\theta+r+\eps \Eu(\theta,r)+ 
\eps\omega_0 u(\theta,r)+\mathcal O_s(\eps^{1+a},\om)
\\
r+\eps \Ev(\theta,r)
+\eps\omega_0 v(\theta,r)+\eps^2 w_{\om_0}(\theta,r)
+\mathcal O_s(\eps^{2+a},\om)
\end{array}\right).
\end{equation*}

Let $n$ be positive integer and $\omega_k\in\{-1,1\}$, $k=0,\dots,n-1$, 
be independent random variables with $\mathbb{P}\{\omega_k=\pm1\}=1/2$ 
and $\Omega_n=\{\omega_0,\dots,\omega_{n-1}\}$. 
Given an initial condition $(\theta_0,r_0)$ we denote:
\be \label{eq:random-seq} 
(\theta_n,r_n):=f^n_{\Omega_n}(\theta_0,r_0)=
f_{\omega_{n-1}}\circ f_{\omega_{n-2}}\circ \cdots
\circ f_{\omega_0}(\theta_0,r_0).
\ee

\begin{itemize}
\item[{\bf [H4]}]  ({\it no common periodic orbits}) 
Suppose for any rational $r=p/q\in\mathbb Q$ 
with $p,q$ relatively prime, $1\le |q|\le 2d$ and any $\theta\in \T$   
$$
\sum_{k=1}^q\left[v_{-1}(\theta+\frac kq,r)- v_1(\theta+\frac kq,r)\right]^2\ne 0.
$$
This prohibits $f_1$ and $f_{-1}$ to have common periodic 
orbits of period $|q|$. 

\item[{\bf [H5]}] ({\it no degenerate periodic points}) 
Suppose for any rational $r=p/q\in\mathbb Q$ 
with $p,q$ relatively prime, $1\le |q|\le d$, the function:
$$\Ev_{p,q}(\theta,r)=\sum_{\substack{k\in \mathbb{Z}\\0<|kq|<d}}\Ev^{kq}(r)e^{2\pi ikq\theta}$$
has distinct non-degenerate zeroes, where $\Ev^{j}(r)$ 
denotes the $j$--th Fourier coefficient of $\Ev(\theta,r)$.
\end{itemize}

A straightforward calculation shows that:
\begin{equation}\label{mapthetanrn}
\begin{array}{rcl}
\theta_n&=&\displaystyle\theta_0+nr_0+\eps \sum_{k=0}^{n-1}
\left(\Eu(\theta_k,r_k)+ \Ev(\theta_k,r_k)\right)
\\
&&\displaystyle+
\eps\sum_{k=0}^{n-1} \omega_k
\left(u(\theta_k,r_k)+v(\theta_k,r_k)\right) 
+\mathcal O_l(n\eps^{1+a})
\bigskip\\
r_n&=&\displaystyle r_0+\eps\sum_{k=0}^{n-1}
\Ev(\theta_k,r_k)+\eps
\sum_{k=0}^{n-1}\omega_kv(\theta_k,r_k)
+\mathcal O_l(n\eps^{2+a})
\end{array}
\end{equation}
Even though these maps might not be area-preserving, 
using normal forms 
we will simplify these maps significantly on a large domain of the cylinder.

\begin{thm}\label{maintheorem}
Assume that in the notations above conditions {\bf [H0-H5]} 
hold. Let $n_\eps \eps^2 \to s>0$ as $\eps\to 0$ for some 
$s>0$. Then as $\eps \to 0$ the distribution of $r_{n_\eps}-r_0$ 
converges weakly to $R_s$, where $R_\bullet$ is 
a diffusion process of the form \eqref{eq:diffusion}, 
with the drift and the variance 
$$
b(R)=\int_0^1E_2(\theta,R)\,d\theta,
\qquad \sigma^2(R)=\int_0^1v^2(\theta,R)\,d\theta. 
$$
for some function $E_2$, defined in (\ref{eq:drift}).
\end{thm}

Define 
\be \label{eq:drift}
E_2(\theta,r)=\Ev(\theta,  r)\,
\partial_\theta S_1(\theta, r)+\E w(\theta,r),\quad 
b(r)=\int E_2(\theta,r)d\theta,
\ee
where $S_1$ solves an equation right below and 
is a certain generating function 
defined in (\ref{eq:Genfunction}--\ref{eq:HomEq}).
$$
\partial_{\theta}S_1(\tilde\theta,\tilde r)+
\Ev(\tilde\theta,\tilde r)-
\partial_{\theta}S_1(\tilde\theta+\tilde r,\tilde r)=0.
$$
One can easily find a solution of this equation by solving the corresponding equation for the Fourier coefficients. To that aim, we write $S_1$ and $\Ev$ 
in their Fourier series:
\begin{equation}
\label{eq:Genfunction}
S_1(\theta,\tilde r)=
\sum_{k\in\mathbb{Z}}S_1^k(\tilde r)e^{2\pi ik\theta},
\end{equation}
$$\Ev(\theta,r)=\sum_{\substack{k\in\mathbb{Z}\\0<|k|\leq d}}\Ev^k(r)e^{2\pi ik\theta}.$$

It is obvious that for $k>d$ and $k=0$ we can take $S_1^k(\tilde r)=0$. For $0<k\leq d$ we obtain the following homological equation for $S_1^k(\tilde r)$:
\begin{equation} \label{eq:HomEq}
2\pi ikS_1^k(\tilde r)\left(1-e^{2\pi ik\tilde r}\right)+\Ev^k(r)=0.
\end{equation}
Clearly, this equation cannot be solved if $e^{2\pi ik\tilde r}=1$, 
i.e. if $k\tilde r\in\mathbb{Z}$. We note that there exists a constant 
$L$, independent of $\eps$, $L<d^{-1}$, such that for all 
$0<k\leq d$, if $\tilde r\neq p/q$ satisfies:
$$
0<|\tilde r-p/q|\leq L,
$$
then $k\tilde r\not \in\mathbb{Z}$. Thus, restricting ourselves to 
the domain $|\tilde r-p/q|\leq L$, we have that if 
$kp/q\not\in\mathbb{Z}$ equation \eqref{eq:HomEq} always 
has a solution, and if $kp/q\in\mathbb{Z}$ this equation has 
a solution except at $\tilde r=p/q$. Moreover, in the case that 
the solution exists, it is equal to:
$$
S_1^k(\tilde r)=\frac{i\Ev^k(r)}{2\pi k\left(1-e^{2\pi ik\tilde r}\right)}.
$$

\section{Nearly integrable exact area-preserving maps}
\label{sec:eapt}
Let $\A=\T \times \R$ be the annulus, $(\theta,r)\in \A$. 
Consider a $C^r$ smooth exact area-preserving twist map 
$$
f:\A \to \A, \qquad f(\theta,r)=(\theta',r'),
$$
namely, 
\begin{itemize}

\item $f$ is exact if the area under any noncontractible 
curve $\gamma$ equals the area under $f(\gamma)$ or, 
equivalently, the flux is zero; 

\item $f$ is area-preserving; 

\item $f$ twists, i.e. for  any $\theta^*$ the image 
of $l_{\theta^*}=\{\theta=\theta^*,\ r\in \R\}$ is 
monotonically twisted, i.e. $f(\theta^*,r)$
has the first component strictly monotone in $r$. 
\end{itemize}

Let 
$F:\R^2 \to \R^2,\ F(x,r)=(x',r'), \quad 
F(x+1,r)=(x'+1,r')$ for all $(x,r)\in \R^2.$ 
be the lift of $f:\A \to \A$. Recall that $h:\R^2\to \R$ is 
called {\it a generating function} of $f$ if we have 
\be \label{eq:generating}
\beal 
\partial_1 h(x,x')=-y\\
\partial_2 h(x,x')=\,y'. 
\enal 
\ee

\begin{thm} (see e.g. \cite{Le})
Any $C^r$ smooth exact area-preserving twist map $f$ 
possesses a generating function $h$ such that the map 
$f$ is given by (\ref{eq:generating}) implicitly and $h$ satisfies 
\begin{itemize}
\item (periodicity) $h(x+1,x'+1) = h(x, x')$;
 
\item $\partial_{12}\, h(x,x') < 0$ for all $(x,x')\in \R^2$.
\end{itemize}
\end{thm} 

Notice that in the case $f_0$ is a $C^r$ smooth integrable 
twist map, given by 
$$
f_0:(x,y)\mapsto (x+\rho(y),y)
$$ 
for some $C^r$ smooth strictly monotone function $\rho(y)$ 
the generating function has the form 
$$
h(x,x')=U(x'-x)
$$
for some $C^{r+1}$ smooth function $U$. Indeed, 
$\partial_1 h(x,x')=-U'(x-x')=-y=-y'$. Thus, 
$U'(\rho(y))\equiv y$. 

\begin{lem} Let $f_\eps:\A\to \A$ be a $C^r$ smooth 
nearly integrable exact area-preserving twist map 
we have 
\be \label{eq:map-expansion}
\beal 
\theta' &=& \theta + \nu(r)  + \eps u_1(\theta,r)
+ \eps^2 u_2(\theta,r) + 
\mathcal O(\eps^3)\ (\text{mod }1 )\ \\ 
r' &=& r+ \eps v(\theta,r)+\eps^2 w(\theta,r)+
\mathcal O(\eps^3)\qquad \quad \qquad \qquad 
\enal 
\ee
for some $C^{r-1}$ smooth functions $u_1,v$ and 
$C^{r-2}$ smooth functions $u_2,w$. In the case $f_\eps$ 
is given by a generating function $h(x,x',\eps)$ 
and $h(x,x')=h(x,x',0)$ we have  
\be \nonumber 
\beal 
x' &=& x+\rho(r) + 
\eps \rho'(r)\,\partial_1 h(x,x+\rho(r)) +
\eps^2 (\rho'(r))^2\, \partial_{12} h(x,x+\rho(r)) 
+\mathcal O(\eps^3)  \qquad  \\ 
r' &=& r+ \eps (\partial_1 h(x,x+\rho(r))- \partial_2 h(x,x+\rho(r)))
\qquad \qquad \qquad \qquad \qquad \quad 
\\
&+&\eps^2 (\rho'(r))^2\,\left(\partial_{12} h(x,x+\rho(r))- 
\partial_{22} h(x,x+\rho(r))\right)
\partial_1 h(x,x+\rho(r))
+\mathcal O(\eps^3).
\enal 
\ee
In the case $\rho(r)$ depends on $\eps$ analogs of 
the above 
formulas are still valid:
{\small
\be \nonumber 
\beal 
x' &=& x+\rho_\eps(r) + 
\eps \rho'_\eps(r)\,\partial_1 h(x,x+\rho_\eps(r)) 
\qquad \qquad \qquad 
\qquad \qquad \qquad \\ 
& + & \eps^2 (\rho'_\eps(r))^2\, 
\left( \partial_{12} h(x,x+\rho_\eps(r)) + \frac 12 
(\partial_{1} h(x,x+\rho_\eps(r)) )^2 \right)
+\mathcal O\left(\frac{\eps\|h_\eps\|_{C^3}}{\|\rho_\eps\|_{C^3}}\right)^3  
\qquad  \\ 
r' &=& r+ \eps (\partial_1 h(x,x+\rho_\eps(r))+
\partial_2 h(x,x+\rho_\eps(r)))
\qquad \qquad \qquad \qquad \qquad \quad 
\\
&+&\eps^2 \rho'_\eps(r)\,\left(\partial_{12} h(x,x+\rho_\eps(r))+
\partial_{22} h(x,x+\rho_\eps(r))\right) \partial_1 h(x,x+\rho_\eps(r))
+\mathcal O\left(\frac{\eps\|h_\eps\|_{C^3}}{\|\rho_\eps\|_{C^3}} \right)^3.
\enal 
\ee
}
\end{lem}

\begin{proof} Let $F_\eps:\R^2\to \R^2$ be 
the lift of $f_\eps$. Since $f_\eps$ is 
a $C^r$ smooth nearly integrable, $F_\eps$ has 
a generating function $h_\eps(x,x')$ of the form 
\[
h_\eps(x,x')=U(x'-x)+\eps h_1(x,x',\eps).
\]
Apply the equations of the generating function 
\be 
\beal 
r &=& - \partial_1 h_\eps (x,x',\eps) &= &  U'(x'-x) 
- \eps \partial_1 h_1(x,x',\eps) \\ 
r' &=& \partial_2 h_\eps (x,x',\eps) &= & U'(x'-x) + 
\eps \partial_2 h_1(x,x',\eps).
\enal 
\ee
Notice that $\rho(U'(\Delta x))=\Delta x$.  
To simplify notations denote $h(x,x')=h(x,x',0)$. 
Rearranging and expanding in $\eps$ we have 
\be \nonumber 
\beal 
x' &=& x+\rho(r) + 
\eps \rho'(r)\,\partial_1 h(x,x+\rho(r)) +
\eps^2 (\rho'(r))^2 \partial_{12} h(x,x+\rho(r)) +\mathcal O(\eps^3)   \\ 
r' &=& r+ \eps (\partial_1 h(x,x+\rho(r))+ \partial_2 h(x,x+\rho(r)))
\qquad \qquad \qquad \qquad \qquad \quad 
\\
&+&\eps^2\,\rho'(r)\,\left(\partial_{12} h(x,x+\rho(r))+ 
\partial_{22} h(x,x+\rho(r))\right)
\partial_1 h(x,x+\rho(r))
+\mathcal O(\eps^3).
\enal 
\ee
This gives a definition of functions $u,v$ and $w$
in terms of the generating function $h$. 
\end{proof}

\begin{cor} \label{cor:map-sing-expansion}
Let 
$h(x,x',\eps)=a(\eps) U(x'-x,\eps)+\eps h_1(x,x',\eps),$
$h_1(x,x')=h_1(x,x',0)$, $\rho(U'(x,\eps),\eps)\equiv x,$
and $x^+=x+a(\eps)\rho(r,\eps)$. 
Then 
\be \nonumber 
\beal 
x' &=& x+a(\eps) \rho (r,\eps) +
\eps a(\eps) \rho'(r,\eps)\,\partial_1 h(x,x^+) &
\qquad \qquad \qquad 
\qquad \qquad \qquad 
\\ 
&& +  \eps^2 a^2(\eps)(\rho'(r,\eps))^2\,  
( \partial_{12} h(x,x^+) +  \frac 12 
(\partial_{1} h(x,x^+) )^2 ) & \\
&& +  
\mathcal O\left(\eps\,a(\eps)\,
(\|U\|_{C^3}+\|h_1\|_{C^3}) \|\rho\|_{C^3}\right)^3  &
\\ 
r' &=& r+ \eps (\partial_1 h_1(x,x^+) +
\partial_2 h_1(x,x^+)) 
\qquad \qquad \qquad & \qquad \qquad \quad 
\\
&&+\eps^2  \rho'(r,\eps)\,
(\partial_{12} h_1(x,x^+)+ 
\partial_{22} h_1(x,x^+)) 
\partial_1 h(x,x^+) &\\ 
& &+ 
\mathcal O\left(\eps\,a(\eps)\,(\|U\|_{C^3}+\|h_1\|_{C^3})\ \|\rho\|_{C^3}\right)^3.&
\enal 
\ee
In the case $a(\eps)=\log \eps$ and $U,h_1,\rho \in C^3$ 
the remainder term is $\mathcal O(\eps \log \eps)^3$. 
\end{cor}

The proof is the straightforward substitution. 

\def\cprime{$'$}

\end{document}

%% file: isoblock-pendulum.pdf_tex
\begingroup%
  \makeatletter%
  \providecommand\color[2][]{%
    \errmessage{(Inkscape) Color is used for the text in Inkscape, but the package 'color.sty' is not loaded}%
    \renewcommand\color[2][]{}%
  }%
  \providecommand\transparent[1]{%
    \errmessage{(Inkscape) Transparency is used (non-zero) for the text in Inkscape, but the package 'transparent.sty' is not loaded}%
    \renewcommand\transparent[1]{}%
  }%
  \providecommand\rotatebox[2]{#2}%
  \ifx\svgwidth\undefined%
    \setlength{\unitlength}{533.91309509bp}%
    \ifx\svgscale\undefined%
      \relax%
    \else%
      \setlength{\unitlength}{\unitlength * \real{\svgscale}}%
    \fi%
  \else%
    \setlength{\unitlength}{\svgwidth}%
  \fi%
  \global\let\svgwidth\undefined%
  \global\let\svgscale\undefined%
  \makeatother%
  \begin{picture}(1,0.51575024)%
    \put(0,0){\includegraphics[width=\unitlength,page=1]{isoblock-pendulum.pdf}}%
    \put(0.17810517,0.00388055){\color[rgb]{0,0,0}\makebox(0,0)[lb]{\smash{$\tau_0$}}}%
    \put(0.70852858,0.00388055){\color[rgb]{0,0,0}\makebox(0,0)[lb]{\smash{$\tau_1$}}}%
    \put(0.36689994,0.03085123){\color[rgb]{0,0,0}\makebox(0,0)[lb]{\smash{$R^u_{01}$}}}%
    \put(0.36689994,0.19267534){\color[rgb]{0,0,0}\makebox(0,0)[lb]{\smash{$R^u_{00}$}}}%
    \put(0.89963529,0.32562464){\color[rgb]{0,0,0}\makebox(0,0)[lb]{\smash{$R^u_{10}$}}}%
    \put(0.89616738,0.44583368){\color[rgb]{0,0,0}\makebox(0,0)[lb]{\smash{$R^u_{11}$}}}%
    \put(-0.00094819,0.20165018){\color[rgb]{0,0,0}\makebox(0,0)[lb]{\smash{$f(R^u_{00})$}}}%
    \put(0.00636786,0.04868449){\color[rgb]{0,0,0}\makebox(0,0)[lb]{\smash{$f(R^u_{10})$}}}%
    \put(0.83952466,0.07090554){\color[rgb]{0,0,0}\makebox(0,0)[lb]{\smash{$f(R^u_{01})$}}}%
    \put(0.8919329,0.22365767){\color[rgb]{0,0,0}\makebox(0,0)[lb]{\smash{$f(R^u_{11})$}}}%
    \put(0.20507586,0.31853852){\color[rgb]{0,0,0}\makebox(0,0)[lb]{\smash{$x_{00}$}}}%
    \put(0.20507586,0.15671445){\color[rgb]{0,0,0}\makebox(0,0)[lb]{\smash{$x_{01}$}}}%
    \put(0.7489846,0.30804992){\color[rgb]{0,0,0}\makebox(0,0)[lb]{\smash{$x_{11}$}}}%
    \put(0.7489846,0.15371771){\color[rgb]{0,0,0}\makebox(0,0)[lb]{\smash{$x_{10}$}}}%
    \put(0,0){\includegraphics[width=\unitlength,page=2]{isoblock-pendulum.pdf}}%
    \put(0.34377739,0.40129133){\color[rgb]{0,0,0}\makebox(0,0)[lb]{\smash{$v^u_1$}}}%
    \put(0.29947324,0.50283817){\color[rgb]{0,0,0}\makebox(0,0)[lb]{\smash{$v^s_1$}}}%
    \put(0,0){\includegraphics[width=\unitlength,page=3]{isoblock-pendulum.pdf}}%
    \put(0.54637132,0.40014153){\color[rgb]{0,0,0}\makebox(0,0)[lb]{\smash{$v^u_2$}}}%
    \put(0.58841978,0.50284435){\color[rgb]{0,0,0}\makebox(0,0)[lb]{\smash{$v^s_2$}}}%
    \put(0,0){\includegraphics[width=\unitlength,page=4]{isoblock-pendulum.pdf}}%
  \end{picture}%
\endgroup%